\documentclass{mhkr}

\usepackage{lineno}

\newcommand{\added}[1]{#1}

\definecolor{darkgreen}{rgb}{0,0.60,0}

\newaliascnt{phenomenon}{theorem}
\newtheorem{phenomenon}[phenomenon]{Phenomenon}
\aliascntresetthe{phenomenon}

\usepackage[style = numeric, sorting = nyt]{biblatex}
\renewbibmacro*{issue+date}{%
	\ifboolexpr{test {\iffieldundef{year}}
		and test {\iffieldundef{issue}}}
	{}
	{\printtext[parens]{%
			\printfield{issue}%
			\setunit*{\addspace}%
			\usebibmacro{date}}}%
	\newunit}

\NewBibliographyString{toappearin}
\NewBibliographyString{submittedto}
\DefineBibliographyStrings{english}{%
	toappearin  = {to appear in},
	submittedto = {submitted to},
}

\renewbibmacro*{in:}{%
	\ifboolexpr{not test {\iffieldundef{pubstate}}
		and (test {\iffieldequalstr{pubstate}{toappearin}}
		or test{\iffieldequalstr{pubstate}{submittedto}})}
	{\printtext{%
			\printfield{pubstate}\intitlepunct}%
		\clearfield{pubstate}}
	{\printtext{%
			\bibstring{in}\intitlepunct}}}

\usepackage{graphicx}
\usetikzlibrary{cd}

\title{The behavior of higher proof theory I: Case $\Sigma^1_2$}
\author{Hanul Jeon}
\email{ \href{mailto:hj344@cornell.edu}{hj344@cornell.edu}}
\urladdr{\href{https://hanuljeon95.github.io}{https://hanuljeon95.github.io}}
\address{Department of Mathematics, Cornell University, Ithaca, NY 14853} 
\thanks{The author would like to thank James Walsh for feedback on this paper, and Henry Towsner for valuable discussion on my work. The research presented in this paper is supported in part by NSF grant DMS–2153975.}

\newcommand{\lag}{\langle}
\newcommand{\rag}{\rangle}

\newcommand{\bbN}{\mathbb{N}}
\newcommand{\bbP}{\mathbb{P}}

\newcommand{\ttL}{\mathtt{L}}

\newcommand{\rank}{\operatorname{rank}}
\newcommand{\field}{\operatorname{field}}

\newcommand{\arity}{\operatorname{arity}}
\newcommand{\supp}{\operatorname{supp}}
\newcommand{\Tr}{\operatorname{Tr}}
\newcommand{\en}{\operatorname{en}}

\newcommand{\WO}{\mathsf{WO}}
\newcommand{\LO}{\mathsf{LO}}
\newcommand{\Clim}{\mathsf{Clim}}
\newcommand{\Grow}{\mathsf{Grow}}
\newcommand{\Dil}{\mathsf{Dil}}
\newcommand{\CK}{\mathsf{CK}}
\newcommand{\Prv}{\mathsf{Prv}}
\newcommand{\True}{\mathsf{True}}

\newcommand{\Con}{\mathsf{Con}}

\newcommand{\RecPreDil}{\mathsf{RecPreDil}}
\newcommand{\SDil}{\mathsf{SDil}}

\newcommand{\Diag}{\operatorname{Diag}}
\newcommand{\RFN}[1][]{\ifthenelse{\equal{#1}{}}{}{#1\mhyphen}\mathsf{RFN}}

\newcommand{\KB}{\mathsf{KB}}
\newcommand{\AC}{\mathsf{AC}}
\newcommand{\CA}{\mathsf{CA}}
\newcommand{\KP}{\mathsf{KP}}
\newcommand{\RCA}{\mathsf{RCA}}
\newcommand{\ACA}{\mathsf{ACA}}
\newcommand{\ATR}{\mathsf{ATR}}
\newcommand{\BI}{\mathsf{BI}}
\newcommand{\ZFC}{\mathsf{ZFC}}
\newcommand{\CH}{\mathsf{CH}}
\newcommand{\PA}{\mathsf{PA}}

\newcommand{\fraka}{\mathfrak{a}}
\newcommand{\frakf}{\mathfrak{f}}

\def\bfcyrillic{\usefont{T2A}{antt}{eb}{n}}
\def\cyrillic{\usefont{T2A}{antt}{m}{n}}

\newcommand{\cyrDe}{{\text{\cyrillic \char"C4}}}
\newcommand{\cyrZhe}{{\text{\cyrillic \char"C6}}}

\newcommand{\bfcyrDe}{{\text{\bfcyrillic \char"C4}}}

\begin{document}
% \linenumbers
\maketitle

\begin{abstract}
    Walsh \cite{Walsh2023characterizations} has shown that comparing proof-theoretic ordinals is equivalent to comparing $\Pi^1_1$-consequence comparison and  $\Pi^1_1$-reflection comparison, all modulo true $\Sigma^1_1$-sentences.
    In this paper, we prove the analogous result for $\Sigma^1_2$-consequences modulo true $\Pi^1_2$-sentences, that is, the equivalence between $\Sigma^1_2$-proof-theoretic ordinal comparison, $\Sigma^1_2$-consequence comparison, and $\Sigma^1_2$-reflection comparison, all modulo true $\Pi^1_2$-sentences. 
    We also examine the connection between $\Sigma^1_2$-proof-theoretic ordinal and $\Sigma^1_2$-analogue of the robust reflection rank in Pakhomov-Walsh \cite{PakhomovWalsh2021Reflection}.
\end{abstract}

\section{Introduction}

\added{
One of the major questions in mathematics in the late 19th and early 20th centuries was establishing a consistent and complete foundation for mathematics.
Various attempts to address this problem have yielded foundational theories for mathematics, such as Cantor's set theory, which evolved into Zermelo-Frankel set theory, and Russell-Whitehead's type theory, which has resulted in modern type theories. Other logicians also introduced various forms of foundational theories, like subsystems of second-order arithmetic, structural set theory, or Feferman's Explicit Mathematics. That is, we have a plethora of possible foundational theories, and it becomes natural to ask about the relationship between different theories. 

Set theorists and proof theorists have compared two different theories in various ways, like comparing their consistency strengths or their arithmetical consequences. 
To make the latter arguments more precise, let us define the precise meaning of the consistency comparison:
\begin{definition}
    Let $S$ and $T$ be theories extending a weak arithmetic $\mathsf{A}$.\footnote{For technical convenience, we fix $\mathsf{A} = \mathsf{PRA}$ in this paper. However, $\mathsf{A}$ can be far weaker, like $\mathsf{I\Delta_0 + Exp}$ or bounded arithmetic $\mathsf{S}^1_2$, since these theories can arithmetize the syntax and prove incompleteness theorems. See \cite[Chapter 7]{BussThesis} for technical details.}
    We define the following:
    \begin{itemize}
        \item $S \le_\Con T \iff \mathsf{A} \vdash \Con(T)\to \Con(S)$, and
        \item $S <_\Con T$ if $S \le_\Con T$ but $T\nleq_\Con S$.
    \end{itemize}
\end{definition}
}

There is no reason to believe $\le_\Con$ behaves like a well-order.\footnote{More precisely, a prewellorder.} \added{In fact, the following folklore result shows that $\le_\Con$ is far from being a well-order. (See \cite{Hamkins2022Nonlinearity} for the proof.)}
\begin{theorem} \pushQED{\qed}
    \begin{enumerate}
        \item There are \added{recursive} theories $S$, $T$ extending $\PA$ such that neither $S\le_\Con T$ nor $T\le_\Con S$ hold.
        \item For \added{recursive} theories $S$, $T$ extending $\PA$ satisfying $S<_\Con T$, we can find an extension $U$ of $\PA$ such that $S <_\Con U <_\Con T$.
        \item There is \added{a recursive sequence of theories} $\lag T_i\mid i\in\bbN\rag $ extending $\PA$ such that $T_0 >_\Con T_1 >_\Con \cdots$. \qedhere 
    \end{enumerate}
\end{theorem}

However, various logicians claimed that $\le_\Con$ behaves like a well-order among `natural' theories. For example, Simpson stated that
\begin{quote}
    It is striking that a great many foundational theories are linearly ordered by $<_\Con$. Of course, it is possible to construct pairs of artificial theories that are incomparable under $<_\Con$. However, this is not the case for the ``natural'' or nonartificial theories which are usually regarded as significant in the foundations of mathematics. (Simpson \cite[p111]{Simpson2010GodelHierarchy}.)
\end{quote}

Koellner also made a similar claim:
\begin{quote}
    Remarkably, it turns out that when one restricts to those theories that ``arise in nature'' the interpretability ordering is quite simple: There are no descending chains and there are no incomparable elements—the interpretability ordering on theories that ``arise in nature'' is a wellordering. (Koellner \cite{Koellner2011IndependenceLC}.)
\end{quote}

\added{
There is another technical phenomenon about the consistency comparison on natural theories that is not addressed in prior literature: When the logicians prove $S<_\Con T$, they prove $T\vdash \Con(S)$. Furthermore, proving $T\vdash \Con(S)$ is essentially the only way to prove $S<_\Con T$ for natural theories $S$ and $T$.
While $T\vdash \Con(S)$ implies $S<_\Con T$, there is no reason to believe that the converse holds:
\begin{theorem} \label{Theorem: Unnatural examples for the consistency jump}
    There are recursive theories $S$, $T$ extending $\PA$ such that $S <_\Con T$, but $T\nvdash \Con(S)$.\footnote{Hamkins' \emph{petulant enumeration} $\ZFC^\bullet$ in \cite[\S 5]{Hamkins2022Nonlinearity} is an example: Fix a stronger theory $T$ (say, $T = \ZFC + \Con(\ZFC)$). In $\ZFC^\bullet$, we enumerate axioms of $\ZFC$, but once we find a proof of contradiction of $T$, we immediately add the contradiction to the enumeration. One can see that $\ZFC <_\Con \ZFC^\bullet$ since $\ZFC^\bullet$ and $T$ are equiconsistent. However, $\ZFC^\bullet \nvdash \Con(\ZFC)$ since if $T$ is consistent, then $\ZFC^\bullet = \ZFC$.}
\end{theorem}

\begin{question} \label{Question: How to explain the phenomena for natural theories}
    Why $\le_\Con$ is (pre)wellordered for \emph{natural} theories?
    Also, for \emph{natural} theories $S$ and $T$, why is the only way to establish $S<_\Con T$ proving $T\vdash \Con(S)$?
\end{question}
}

Another interesting aspect of consistency strength is that it correlates with the consequences of theories. Again, Simpson described
\begin{quote}
     Assuming that the language of $T_1$ is part of the language of $T_2$, let us write $T_1 \subset T_2$ to mean that the sentences that are theorems of $T_1$ form a proper subset of the sentences in the language of $T_1$ which are theorems of $T_2$. We may think of $T_2 \supset T_1$ as meaning that $T_2$ is ``more powerful'' or ``stronger'' than $T_1$. In many cases, the $<_\Con$ ordering and the $\subset$ ordering coincide. (Simpson \cite[p.111--112]{Simpson2010GodelHierarchy}.)
\end{quote}

Of course, \added{some theories deviate from Simpson's observations:} $\ZFC$ and $\ZFC+\mathsf{CH}$ have the same consistency strength, but $\ZFC \subsetneq \ZFC + \mathsf{CH}$.
Still, $<_\Con$ ordering \added{tends to agree with} \emph{projective} consequences of a theory \added{for natural theories}, which Steel describes:
\begin{phenomenon}[Steel {\cite[p.159]{Steel2014GodelProgram}}] \label{Phenomenon: SteelObservation}
    Let $S \subseteq_{\Pi^1_\infty} T$ mean every projective consequence of $S$ is a theorem of $T$.
    For natural theories $S$, $T$ whose consistency strength is at least that of $\ZFC$ with the Projective Determinacy.\footnote{\added{Steel requires $S$ and $T$ be extensions of $\ZFC$ with infinitely many Woodin cardinals, but the author thinks Projective Determinacy is more appropriate in this context.}} Then either
    \begin{equation*}
        S \subseteq_{\Pi^1_\infty} T \text{ or } T \subseteq_{\Pi^1_\infty} S.
    \end{equation*}
    Also, the projective consequences of sufficiently strong natural theories monotonically increase along with their consistency strength.
\end{phenomenon}

A special case for this phenomenon at the level of $\Sigma^1_2$ is the following:
\begin{phenomenon} \label{Phenomenon: Sigma 1 2 observation}
    Let $S\subseteq_{\Sigma^1_2} T$ mean every $\Sigma^1_2$-consequence of $S$ is a theorem of $T$. For natural \added{set} theories $S$, $T$, we have
    \begin{equation*}
        S \subseteq_{\Sigma^1_2} T \text{ or } T \subseteq_{\Sigma^1_2} S.
    \end{equation*}
\end{phenomenon}

A part of the above phenomenon \added{is explained} by the Shoenfield absoluteness theorem: The main machinery to get the independence results in set theory is forcing and inner model, and the Shoenfield absoluteness theorem says none of these change the validity of $\Sigma^1_2$ sentences. 
As an example, forcing and inner model arguments show that $\ZFC$, $\ZFC+\CH$, and $\ZFC+\lnot\CH$ have the same consistency strength. Hence Shoenfield absoluteness theorem suggests all of $\ZFC$, $\ZFC+\CH$, $\ZFC+\lnot\CH$ have the same $\Sigma^1_2$-consequences. Since forcing and inner models are essentially the only practical tools to get another theory from one theory, we can speculate that theories with the same strength have the same $\Sigma^1_2$-consequences.
However, `stronger' theories can have more $\Sigma^1_2$-consequences than weaker theories: For example, $\ZFC$ with the existence of an inaccessible cardinal proves `there is a transitive model of $\ZFC$', which is equivalent to the $\Sigma^1_2$-statement `there is a countable well-founded model of $\ZFC$' that is not a theorem of $\ZFC$.

\subsection{Ordinal analysis}
\added{The observation that natural theories are prewellordered under the consistency comparison} suggests there should be a way to well-order formal theories coherently with their strengths. Indeed, proof theorists have a way for it called the \emph{proof-theoretic ordinal}, which appears in the \emph{ordinal analysis} \added{of} a theory.
Ordinal analysis was initiated by Gentzen's proof of the consistency of Peano arithmetic $\PA$ modulo the well-foundedness of $\varepsilon_0$. After years, Takeuti \cite{Takeuti1967ConsistencySubsystemAnalysis} provided an ordinal analysis for $\Pi^1_1\mhyphen\CA_0$, which is deemed to be impredicative, so providing a proof-theoretic analysis was challenging. \added{From this point, other proof theorists developed an ordinal analysis for stronger theories, such as iterated inductive definitions and Kripke-Platek set theory with large ordinals.} Recently, there were claims on the ordinal analysis for the full second-order arithmetic, from  Arai \cite{Arai2023PiN} and Towsner \cite{Towsner2024ProofsModifyProofs}.

The current project on ordinal analysis showed fruitful results in calibrating the strength of theories in terms of proof-theoretic ordinal
\begin{equation*}
    |T|_{\Pi^1_1} = \sup \{\alpha\mid \text{$\alpha$ is recursive and } T\vdash\text{``$\alpha$ is well-ordered''}\}.
\end{equation*}
However, the proof-theoretic ordinal of a theory does not precisely correspond to the consistency strength of a given theory: For example, both $\mathsf{PA}$ and $\mathsf{PA + Con(PA)}$ have the same proof-theoretic ordinal $\varepsilon_0$, but the latter theory is stronger than the former. Then, what does the proof-theoretic ordinal gauge? Walsh addressed an answer in \cite{Walsh2023characterizations}, \added{and let us introduce some notions before introducing Walsh's results:}
\begin{definition}
    Let $\Gamma$ be a complexity class (like $\Pi^1_1$ or $\Sigma^1_2$) and $T$ be a theory. We say \emph{$\phi$ is $T$-provable with $\Gamma$-oracle} if there is a true $\Gamma$-sentence $\theta$ such that $T + \theta \vdash \phi$. We denote it by $T \vdash^\Gamma \phi$.

    For two theories $S$ and $T$, we denote $S \subseteq^{\check{\Gamma}}_\Gamma T$ if every $\Gamma$-sentence that is $S$-provable with $\check{\Gamma}$-oracle is also $T$-provable with $\check{\Gamma}$-oracle, where $\check{\Gamma} = \{\lnot\phi\mid \phi\in \Gamma\}$.
\end{definition}

\begin{definition} \label{Definition: Gamma-RFN}
    Let $\Gamma$ be a complexity class and let $T$ be a theory. Then $\RFN[\Gamma](T)$ is the following statement:
    \begin{equation*}
        \RFN[\Gamma](T) \equiv \forall^0 \phi\in \Gamma [\Prv_T(\phi)\to \True_\Gamma(\phi)]
    \end{equation*}
    We say $T$ is \emph{$\Gamma$-sound} if every $T$-provable $\Gamma$-sentence is true. In other words, $T$ is $\Gamma$-sound if $\RFN[\Gamma](T)$ is true.
\end{definition}
We may understand $\RFN[\Gamma](T)$ as a generalization of $\Con(T)$: In fact, we can prove that $\Con(T)$ is equivalent to $\RFN[\Pi^0_1](T)$. 
\added{We are ready to state Walsh's result, which is as follows:}
\begin{theorem}[Walsh \cite{Walsh2023characterizations}] \label{Theorem: Walsh's characterizations of PTO}
    \pushQED{\qed}
    Let $S$, $T$ be $\Pi^1_1$-sound extensions of $\ACA_0$.
    \begin{enumerate}
        \item $|S|_{\Pi^1_1}\le |T|_{\Pi^1_1}$ iff $S\subseteq^{\Sigma^1_1}_{\Pi^1_1} T$.\footnote{\added{\cite{Walsh2023characterizations} states $S$ and $T$ to be $\Sigma^1_1$-definable, but the proof of the equivalence in \cite[\S 3.1]{Walsh2023characterizations} works without any definability constraint on $S$ and $T$.}}
        \item Furthermore, if $S$ and $T$ are arithmetically definable, then $|S|_{\Pi^1_1}\le |T|_{\Pi^1_1}$ iff
        \begin{equation*}
            \ACA_0 \vdash^{\Sigma^1_1} \RFN[\Pi^1_1](T)\to\RFN[\Pi^1_1](S). \qedhere 
        \end{equation*}
    \end{enumerate}
\end{theorem}

That is, Walsh's result claims the following three notions on strength comparison are `equivalent' for arithmetically definable $\Pi^1_1$-sound extensions of $\ACA_0$:
\begin{enumerate}
    \item Comparing the proof-theoretic ordinal $|T|_{\Pi^1_1}$.
    \item Comparing the $\Pi^1_1$-consequences of $T$ (modulo true $\Sigma^1_1$-sentences)
    \item Comparing the $\Pi^1_1$-reflection of $T$ (modulo true $\Sigma^1_1$-sentences), which is a $\Pi^1_1$-analogue of $\Con(T)$.
\end{enumerate}
For an addendum for the $\Pi^1_1$-reflection comparison, modifying Walsh's proof yields the following, which looks more coherent with $<_\Con$:
\begin{theorem} \label{Theorem: PTO strict comparison and Pi 1 1 reflection comparison} 
    Let $S$, $T$ be arithmetically definable $\Pi^1_1$-sound extensions of $\ACA_0$. Then 
    \begin{equation*}
        |S|_{\Pi^1_1} < |T|_{\Pi^1_1} \iff T \vdash^{\Sigma^1_1} \RFN[\Pi^1_1](S).
    \end{equation*}
\end{theorem}
\added{
The author suggests Walsh's result as evidence of the following phenomenon:
\begin{phenomenon} \label{Phenomenon: Pi 1 1 comparison for natural theories}
    For natural theories extending $\ACA_0$, the following three coincide:
    \begin{enumerate}
        \item Comparing the proof-theoretic ordinal of natural theories.
        \item Comparing the $\Pi^1_1$-consequences of natural theories.
        \item $\Pi^1_1$-reflection comparison of natural theories.
    \end{enumerate}
    Moreover, for two natural theories $S$ and $T$ extending $\ACA_0$, $|S|_{\Pi^1_1} < |T|_{\Pi^1_1}$ iff $T \vdash \RFN[\Pi^1_1](S)$.
\end{phenomenon}

We can adopt \autoref{Phenomenon: Pi 1 1 comparison for natural theories} to explain the phenomena in the consistency comparison for natural theories as follows: Let us observe that when logicians make the consistency comparison $S\le_\Con T$, they do more than $S\le_\Con T$.
Here is the list of what logicians do for the consistency comparison:
\begin{enumerate}
    \item For two subsystems of second-order arithmetic $S$ and $T$, the proof of $S<_\Con T$ usually follows from either ``$T$ proves an $\omega$- or $\beta$-model of $S$'' or ``$T$ proves $\RFN[\Gamma](S)$.'' Also, the equiconsistency of $S$ and $T$ usually follows from a stronger result, like the $\Pi^1_1$-conservativity of $T$ over $S$.
    \item When set theorists prove $S <_\Con T$, they typically prove that ``$T$ proves there is a transitive model of $S$.'' Also, when set theorists prove $S\le_\Con T$, they typically rely on forcing and inner models that respect $\Sigma^1_2$-truth.
    \item Ordinal analysis-based consistency comparisons also tend to give stronger results than mere consistency comparisons. For example, Rathjen \cite{Rathjen2014RelativizedOrdinalKP} proved that Power Kripke-Platek set theory $\KP(\mathcal{P})$ and $\mathsf{Z} + \{\text{$V_\xi$ exists}\mid \xi<\mathrm{BH}\}$ are equiconsistent by showing that the two theories prove the same $\Pi_2^\mathcal{P}$-sentences, which also implies that the two theories prove the same $\Pi^1_2$-consequences.
    \item Continuing from the ordinal analysis-based consistency comparison, the independence of Goodstein's theorem over $\PA$ can be shown in a way that Goodstein's theorem implies there is no primitive recursive infinite descending sequence over $\varepsilon_0$.
    Rathjen \cite{Rathjen2015Goodstein} proved that over $\RCA_0$, Goodstein's theorem implies $\varepsilon_0$ is well-founded, which is equivalent to $\RFN[\Pi^1_1](\ACA_0)$. Hence, $\ACA_0 <_\Con \RCA_0 + \text{Goodstein's theorem}$ actually follows from that the latter theory proves $\RFN[\Pi^1_1](\ACA_0)$.
    \item There are `irregular' consistency comparisons for versions of set theories, like \cite{Jeon2023WA}. However, one can check \cite{Jeon2023WA} shows a $\RFN[\Gamma]$ of a weaker theory over a stronger theory for $\Gamma$ containing every projective pointclass.
\end{enumerate}

Note that the following examination only applies to classical theories. For constructive theories, the projective hierarchy may collapse \cite{Veldman2022ProjectiveSets}, so projective reflection comparison or projective consequence comparison may not make sense. Also, the following list works only for theories stronger than $\RCA_0$. For theories at the level of Bounded arithmetic or Elementary arithmetic, the author does not rule out the possibility that the consistency comparison does not happen by stronger comparisons.

The previous list indicates the following phenomenon:
\begin{phenomenon} \label{Phenomenon: Consistency comparison and Pi 1 1 comparison may coincide}
    Let $S$ and $T$ be classical natural theories extending $\RCA_0$.
    When logicians prove $S$ and $T$ are equiconsistent, what they actually prove is (or implies) $S$ and $T$ have the same $\Pi^1_1$-consequences.
    When logicians prove $S <_\Con T$, what they actually prove is (or implies) $T \vdash \RFN[\Pi^1_1](S)$. 
\end{phenomenon}
The reader may think \autoref{Phenomenon: Consistency comparison and Pi 1 1 comparison may coincide} suggests that for natural theories, the consistency comparison and the $\Pi^1_1$-comparison may coincide. In particular, this coincidence implies that the consistency comparison is prewellordered for natural theories since $\Pi^1_1$-reflection comparison prewellorders natural theories by Walsh's result.
However, the author warns that \autoref{Phenomenon: Consistency comparison and Pi 1 1 comparison may coincide} itself \emph{does not} say that the consistency comparison and $\Pi^1_1$-reflection comparison coincide for natural theories. The point of \autoref{Phenomenon: Consistency comparison and Pi 1 1 comparison may coincide} is focused on results in prior and current literature, and it does not rule out the possibility that in a far future, someone proves $S<_\Con T$ for natural $S$ and $T$ with certain criteria by directly proving, for example, $T\vdash \Con(S)$ or $\mathsf{A}\nvdash \Con(S)\to \Con(T)$, and not proving $T\vdash \RFN[\Pi^1_1](S)$.

Even if \autoref{Phenomenon: Consistency comparison and Pi 1 1 comparison may coincide} does not imply the consistency comparison and $\Pi^1_1$-reflection comparison coincide for natural theories, it still explains why logicians have witnessed that the consistency comparison for natural theories is prewellordered:
What logicians have done for the consistency comparison is, in fact, the $\Pi^1_1$-reflection comparison or its strengthening, which is prewellordered for natural theories by \autoref{Phenomenon: Pi 1 1 comparison for natural theories}.
In other words, logicians may have misunderstood the prewellorderedness of the consistency comparison for natural theories since they have witnessed is (or implies) $\Pi^1_1$-reflection comparison, and have not seen a consistency comparison that is not a $\Pi^1_1$-reflection comparison.
}

\subsection{Shifting from \texorpdfstring{$\Pi^1_1$}{Pi 1 1} to \texorpdfstring{$\Sigma^1_2$}{Sigma 1 2}}
\added{A consequence of \autoref{Phenomenon: Pi 1 1 comparison for natural theories} is that the $\Pi^1_1$-consequence comparison for natural theories extending $\ACA_0$ is prewellordered. In particular, for two natural theories $S$, $T$ extending $\ACA_0$, either $S \subseteq_{\Pi^1_1} T$ or $T\subseteq_{\Pi^1_1} S$. This is close to Steel's observation \autoref{Phenomenon: SteelObservation}, but \autoref{Phenomenon: Pi 1 1 comparison for natural theories} is insufficient to imply \autoref{Phenomenon: Sigma 1 2 observation} witnessed for natural set theories. We may hope that by fortifying Walsh's characterization from $\Pi^1_1$ to beyond, we may supply evidence for \autoref{Phenomenon: Sigma 1 2 observation} and $\Sigma^1_2$-analogue of \autoref{Phenomenon: Pi 1 1 comparison for natural theories}.}

Then, how to generalize Walsh's result beyond $\Pi^1_1$? To \added{illustrate} this, let us \added{revisit} the \added{primary} motivation behind Walsh's result. 
The underlying motivation behind Walsh's results about the proof theory, which we will call \emph{$\Pi^1_1$-proof theory}, comes from the Kleene normal form theorem, which says every true $\Pi^1_1$-sentence has a recursive ordinal complexity:
\begin{theorem}[Kleene normal form theorem] \label{Theorem: Kleene normal form}
    \pushQED{\qed}
    \added{For every $\Pi^1_1$-formula $\phi(X)$ we can effectively find a Turing machine $\alpha_X$ for a linear order with an oracle variable $X$ such that the following holds:
    \begin{enumerate}
        \item $\ACA_0$ proves ``For every real $X$, $\alpha_X$ is a linear order.''
        \item $\ACA_0$ proves $\forall^1 X [\phi(X)\lr \WO(\alpha_X)]$. \qedhere
    \end{enumerate}
    }
\end{theorem}

To get an ordinal characteristic of a theory representing more complex consequences beyond $\Pi^1_1$, we need a way to assign an ordinal complexity to a more complex formula. But is it possible? The answer is yes for a $\Sigma^1_2$-formula, the case we will focus on throughout this paper. However, we need a different tool to calibrate the complexity of a $\Sigma^1_2$-formula, \added{which} we will call a \emph{pseudodilator}. For each $\Sigma^1_2$-sentence, we can assign a recursive pseudodilator $F$, a function from ordinals to linear orders. $F$ witnesses the `ordinal complexity' $\alpha$ when $\alpha$ is the least ordinal making $F(\alpha)$ ill-founded, which we will call a \emph{climax} $\Clim(F)$ of $F$. 

Now let us consider the \emph{$\Sigma^1_2$-proof-theoretic ordinal $s^1_2(T)$ of $T$} defined as follows:
\begin{equation*}
    s^1_2(T) = \sup\{\Clim(F) \mid T\vdash \text{$F$ is a pseudodilator}\}.\footnotemark
\end{equation*}
\footnotetext{The $\Sigma^1_2$-proof-theoretic ordinal $s^1_2(T)$ is first defined by Aguilera-Pakhomov \cite[\S 7]{AguileraPakhomov2023Pi12} under the name \emph{$\Sigma^1_2$-soundness ordinal}. $s^1_2(T)$ represents the amount of $\Sigma^1_2$-consequences $T$ can prove, so we call it a $\Sigma^1_2$-proof-theoretic ordinal instead.
Aguilera-Pakhomov did not use the word climax or pseudodilator, and they parsed each word in their definitions. The term \emph{climax} is due to the author, and \emph{pseudodilator} is due to Justin Moore that was suggested during the author's A exam.}
% (See \autoref{Definition: s12 ordinal} for a more precise definition.)
We can ask whether $s^1_2(T)$ enjoys the properties that the $\Pi^1_1$-proof-theoretic ordinal enjoys, and the following results in this paper answer the question:

\begin{theorem*} \phantom{a}
    \begin{enumerate}
        \item (\autoref{Proposition: First equivalence for Sigma 1 2})  For $\Sigma^1_2$-sound extensions $S$, $T$ of $\ACA_0$, we have
        \begin{equation*}
            S\subseteq^{\Pi^1_2}_{\Sigma^1_2} T\iff s^1_2(S)\le s^1_2(T).
        \end{equation*}
        
        \item (\autoref{Proposition: Second equivalence LtoR for Sigma 1 2} and \autoref{Proposition: Second equivalence RtoL for Sigma 1 2}) Furthermore, if $S$ and $T$ are arithmetically definable $\Sigma^1_2$-sound extensions of $\Sigma^1_2\mhyphen\AC_0$, then we have
        \begin{equation*}
            s^1_2(S)\le s^1_2(T)\iff \Sigma^1_2\mhyphen \AC_0\vdash^{\Pi^1_2} \RFN[\Sigma^1_2](T)\to \RFN[\Sigma^1_2](S).
        \end{equation*}
        See \autoref{Definition: Gamma-RFN} for the meaning of $\RFN[\Sigma^1_2](T)$.        
        The requirement `$S$ and $T$ being extensions of $\Sigma^1_2\mhyphen\AC_0$' is optimal as presented in \autoref{Example: Desired Second equivalence LtoR fails}.
        \item \added{(\autoref{Theorem: Strict s 1 2 comparison equivalence condition}) For arithmetically definable $\Sigma^1_2$-extensions $S$, $T$ of $\Sigma^1_2\mhyphen\AC_0$, we have}
        \begin{equation*}
            \added{s^1_2(S) < s^1_2(T) \iff T \vdash^{\Pi^1_2} \RFN[\Sigma^1_2](S).}
        \end{equation*}
    \end{enumerate}
\end{theorem*}

\subsection{Reflection rank} \label{Subsection: Reflection rank}
\added{Walsh's characterization of ordinal analysis would not be the only result we shift from $\Pi^1_1$ to $\Sigma^1_2$. We will also generalize \emph{reflection rank} defined by Pakhomov and Walsh \cite{PakhomovWalsh2021Reflection} into the $\Sigma^1_2$-context.

As we have raised in \autoref{Question: How to explain the phenomena for natural theories}, logicians have proved $T\vdash \Con(S)$ for natural $S$ and $T$ is to prove $S<_\Con T$. Under the perspective of \autoref{Phenomenon: Consistency comparison and Pi 1 1 comparison may coincide}, logicians have proved $T\vdash \RFN[\Pi^1_1](S)$ or its strengthening to prove $S <_\Con T$. This suggests that the relation $\prec_{\Pi^1_1}$ has a special role in stratifying the consistency strength of natural theories, where
\begin{equation*}
    S\prec_{\Gamma} T \iff T\vdash \RFN[\Gamma](S)
\end{equation*}
for $\Gamma$-sound r.e. extensions $S$ and $T$ of $\ACA_0$.
Indeed, Pakhomov and Walsh \cite{PakhomovWalsh2021Reflection} proved that $\prec_{\Pi^1_1}$ for $\Pi^1_1$-sound r.e.\ extensions of $\ACA_0$ is well-founded.}
Thus, every $\Pi^1_1$-sound r.e.\ extensions of $\ACA_0$ admits a $\prec_{\Pi^1_1}$-rank, which we will call the reflection rank. Pakhomov-Walsh \cite{PakhomovWalsh2021ReflectionInfDeriv} proved that the reflection rank coincides with the proof-theoretic ordinal for $\Pi^1_1$-sound r.e.\ extensions of $\ACA_0^+$, and proved in \cite[\S 5.4]{PakhomovWalsh2021Reflection} that a variant of the reflection rank what they called \emph{robust reflection rank} is `equivalent' to the proof-theoretic ordinal: For $\Gamma=\Pi^1_1$, define
\begin{equation*}
    S\prec^{\check{\Gamma}}_\Gamma T \iff T\vdash^{\check{\Gamma}} \RFN[\Gamma](S).
\end{equation*}
Then we have the following:
\begin{theorem*} \phantom{a}
    \begin{enumerate}
        \item $\prec^{\Sigma^1_1}_{\Pi^1_1}$ is well-founded for $\Pi^1_1$-sound r.e.\ extensions of $\ACA_0$.
        
        \item (\autoref{Theorem: Proof theoretic ordinal and Pi 1 1 rank equivalence}) Let $\rank^{\Sigma^1_1}_{\Pi^1_1}(T)$ be the $\prec^{\Sigma^1_1}_{\Pi^1_1}$-rank of $T$. For $\Pi^1_1$-sound r.e.\ extensions $S$, $T$ of $\ACA_0$, we have
        \begin{equation*}
            |S|_{\Pi^1_1}\le |T|_{\Pi^1_1} \iff \rank^{\Sigma^1_1}_{\Pi^1_1}(S)\le \rank^{\Sigma^1_1}_{\Pi^1_1}(T).
        \end{equation*}

        \item (\cite[Theorem 5.20]{PakhomovWalsh2021Reflection}) Let $T$ be a $\Pi^1_1$-sound extension of $\ACA_0$. If $\rank^{\Sigma^1_1}_{\Pi^1_1}(T)=\alpha$ then $|T|_{\Pi^1_1}=\varepsilon_\alpha$.
    \end{enumerate}
\end{theorem*}

It turns out that the $\Sigma^1_2$-proof-theoretic ordinal also has an analogue of the previous result:

\begin{theorem*} \phantom{a}
    \begin{enumerate}
        \item (\autoref{Proposition: Sigma 1 2 modulo Pi 1 2 reflection rank wellfoundedness}) $\prec^{\Pi^1_2}_{\Sigma^1_2}$ is well-founded for $\Sigma^1_2$-sound r.e.\ extensions of $\Sigma^1_2\mhyphen\AC_0$.
        
        \item (\autoref{Corollary: s 1 2 and Sigma 1 2 reflection rank coincides}) 
        Let $\rank^{\Pi^1_2}_{\Sigma^1_2}(T)$ be the $\prec^{\Pi^1_2}_{\Sigma^1_2}$-rank of $T$
        For $\Sigma^1_2$-sound r.e.\ extensions $S$, $T$ of $\Sigma^1_2\mhyphen\AC_0$,  we have
        \begin{equation*}
            s^1_2(S)\le s^1_2(T) \iff \rank^{\Pi^1_2}_{\Sigma^1_2}(S)\le \rank^{\Pi^1_2}_{\Sigma^1_2}(T).
        \end{equation*}
    \end{enumerate}
\end{theorem*}

\subsection*{The structure of the paper}
In \autoref{Section: Prelims}, we review preliminaries, which include the reflection principle and set-theoretic facts.
Preliminaries for dilators \added{is} too long to be a subsection of \autoref{Section: Prelims}, so we separated it into \autoref{Section: Dilators}. In \autoref{Section: Pi 1 1 and Pi 1 2}, we review facts about $\Pi^1_1$- and $\Pi^1_2$-proof theory. We also prove the correspondence between $\prec^{\Sigma^1_1}_{\Pi^1_1}$-rank and the proof-theoretic ordinal in \autoref{Section: Pi 1 1 and Pi 1 2}, which was not known in previous literature.
In \autoref{Section: Pseudodilators}, we define pseudodilators and $s^1_2$-ordinal and briefly view their properties, which are the main protagonists of this paper. 
In \autoref{Section: s12 metaproperties}, we prove the $\Sigma^1_2$ version of Walsh's first and the second equivalence for the characterization of ordinal analysis \cite{Walsh2023characterizations}, and in \autoref{Section: Sigma 1 2 reflection rank}, we examine the connection between $s^1_2$ and the $\prec^{\Pi^1_2}_{\Sigma^1_2}$-rank.
In \autoref{Section: Pi 1 2 case}, we briefly explore the properties of $\prec^{\Sigma^1_2}_{\Pi^1_2}$-rank.
In \autoref{Section: Finale}, we discuss the philosophical meaning of $s^1_2$ and list some open problems.

\section{Preliminaries}
\label{Section: Prelims}

In this paper, we often work over the language of second-order arithmetic, in which there are two types of objects, natural numbers and real numbers. The standard way to distinguish these two is by using uppercase letters for real numbers and lowercase letters for natural numbers. \added{We try to follow this convention, but we will not strictly follow it. However, lowercase $i,j,k,l,m,n$ will be reserved for natural numbers, and they do not represent real numbers. We also use} indexed quantifiers $\forall^i$ and $\exists^i$ to indicate the range of the quantification: Throughout this paper, $\forall^0$ and $\exists^0$ are quantifiers over natural numbers, and $\forall^1$ and $\exists^1$ are quantifiers over real numbers. 

\subsection{Complexity class and $\Gamma$-Reflection}
A \emph{complexity class} \added{is} a collection of formulas. Throughout this paper, the complexity class will always mean one of $\Sigma^1_n$ or $\Pi^1_n$ for a natural number $n$. If a formula or a sentence is a member of $\Gamma$, then we call it \emph{$\Gamma$-formula} or \emph{$\Gamma$-sentence} respectively. Also, for a complexity class $\Gamma$, we define $\check{\Gamma}$ by
\begin{equation*}
    \check{\Gamma} = \{\lnot\phi \mid \phi \in \Gamma\}.
\end{equation*}
It is \added{folklore} that $\ACA_0$ proves there is a universal $\Pi^1_n$-formula for $n\ge 1$. \added{The existence of a universal $\Sigma^1_1$-formula follows from \cite[Lemma V.1.4]{Simpson2009} and the existence of a universal lightface $\Pi^0_1$-formula that is provable from $\RCA_0$ (See \cite[Definition VII.1.3]{Simpson2009}.) The existence of other universal formulas follows from an inductive argument.} This enables us to talk about the truth of $\Pi^1_n$- or $\Sigma^1_n$-formulas.
\begin{proposition} \pushQED{\qed}
    Let $\Gamma$ be one of $\Pi^1_n$ or $\Sigma^1_n$ for $n\ge 1$. Then there is a $\Gamma$-formula $\phi(i,X)$ such that for every $\Gamma$-formula $\psi(X)$, we can find a natural number $n$ such that
    \begin{equation*}
        \added{\ACA_0 \vdash  \forall^1 X [\phi(n,X)\lr \psi(X)].}
        \qedhere 
    \end{equation*}
\end{proposition}
We denote a fixed universal $\Gamma$-formula by $\True_\Gamma(\phi)$, where $\phi$ is a code for a $\Gamma$-formula.
The following lemma says adding a true $\check{\Gamma}$ sentence does not change the $\Gamma$-reflection principle of a given theory:
\begin{lemma}[$\ACA_0$] \label{Lemma: RFN is stable under adding true check Gamma statements}
    Suppose that $\phi$ is in $\check{\Gamma}$. If \added{$T$ is an extension of $\ACA_0$ and} $\phi$ is true, then $\RFN[\Gamma](T)\to \RFN[\Gamma](T+\phi)$.
\end{lemma}
\begin{proof}
    Suppose that $T+\phi \vdash \psi$ holds for a $\Gamma$-sentence $\psi$. Then $T\vdash \phi\to\psi$, and $\phi\to\psi$ is equivalent to a $\Gamma$-sentence, whose equivalence is provable over the first-order logic. By $\RFN[\Gamma](T)$, $\phi\to\psi$ is true. Since $\phi$ is true, $\psi$ is also true. 
\end{proof}

The following lemma says being $\Gamma$-sound is transitive: That is, if $T$ is $\Gamma$-sound and if $T\vdash \text{$S$ is $\Gamma$-sound}$, then $S$ is also $\Gamma$-sound.
\begin{lemma}[$\ACA_0$] \label{Lemma: Transitivity of Gamma-soundness for r.e. theories}
    Let $S$, $T$ be extensions of $\ACA_0$ such that $T$ is $\Gamma$-sound and $S$ is r.e.\
    If $T\vdash \RFN[\Gamma](S)$, then $\RFN[\Gamma](S)$ is true.
\end{lemma}
\begin{proof}
    If $\Gamma = \Pi^1_n$, then one can easily check that the statement $\RFN[\Pi^1_n](S)$ is $\Pi^1_n$. However, this line of argument does not work when $\Gamma=\Sigma^1_n$ since $\RFN[\Sigma^1_n](S)$ takes the form
    \begin{equation*}
        \forall^0 \phi\in \Sigma^1_n [\Prv_S(\phi) \to \True_{\Sigma^1_n}(\phi)]
    \end{equation*}
    whose complexity takes the form $\forall^0 (\Sigma^0_1\to \Sigma^1_n)$. However, we can still prove $\RFN[\Sigma^1_n](S)$ from $T\vdash \RFN[\Sigma^1_n](S)$ as follows:

    Let $\phi\in \Sigma^1_n$ be a (natural number coding a) $\Sigma^1_n$-formula. If $\Prv_S(\phi)$, then we have $T\vdash \Prv_S(\phi)$ since $\Prv_S(\phi)$ is a true $\Sigma^0_1$ statement. Since $T\vdash \RFN[\Sigma^1_n](S)$, we have $T\vdash \True_{\Sigma^1_n}(\phi)$, which implies $\True_{\Sigma^1_n}(\phi)$ by $\Sigma^1_n$-soundness of $T$.
\end{proof}

\added{
The following theorem expresses the $\Gamma$-reflection version of G\"odel's second incompleteness theorem, which is a modification of \cite[Theorem 2.9]{TowsnerWalsh2024Classification}
\begin{proposition} \label{Proposition: Reflection version of second incompleteness}
    Suppose that $\Gamma$ is either $\Pi^1_n$ or $\Sigma^1_n$ for $n\ge 1$.
    If $T$ is a $\Gamma$-sound, $\check{\Gamma}$-definable extension of $\ACA_0$, then $T$ does not prove $\RFN[\Gamma](T)$.
\end{proposition}
\begin{proof}
    Suppose that $T\vdash \RFN[\Gamma](T)$, and let $\phi$ be a conjunction of all statements in $T$ employed in the proof of $\RFN[\Gamma](T)$ from $T$.
    The statement $T\vdash \phi$ is $\RCA_0$-provably equivalent to a $\check{\Gamma}$-assertion, and we have
    $\ACA_0+\phi+ \text{``} T\vdash \phi\text{''}\vdash \RFN[\Gamma](T)$.

    Working in the theory $\ACA_0 + \phi + \text{``} T\vdash \phi\text{''}$, we know $\RFN[\Gamma](T)$ holds. $T$ proves $\phi$, so $\RFN[\Gamma](\ACA_0+\phi)$. Since ``$T\vdash \phi$'' is equivalent (both over the metatheory and $\ACA_0$-provably) to a $\check{\Gamma}$-assertion, we have $\RFN(\ACA_0 + \phi + \text{``} T\vdash \phi\text{''})$.
    In summary, we get
    \begin{equation*}
        \ACA_0+ \phi + \text{``} T\vdash \phi\text{''}\vdash \RFN[\Gamma]\bigl(\ACA_0+\phi + \text{``} T\vdash \phi\text{''}\bigr).
    \end{equation*}
    In particular, $\ACA_0+\phi + \text{``} T\vdash \phi\text{''}$ proves its own consistency.
    Observe that the theory $\ACA_0+\phi + \text{``} T\vdash \phi\text{''}$ is recursive, so we get a contradiction by G\"odel's second incompleteness theorem.
\end{proof}
}

We will mainly focus on the $\Sigma^1_2$-formulas in this paper.
The class of $\Sigma^1_2$-formulas is not closed under number quantifiers. That is, for a $\Sigma^1_2$-formula $\phi(n)$, there is no guarantee to see that $\forall^0 n \phi(n)$ is equivalent to a $\Sigma^1_2$-formula in general. If we have $\Sigma^1_2\mhyphen\AC_0$, however, we can reduce $\forall^0 n \phi(n)$ to a $\Sigma^1_2$-formula. Moreover, the reduction of $\forall^0 n\phi(n)$ implies $\forall^0 n \phi(n)$ without $\Sigma^1_2\mhyphen\AC_0$. We state this fact as follows: 
\begin{lemma} \label{Lemma: Reducing forall n Sigma 1 2}
    Let $\phi(n)$ be a $\Sigma^1_2$-formula. Then we can find a $\Sigma^1_2$-formula $\psi$ such that
    \begin{enumerate}
        \item $\ACA_0$ proves $\psi\to\forall^0 n \phi(n)$.
        \item $\Sigma^1_2\mhyphen\AC_0$ proves $\forall^0 n \phi(n) \to \psi$.
    \end{enumerate}
    We call $\psi$ a \emph{reduction of $\forall n \phi(n)$}.
\end{lemma}
\begin{proof}
    Let $\phi(n)\equiv \exists^1 X \theta(n,X)$ for a $\Pi^1_1$-formula $\theta$. Then consider the following statement:
    \begin{equation*}
        \psi \equiv \exists^1 F \forall^0 n \theta\bigl(n,F(n)\bigr). \qedhere 
    \end{equation*}
\end{proof}

\subsection{Ordinals and well-orders}
Throughout this paper, we conflate well-orders with ordinals. Thus, for example, we write $\alpha\le \beta$ for well-orders $\alpha$ and $\beta$ if there is an embedding $e\colon \alpha\to\beta$. But the choice of the words is not completely random: We prefer the word `well-order' in a given argument if the argument is formalizable over second-order arithmetic. When we completely work over set theory or externally (i.e., over a metatheory), we prefer the word `ordinal.' For a given linear order $\alpha$, $\WO(\alpha)$ means $\alpha$ is well-ordered.

The \emph{lightface projective ordinals} $\delta^1_n$ have a focal role in descriptive set theory. $\delta^1_n$ is defined by the supremum of all $\Delta^1_n$-definable well-orders of domain $\omega$. In this paper, we will only see two of them, namely $\delta^1_1$ and $\delta^1_2$.
The following results are well-known folklore:
\begin{proposition} \phantom{a} \label{Proposition: Characterizing delta 1 n} \pushQED{\qed}
    \begin{enumerate}
        \item (Spector \cite{Spector1955RecursiveWO}) $\delta^1_1=\omega_1^\CK$.
        \item (\cite[Corollary V.8.3]{Barwise1975}) $\delta^1_2$ is the least stable ordinal $\sigma$, that is, the least $\sigma$ such that $L_\sigma\prec_{\Sigma_1} L$. \qedhere 
    \end{enumerate}
\end{proposition}

\begin{proposition}[{\cite[Corollary V.7.9]{Barwise1975}}]
    \pushQED{\qed}
    \label{Proposition: Set theoretic characterization of delta 1 2}
    $\delta^1_2$ is the set of all $\Sigma_1$-definable ordinals over $L$ without parameters. Similarly, $L_{\delta^1_2}$ is the set of all $\Sigma_1$-definable sets over $L$ without parameters. \qedhere 
\end{proposition}

\begin{proposition} \label{Proposition: Reducing Pi 1 1 and Sigma 1 2 into L} 
    Every $\Pi^1_1$-sentence is equivalent to a $\Sigma_1$-statement over $L_{\omega_1^\CK}$. Every $\Sigma^1_2$-sentence is equivalent to a $\Sigma_1$-statement over $L_{\delta^1_2}$. 
\end{proposition}
\begin{proof}\added{
    For a real number $X$, let $\omega_1^X$ be the least ordinal $\alpha$ such that $L_\alpha[X]\vDash \KP$. Then we can see the following: For every $\Pi^1_1$-formula $\phi(X)$ with only real free variable $X$ displayed, we can find a $\Sigma_1$-formula $\psi(X)$ such that for every real $X$,
    \begin{equation} \label{Formula: Pi 1 1 equivalent to Sigma 1 over an admissible}
        \phi(X) \iff L_{\omega_1^X}[X]\vDash \psi(X).
    \end{equation}
    To see this, we first appeal to the Kleene normal form theorem to find a recursive linear order $\alpha_X$ such that $\phi(X)$ and $\WO(\alpha_X)$ are equivalent. Then we take
    \begin{equation*}
        \psi(X) \equiv \exists \xi\in \Ord \exists f [\text{$f$ is an isomorphism from $\xi$ to $\alpha_X$}]
    \end{equation*}
    The right-to-left implication of \eqref{Formula: Pi 1 1 equivalent to Sigma 1 over an admissible} is clear. For the left-to-right implication, fix $X$ and suppose that $\alpha_X$ is well-founded. Since $\alpha_X$ is recursive in $X$, $\alpha_X \in L_{\omega_1^X}[X]$. 
    By mimicking the proof of $\Sigma$-recursion theorem \cite[Theorem I.6.4]{Barwise1975}, we can see that if $F$ is a $\Sigma$-definable function over $L_{\omega_1^X}[X]$, one can define a $\Sigma_1^{L_{\omega_1^x}[X]}$-definable $G\colon \alpha_X\to L_{\omega_1^X}[X]$ such that for every $i\in \alpha_X$,
    \begin{equation*}
        G(i) = F\bigl(\bigl\lag (j,G(j)) \bigm| \alpha_X\vDash j<i\bigr\rag\bigr).
    \end{equation*}
    We take $F$ by $F(x) = \min (\Ord \setminus \ran x)$, which is uniformly $\Sigma$-definable over an admissible set. We can also see that the resulting $\Sigma$-recursively defined function $G$ gives an embedding from $\alpha_X$ to $\omega_1^X$ whose range is an ordinal, and $\Sigma$-Replacement over $L_{\omega_1^X}[X]$ implies $\xi:=\ran G$ is in $L_{\omega_1^X}[X]$. 
    
    This completes the proof of the first part of the proposition.
    To see the remaining part, suppose that $\phi(X)$ is a $\Pi^1_1$-formula, and let $\psi(X)$ be a $\Sigma_1$-formula satisfying \eqref{Formula: Pi 1 1 equivalent to Sigma 1 over an admissible}.
    By Shoenfield absoluteness, we have
    \begin{equation*}
        \exists^1 X \phi(X) \iff L\vDash \exists^1 X \phi(X) \iff L\vDash \exists X\in \mathbb{R} [L_{\omega_1^X}[X]\vDash \psi(X)].
    \end{equation*}
    Observe that the function $X\mapsto \omega_1^X$ is $\Sigma_1$-definable in $L$, together with the function $\alpha\mapsto L_\alpha$. Hence the statement ``$\exists X\in \mathbb{R} [L_{\omega_1^X}[X]\vDash \psi(X)]$'' is $\Sigma_1$. $L_{\delta^1_2}\prec_{\Sigma_1} L$ by \autoref{Proposition: Characterizing delta 1 n}, so we have}
    \begin{equation*}
        \added{\exists^1 X\psi(X) \iff L_{\delta^1_2}\vDash \exists X\in \mathbb{R} [L_{\omega_1^X}[X]\vDash \psi(X)].} \qedhere 
    \end{equation*}
\end{proof}

\subsection{Logical aspects of dilators}
\autoref{Lemma: Being a dilator is captured by countable WOs} shows the sentence `$D$ is a dilator' is $\Pi^1_2$ for a countable dilator $D$:
\begin{equation*}
    \Dil(D) \equiv \forall^1 \alpha [\WO(\alpha)\to \WO(D(\alpha))].
\end{equation*}
Similar to Kleene normal form theorem, which says every $\Pi^1_1$-formula is equivalent to a formula of the form $\WO(\alpha)$, the \emph{$\Pi^1_2$-completeness theorem of dilators} we will introduce says every $\Pi^1_2$-formula is equivalent to a formula of the form $\Dil(D)$:
\begin{theorem}[Girard \cite{Catlow1994}, \cite{Girard1982Logical}] \pushQED{\qed}
    For every $\Pi^1_2$ formula $\phi(X)$ we can find a Turing machine $D_X$ with an oracle variable $X$ such that $\ACA_0$ proves the following:
    \begin{enumerate}
        \item For every real $X$, $D_X$ is a predilator.
        \item $\forall^1 X [\Dil(D_X) \lr \phi(X)]$.
    \end{enumerate}
\end{theorem}\added{
\begin{proof}[Sketch of the proof]
    There are two types of proof: We can manually build the corresponding predilator, or we can use $\beta$-logic. We follow the former in this sketch. For a proof using $\beta$-logic, see \cite[Lemma 7]{AguileraPakhomov2024Nonlinearity}. 

    Suppose that $\phi(X) \equiv \forall^1 Y \lnot \psi(X,Y)$ for some $\Pi^1_1$-formula $\psi(X,Y)$. A version of Kleene's normal form theorem implies we can find a Turing machine $\alpha_{X,Y}$ with oracle variables $X$ and $Y$ such that $\ACA_0$ proves
    \begin{enumerate}
        \item For every real $X$ and $Y$, $\alpha_{X,Y}$ is a linear order.
        \item $\forall^1 X,Y [\WO(\alpha_{X,Y})\lr \psi(X,Y)]$.
    \end{enumerate}
    Then for a given linear order $\beta$, we build a primitive recursive tree $T_{\alpha,X}(\beta)$ trying to find a real $Y$ and an embedding $\alpha(X,Y)\to \beta$, with techniques used in \cite[Lemma 8]{AguileraPakhomov2023Pi12} or \autoref{Lemma: Disjunctive join of two predilators}.
    We can see an infinite branch of $T_{\alpha,X}(\beta)$ witnesses the validity of $\exists^1 Y \psi(X,Y)$.
    Then turn the tree $T_{\alpha,X}$ into a predilator by imposing the Kleene-Brouwer order.
\end{proof}
}

Thus, we can identify every $\Pi^1_2$ sentence with $\Dil(D)$ for some \emph{primitive recursive} predilator $D$. Similarly, every $\Sigma^1_2$-sentence can be identified with $\lnot\Dil(D)$ for some primitive recursive predilator $D$. 

We will often quantify over the set of recursive predilators, and so checking the complexity of being a recursive predilator is important. The following proposition shows that the collection of primitive recursive predilators has a low complexity:
\begin{proposition}
    The set $\RecPreDil$ of all codes of recursive predilators is $\Pi^0_2$.
\end{proposition}
\begin{proof}
    A recursive object is a recursive predilator if there is a \emph{recursive} set of terms with arity assignment and recursive comparison rules, all of which can be coded by a natural number. Also, we can see that there are countably many arity diagrams, and in fact, we can recursively enumerate all arity diagrams. Thus, a statement `a natural number codes an arity diagram' is $\Delta^0_1$. By inspecting a definition of a predilator, we can see that a real $D$ predilator is $\Pi^0_2(D)$, so we get the desired result.
\end{proof}

By $\Pi^1_1$- and $\Pi^1_2$-completeness of recursive well-orders and dilators, we can represent $\Pi^1_1$ and $\Pi^1_2$ sentences as of the form $\WO(\alpha)$ or $\Dil(D)$ for recursive $\alpha$ and $D$ respectively.
We may ask if there are ways to form new well-orders or dilators that correspond to logical connectives, which take formulas and form a new formula. The following constructions will illustrate operators for linear orders and predilators that correspond to logical connectives:

\begin{lemma}[$\ACA_0$] \label{Lemma: Conjunction of two predilators}
    Let $D_0$, $D_1$ be predilators of field subsets of $\bbN$. Then we can construct a predilator $D=D_0\land D_1$ with the following properties:
    \begin{enumerate}
        \item $\Dil(D)\lr (\Dil(D_0)\land \Dil(D_1))$.
        \item If $D_0$ is a dilator, then for every well-order $\alpha$, $\WO(D_1(\alpha))$ iff $\WO(D(\alpha))$.
        \item If both of $D_0$ and $D_1$ are recursive, then so is $D$.
    \end{enumerate}
\end{lemma}
\begin{proof}
    Take $D(x) = D_0(x)+ D_1(x)$ as an ordered sum. More precisely, $D_0+D_1$ is a dilator whose field is the disjoint sum
    \begin{equation*}
        \field(D_0+D_1) = \{(0,t)\mid t\in \field(D_0)\} \cup \{(1,t)\mid t\in \field(D_1)\}, 
    \end{equation*}
    and $\arity(i,t) = \arity t$.
    The comparison between two terms is given as follows: For two terms $(k_0,t_0),(k_1,t_1) \in \field(D_0+D_1)$ and an arity diagram $\cyrDe$ between them, we have $(k_0,t_0) <_\cyrDe (k_1,t_1)$ iff either
    \begin{enumerate}
        \item $k_0<k_1$, or
        \item $k_0=k_1$ and $t_0<_\cyrDe t_1$.
    \end{enumerate}
    We can see that $D_0+D_1$ is a semidilator and is a predilator if both $D_0$ and $D_1$ are predilators. Proving the statements of the lemma is straightforward, so we leave it to the reader. 
\end{proof}

\begin{lemma}[Aguilera-Pakhomov {\cite[Lemma 8]{AguileraPakhomov2023Pi12}}, $\ACA_0$] \pushQED{\qed}
    For linear orders $\alpha$ and $\beta$ of fields subsets of $\bbN$, we can effectively construct a predilator $D=(\alpha\to\beta)$ such that
    \begin{enumerate}
        \item There is an embedding $e\colon \beta\to D(\alpha)$.
        \item $D$ is a dilator iff $\WO(\alpha)\to\WO(\beta)$ holds.
        \item If $D$ is not a dilator, then $D(\gamma)$ is illfounded iff $\alpha\le \gamma$ for any linear order $\gamma$.
    \end{enumerate}
    Furthermore, if both $\alpha$ and $\beta$ are recursive, then so is $D$. \qedhere 
\end{lemma}

The following construction is a variant of \cite[Lemma 5.2]{PakhomovWalsh2023omegaReflection}, which is also stated in \cite[Proposition 23]{AguileraPakhomov2023Pi13spectrum}. \added{It} gives a `disjunction' of two predilators, 
\begin{lemma}[$\ACA_0$] \label{Lemma: Disjunctive join of two predilators}
    Let $D_0$ and $D_1$ be two predilators of field subsets of $\bbN$. Then we can construct a new predilator $D=D_0\lor D_1$ such that the following holds:
    \begin{enumerate}
        \item $D$ is a dilator iff either $D_0$ or $D_1$ is a dilator.
        \item If $\alpha$ is a well-order, $D_0$ is a dilator and $D_1(\alpha)$ is ill-founded with an infinite decreasing sequence $b$, then there is an embedding $e^b_\alpha \colon D_0(\alpha)\to D(\alpha)$.
        %Furthermore, if $i\colon \alpha\to\beta$ is an embedding and $i(b)$ is the pointwise application of $i$ to $b$, then the following diagram commutes:
        %\begin{equation*}
        %    \begin{tikzcd}
        %    	{D_0(\alpha)} & {D_0(\beta)} \\
        %    	{D(\alpha)} & {D(\beta)}
        %    	\arrow["{D_0(i)}", from=1-1, to=1-2]
        %    	\arrow["{D(i)}"', from=2-1, to=2-2]
        %    	\arrow["{e^b_\alpha}"', from=1-1, to=2-1]
        %    	\arrow["{e^{i(b)}_\beta}", from=1-2, to=2-2]
        %    \end{tikzcd}
        %\end{equation*}
        \item If both of $D_0$ and $D_1$ are recursive, then so is $D$.
    \end{enumerate}
\end{lemma}
\begin{proof}
    For a fixed linear order $\alpha$, let $T(\alpha)$ be the tree trying to construct decreasing sequences on $D_0(\alpha)$ and $D_1(\alpha)$ simultaneously. More precisely, $T(\alpha)$ is the set of all sequences of the form
    \begin{equation*}
        \lag (x_0,y_0),\cdots, (x_n,y_n)\rag \in (D_0(\alpha)\times D_1(\alpha))^{<\omega}
    \end{equation*}
    such that $x_0>_{D_0(\alpha)}\cdots >_{D_0(\alpha)} x_n$ and $y_0>_{D_1(\alpha)}\cdots >_{D_1(\alpha)} y_n$.
    Then define $D(\alpha) = (T(\alpha),<_{\KB})$, but we need to specify the meaning of $<_{\KB}$ to make the point clear that the ordertype of $D(\alpha)$ only depends on $D_0$, $D_1$ and the ordertype of $\alpha$.
    
    Let $c=\lag c_0,\cdots, c_m\rag$, $c'=\lag c_0',\cdots,c_n'\rag$ be members of $T(\alpha)$. Then we say $c <_{\KB} c'$ holds if one of the following holds:
    \begin{enumerate}
        \item If $c'$ is a proper initial segment of $c$.
        \item If it is not the first case, and suppose that $i\le\min(m,n)$ is the least natural number such that $c_i\neq c_i'$. Furthermore, assume that $c_i$ and $c_i'$ take the following form:
        \begin{itemize}
            \item $c_i = (s(\vec{\xi}),t(\vec{\eta}))$.
            \item $c_i' = (s'(\vec{\xi}'),t'(\vec{\eta}'))$.
        \end{itemize}
        Here $s,s'$ are $D_0$-terms, $t,t'$ are $D_1$-terms, and $\vec{\xi}, \vec{\xi}', \vec{\eta}, \vec{\eta}'$ are finite increasing sequences over $\alpha$. Since $D_0$ and $D_1$ are countable, we may view $s,s',t,t'$ as natural numbers.
        Then we say $c<_{\KB} c'$ when we have
        \begin{equation*}
            c_i \vartriangleleft c'_i \equiv 
            (s,t,\vec{\xi},\vec{\eta}) < (s',t',\vec{\xi}',\vec{\eta}')\text{ under the lexicographic order over $\bbN\times\bbN\times [\alpha]^{<\omega}\times [\alpha]^{<\omega}$,}
        \end{equation*}
        where $[\alpha]^{<\omega}$ is the set of finite subsets over $\alpha$, identified with their increasing enumeration, with the lexicographic order. Here we view $D_i$-terms as natural numbers.
    \end{enumerate}
    We can see that $<_\KB$ is a strict linear order over $T(\alpha)$.
    Hence $D$ is a semidilator since the comparison only depends on the terms and the relative order of elements of $\alpha$.\footnote{Alternatively, we can define a support function $\supp_\alpha\colon D(\alpha)\to [\alpha]^{<\omega}$ by $\supp_\alpha(c)=\bigcup_{i<m} \vec{\xi}_i\cup \vec{\eta}_i$, where $c=\lag(s_i(\vec{\xi}_i), t_i(\vec{\eta}_i)) \mid i<m\rag$.}
    Furthermore, $D$ is recursive if both $D_0$ and $D_1$ are recursive.

    To see $D$ is a predilator, let us apply \autoref{Lemma: Monotone semidilator iff predilator}: Let $f,g\colon \alpha\to\beta$ be strictly increasing functions between two linear orders $\alpha$ and $\beta$ such that $f\le g$. We can see that for $c=\lag(s_i(\vec{\xi}_i), t_i(\vec{\eta}_i)) \mid i<m\rag$,
    \begin{equation*}
        D(f)(c) = \lag(s_i(f[\vec{\xi}_i]), t_i(f[\vec{\eta}_i])) \mid i<m\rag.
    \end{equation*}
    So by the definition of $<_\KB$ over $D(\beta)$, we can see that $D(f)(c)\le D(g)(c)$ holds. 
    
    Now let us claim that for each well-order $\alpha$,
    \begin{equation*}
        \WO(D(\alpha))\lr \WO(D_0(\alpha))\lor \WO(D_1(\alpha)).
    \end{equation*}
    Clearly $D(\alpha)$ is ill-founded if both of $D_0(\alpha)$ and $D_1(\alpha)$ are ill-founded. Conversely, suppose that $D(\alpha)$ is ill-founded, so we have an infinite decreasing sequence
    \begin{equation*}
        c(0) >_\KB c(1) >_\KB c(2) >_\KB \cdots.
    \end{equation*}
    We claim that the length of $c(i)$ is unbounded, so we can find a decreasing sequence over $D_0(\alpha)$ and $D_1(\alpha)$.
    
    Suppose not, assume that the length of $c(i)$ is stabilized by, say, $n$. For notational convenience, we may assume that the length of all of $c(i)$ is $n$, so $c(i) = (c_0(i),\cdots, c_{n-1}(i))$. If there are infinitely many $i$ such that $c_0(i)$ are different, then there must be an infinite sequence $\lag i_k\mid k<\omega\rag$ such that $c_0(i_0) \vartriangleright c_0(i_1) \vartriangleright \cdots$, which is impossible. Hence $c_0(i)$ is eventually constant. By the same argument, we can see that all of $c_1(i)$, $c_2(i)$, $\cdots$, $c_{n-1}(i)$ are eventually constant, so we get a contradiction. 
    
    Hence, if one of $D_0$ or $D_1$ is a dilator, then so is $D$.
    It remains to construct $e^b_\alpha$ \added{that} we promised. It follows from modifying the construction in the proof of \cite[Lemma 5.2]{PakhomovWalsh2023omegaReflection}
    Suppose that we are given an infinite decreasing sequence $b=\lag b_n \mid n<\omega \rag$ over $D_1(\alpha)$. 
    
    Now let us enumerate elements of $D_0(\alpha)$ using the lexicographic order $<_{\bbN\times [\alpha]^{<\omega}}$, which is possible since
    \begin{equation*}
        D_0(\alpha) = \{\lag t, \vec{\xi}\rag\in \field(D_0)\times [\alpha]^{<\omega}\mid \arity(t) = |\vec{\xi}|\}.
    \end{equation*}
    Since $\alpha$ is a well-order, $<_{\bbN\times [\alpha]^{<\omega}}$ is also a well-order. For a notational convenience, we assume that we enumerate elements of $D_0(\alpha)$ by $\{\lag t_\nu,\vec{\xi}_\nu\rag\mid \nu \in \kappa\}$, where $\kappa$ is a well-order isomorphic to a subset of $\field(D_0)\times[\alpha]^{<\omega}$ with the lexicographic order.

    Define $e^b_\alpha(\lag t_0,\vec{\xi}_0\rag) = \lag (t_0(\vec{\xi}), b_0)\rag$. Now for a given $\nu\in\kappa$, consider the two cases:
    \begin{enumerate}
        \item $t_\nu(\vec{\xi}_\nu) >_{D_0(\alpha)} t_\mu (\vec{\xi}_\mu)$ for all $\mu<\nu$.
        \item $t_\nu(\vec{\xi}_\nu) <_{D_0(\alpha)} t_\mu (\vec{\xi}_\mu)$  for some $\mu<\nu$.
    \end{enumerate}

    In the first case, set $e^b_\alpha(t_\nu(\vec{\xi}_\nu))=\lag(t_\nu(\vec{\xi}_\nu),b_{0})\rag$. In the second case, let us pick $\mu'<\nu$ satisfying
    \begin{equation} \label{Formula: mu' satisfying minimality}
        t_{\mu'}(\vec{\xi}_{\mu'}) = \min_{D_0(\alpha)} \{t_\mu(\vec{\xi}_\mu) \mid \mu <\nu\land t_\nu(\vec{\xi}_\nu) <_{D_0(\alpha)} t_\mu(\vec{\xi}_\mu) \}, \tag{$\spadesuit$}
    \end{equation}
    where the minimum is computed under the order over $D_0(\alpha)$ induced by a dilator $D_0$. It is possible since $D_0$ is a dilator, so $D_0(\alpha)$ is a well-order. Then define
    \begin{equation*}
        e^b_\alpha(t_\nu(\vec{\xi}_\nu)) = e^b_\alpha(t_{\mu'}(\vec{\xi}_{\mu'})) ^\frown \lag (t_\nu(\vec{\xi}_\nu), b_{l})\rag
    \end{equation*}
    where $l$ is the length of $e^b_\alpha(t_{\mu'}(\vec{\xi}_{\mu'}))$.
    Then we can inductively show the following holds:
    \begin{enumerate}
        \item If $e^b_\alpha(t_\nu(\vec{\xi}_\nu))$ has length $l$, then it takes the form
        \begin{equation*} 
            \lag (t_{\nu_0}(\vec{\xi}_{\nu_0}),b_{0}), (t_{\nu_1}(\vec{\xi}_{\nu_1}),b_1),\cdots, (t_{\nu_{l-1}}(\vec{\xi}_{\nu_{l-1}}),b_{l-1})\rag
        \end{equation*}
        with $\nu_0<\nu_1<\cdots<\nu_{l-1}=\nu$.
        \item $e^b_\alpha\colon D_0(\alpha)\to D(\alpha)$ is order-preserving. 
    \end{enumerate}
    The first item follows almost immediately from the definition of $e^b_\alpha$. Let us claim the second by induction on $\nu\in\kappa$:
    Suppose inductively that $e^b_\alpha\colon \{t_\mu(\vec{\xi}_\mu)\mid \mu<\nu\}\to D(\alpha)$ is order-preserving. Fix $\mu<\nu$, and divide the cases:
    \begin{enumerate}
        \item Suppose that $t_\mu(\vec{\xi}_\mu) > t_\nu(\vec{\xi}_\nu)$. Let $\mu'<\nu$ be the index satisfying \eqref{Formula: mu' satisfying minimality}, that is, $t_{\mu'}(\vec{\xi}_{\mu'})$ is the $D_0(\alpha)$-least value among $t_\zeta(\vec{\xi}_\zeta)$ greater than $t_\nu(\vec{\xi}_\nu)$ such that $\zeta<\nu$.
        Then we have
        \begin{equation*}
            e^b_\alpha(t_\nu(\vec{\xi}_\nu)) = e^b_\alpha(t_{\mu'}(\vec{\xi}_{\mu'})) ^\frown \lag (t_\nu(\vec{\xi}_\nu), b_{l})\rag <_{D(\alpha)} e^b_\alpha(t_{\mu'}(\vec{\xi}_{\mu'})) \le_{D(\alpha)} e^b_\alpha(t_\mu(\vec{\xi}_\mu)),
        \end{equation*}
        where the last inequality is held by the inductive assumption on $e^b_\alpha$.

        \item Now suppose that $t_\mu(\vec{\xi}_\mu) < t_\nu(\vec{\xi}_\nu)$. If $t_\zeta(\vec{\xi}_\zeta) < t_\nu(\vec{\xi}_\nu)$ holds for all $\zeta<\mu$, then
        \begin{equation*}
            e^b_\alpha(t_\nu(\vec{\xi}_\nu))=\lag t_\nu(\vec{\xi}_\nu) , b_{0}\rag <_{D(\alpha)} \lag(t_\zeta(\vec{\xi}_\zeta),b_{0}),\cdots \rag = e^b_\alpha(t_\mu(\vec{\xi}_\mu)),
        \end{equation*}
        for some $\zeta\le \mu<\nu$. We can see that the above inequality holds since $(t_\zeta(\vec{\xi}_\zeta),b_{0}) \vartriangleleft (t_\nu(\vec{\xi}_\nu),b_{0})$, which follows from $\zeta<\nu$ and the definition of the enumeration $\{t_\nu(\vec{\xi}_\nu)\mid \nu\in\kappa\}$.
        
        Now consider the case when $t_\zeta(\vec{\xi}_\zeta) > t_\nu(\vec{\xi}_\nu)$ holds for some $\zeta<\mu$.
        Choose again $\mu'<\nu$ satisfying \eqref{Formula: mu' satisfying minimality}. Then we have $e^b_\alpha(t_\mu(\vec{\xi}_\mu))<e^b_\alpha(t_{\mu'}(\vec{\xi}_{\mu'}))$, so we have the following two possible cases:
        \begin{enumerate}
            \item $e^b_\alpha(t_\mu(\vec{\xi}_\mu))\supsetneq e^b_\alpha(t_{\mu'}(\vec{\xi}_{\mu'}))$, or
            \item There is the least $i$ such that the $i$th component of $e^b_\alpha(t_\mu(\vec{\xi}_\mu))$ is different from that of $e^b_\alpha(t_{\mu'}(\vec{\xi}_{\mu'}))$, and the former is $\vartriangleleft$-less than the latter.
        \end{enumerate}
        In the latter case, we have $e^b_\alpha(t_\mu(\vec{\xi}_\mu))<e^b_\alpha(t_\nu(\vec{\xi}_\nu))$ since $e^b_\alpha(t_\nu(\vec{\xi}_\nu))\supseteq e^b_\alpha(t_{\mu'}(\vec{\xi}_{\mu'}))$.
        In the former case, we have
        \begin{equation*}
            e^b_\alpha(t_\mu(\vec{\xi}_\mu)) = e^b_\alpha(t_{\mu'}(\vec{\xi}_{\mu'}))^\frown \lag (t_{\mu''}(\vec{\xi}_{\mu''}),b_{l}),\cdots, (t_{\mu}(\vec{\xi}_{\mu}),b_{l'}) \rag
        \end{equation*}
        for some $\mu'<\mu''\le \mu$ and $l\le l'<\omega$, and we have
        \begin{equation*}
            e^b_\alpha(t_\nu(\vec{\xi}_\nu)) = e^b_\alpha(t_{\mu'}(\vec{\xi}_{\mu'}))^\frown \lag (t_{\nu}(\vec{\xi}_{\nu}),b_{l}) \rag.
        \end{equation*}
        Since $\mu''<\nu$, we have $(t_{\mu''}(\vec{\xi}_{\mu''}),b_{l}) \vartriangleleft (t_{\nu}(\vec{\xi}_{\nu}),b_{l})$, so  $e^b_\alpha(t_\mu(\vec{\xi}_\mu)) <  e^b_\alpha(t_\nu(\vec{\xi}_\nu))$.
    \end{enumerate}
    Thus $e^b_\alpha$ is an order-preserving map, as desired.
\end{proof}

The following construction is a countably infinite variant of \autoref{Lemma: Disjunctive join of two predilators}:
\begin{lemma}[$\ACA_0$] \label{Lemma: Disjunctions of countably many predilators}
    Let $\lag D_n\mid n<\omega\rag$ be an enumeration of countable predilators. Then we can find a \emph{disjuctive join} $D=\bigvee_{n<\omega} D_n$ of $\lag D_n\mid n<\omega\rag$ such that for every well-order $\alpha$,
    \begin{equation*}
        \WO(D(\alpha))\iff \exists n<\omega \WO(D_n(\alpha)).
    \end{equation*}
    Furthermore, if each of $D_n$ and the enumeration $\{D_n\mid n<\omega\}$ are recursive, then so is $D$.
\end{lemma}
\begin{proof}
    First, for each $\alpha$, define a partial order $P(\alpha)$ as follows: Its domain consists of the following two-dimensional sequence
    \begin{equation*}
        c = 
        \begin{pmatrix}
            c^0_0 & c^0_1 & \cdots & c^0_n \\
            & c^1_1 & \cdots & c^1_n \\
            & & \ddots & \vdots \\
            & & & c^n_n
        \end{pmatrix}
    \end{equation*}
    and each $i$th row of the sequence is $<_{D_n(\alpha)}$-decreasing. The order over $P(\alpha)$ is given by an extension.
    Then consider $D(\alpha) = (P(\alpha),<_\mathsf{KB})$. Then $D$ is a predilator by a similar argument in the proof of \autoref{Lemma: Disjunctive join of two predilators}: 
    More formally, we compare $c,d\in P(\alpha)$ as follows. $c<_\KB d'$ if and only if either
    \begin{enumerate}
        \item $d$ is a proper initial segment of $c$, or
        \item Suppose that $i$ is the least natural number such that the $i$th column of $c$ and that of $d$ are different. Then 
        \begin{equation*} \label{Formula: vartriangleleft definition}
            (c^0_i,\cdots, c^i_i) \vartriangleleft (d^0_i,\cdots, d^i_i) 
        \end{equation*}
        holds, which is given as follows: If $c^j_i = s^j_i(\vec{\xi}^j_i)$ and $d^j_i = t^j_i(\vec{\eta}^j_i)$, then \eqref{Formula: vartriangleleft definition} is equivalent to
        \begin{equation*}
            (\lag s^0_i,\cdots,s^i_i\rag, \vec{\xi}^0_i\cup \cdots \cup \vec{\xi}^i_i) < (\lag t^0_i,\cdots,t^i_i\rag, \vec{\eta}^0_i\cup \cdots \cup \vec{\eta}^i_i)
        \end{equation*}
        over $\bbN\times \alpha^{<\omega}$ under the lexicographic order, where $\lag n_0,\cdots,n_{k-1}\rag$ is a bijective primitive recursive function coding a finite tuple into a natural number.        
    \end{enumerate}
    Then $<_\KB$ is a strict linear order over $P(\alpha)$, and we can see that $D$ is a predilator. Clearly, $D$ is recursive if $\lag D_n\mid n<\omega\rag$ is recursive. 

    It is clear that $D(\alpha)$ is ill-founded if all of $D_i(\alpha)$ is ill-founded. Conversely, suppose that $D(\alpha)$ is ill-founded, so we have an infinite decreasing sequence
    \begin{equation*}
        c(0) >_\KB c(1) >_\KB c(2) >_\KB \cdots.
    \end{equation*}
    By the same argument we provided in the proof of \autoref{Lemma: Disjunctive join of two predilators}, we have that the number of columns of $c(i)$ is unbounded. From this, we can extract an infinite decreasing sequence over $D_i(\alpha)$ for every $i$, so $D_i(\alpha)$ is ill-founded for every $i$.
\end{proof}

\section{\texorpdfstring{$\Pi^1_1$-}{Pi 1 1-}, \texorpdfstring{$\Pi^1_2$-}{Pi 1 2-}proof theory, and Reflection rank}
\label{Section: Pi 1 1 and Pi 1 2}

\subsection{\texorpdfstring{$\Pi^1_1$}{Pi 1 1}-proof theory}
In this subsection, we will briefly review a more technical aspect of $\Pi^1_1$-proof theory that we have not explained in the introduction. 

Proof theorists gauged the strength of a given theory by examining how far the given theory can be capable of recursive well-orders. 
\begin{definition}
    Let $T$ be a $\Pi^1_1$-sound extension of $\ACA_0$. We define its \emph{proof-theoretic ordinal} by
    \begin{equation*}
        |T|_{\Pi^1_1} = \sup\{\alpha\in \mathsf{RWO} \mid T\vdash\WO(\alpha)\},
    \end{equation*}
    where $\mathsf{RWO}$ is the set of all recursive well-orders, and $\WO(\alpha)$ is the assertion that a linear order $\alpha$ is well-ordered.
\end{definition}

The subscript $\Pi^1_1$ of $|T|_{\Pi^1_1}$ means the proof-theoretic ordinal of $T$ gauges its robustness via gauging its $\Pi^1_1$-consequences.
The Kleene normal form theorem justifies the reason for using primitive recursive well-orders. 
\added{Note that the definition of $|T|_{\Pi^1_1}$ is strict supremum since if $T\vdash \WO(\alpha)$, then $T\vdash \WO(\alpha+1)$.

We claim that the assertion `$|T|_{\Pi^1_1}$ is well-ordered' is equivalent to the $\Pi^1_1$-reflection of $T$, but this statement itself is problematic since $|T|_{\Pi^1_1}$ is already an ordinal, which is always well-founded. To make the statement precise, we define the following subsidiary linear order, which is isomorphic to the proof-theoretic ordinal:
\begin{definition}
    Let $T$ be a $\Pi^1_1$-sound extension of $\ACA_0$. We define $\|T\|_{\Pi^1_1}$ by the ordered sum of $T$-provably recursive well-orders. That is, we define
    \begin{equation*} \textstyle
        \|T\|_{\Pi^1_1} := \sum \bigl\{\alpha \in \mathsf{RWO} \bigm| T\vdash\WO(\alpha)\bigr\}.
    \end{equation*}
\end{definition}
To make the definition precise, we assume that $T$ in the previous definition is not a set of theorems, but an enumeration $\lag \sigma_n\mid n\in \bbN\rag$ of theorems of ordertype $\bbN$. For a recursive or arithmetically definable theory as a set, we can find such an enumeration of the theorems of the same definitional complexity. The ordered sum is indexed by the index for the statement $\WO(\alpha)$ in the enumeration.
}
\begin{theorem}\label{Theorem: Pi 1 1 reflection and WOness of proof-theoretic ordinal}
    Let $T$ be a $\Pi^1_1$-sound extension of $\ACA_0$. \added{
    \begin{enumerate}
        \item The ordertype of $\|T\|_{\Pi^1_1}$ is $|T|_{\Pi^1_1}$.
        \item If $T$ is recursive or arithmetically definable, then so is $\|T\|_{\Pi^1_1}$.
        \item For an arithmetically definable $T$, 
        \begin{equation*}
            \ACA_0\vdash \WO(\|T\|_{\Pi^1_1}) \lr \RFN[\Pi^1_1](T).
        \end{equation*}
    \end{enumerate}}
\end{theorem}
\begin{proof}
    \added{
    \begin{enumerate}
        \item To see $|T|_{\Pi^1_1}\le \|T\|_{\Pi^1_1}$, suppose that $T\vdash \WO(\alpha)$ for a recursive linear order $\alpha$. Then $\|T\|_{\Pi^1_1}$ has a component isomorphic to $\alpha$, so $\alpha \le \|T\|_{\Pi^1_1}$. Then the definition of $|T|_{\Pi^1_1}$ yields the desired result.
        For $\|T\|_{\Pi^1_1} \le |T|_{\Pi^1_1}$, observe that if $\|T \|_{\Pi^1_1} = \sum_{i\in\bbN} \alpha_i$ for recursive $\alpha_i$ such that $T\vdash \WO(\alpha_i)$, we have that $T\vdash \WO\bigl(\sum_{i<n} \alpha_i\bigr)$ for every $n$. 
        The ordertype of $\|T\|_{\Pi^1_1}$ is the supremum of the ordertype of $\sum_{i<n}\alpha_i$ for $n\in \bbN$, which is less than or equal to $|T|_{\Pi^1_1}$. This finishes the proof.
        \item Clear.
        \item Let us work in $\ACA_0$. Suppose that $\|T\|_{\Pi^1_1}$ is a well-order and $T\vdash \phi$ for a $\Pi^1_1$-sentence $\phi$. By the Kleene normal form theorem, we can find a recursive linear order $\alpha$ such that $\phi$ and $\WO(\alpha)$ are equivalent (over the metatheory) and $\ACA_0\vdash \phi \lr \WO(\alpha)$. Hence $T\vdash \WO(\alpha)$, so $\alpha$ embeds to $\|T\|_{\Pi^1_1}$ and so $\alpha$ is a well-order. This shows $\phi$ is true. In summary, we have that every $\Pi^1_1$-consequence of $T$ is true.
        Conversely, if $\RFN[\Pi^1_1](T)$ holds, then each additive component of $\|T\|_{\Pi^1_1}$ is a well-order. Hence $\|T\|_{\Pi^1_1}$ is also a well-order. \qedhere 
    \end{enumerate}}
\end{proof}

\added{
Also, the following folklore result credited to Kriesel states that adding a set of true $\Sigma^1_1$-sentences does not change the proof-theoretic ordinal. (See \cite[Proposition 2.1]{Walsh2023characterizations} for its proof.)
\begin{theorem} \pushQED{\qed} \label{Theorem: Kriesel's theorem on PTO}
    Let $T$ be a $\Pi^1_1$-sound extension of $\ACA_0$, and $X$ be a set of true $\Sigma^1_1$-sentences. Then $|T|_{\Pi^1_1} = |T + X|_{\Pi^1_1}$. \qedhere 
\end{theorem}

As stated in \autoref{Theorem: Walsh's characterizations of PTO}, Walsh provided equivalent conditions for comparing proof-theoretic ordinals in \cite{Walsh2022incompleteness} \autoref{Theorem: PTO strict comparison and Pi 1 1 reflection comparison} did not appear in \cite{Walsh2022incompleteness}, so we provide its proof:
\begin{proof}[Proof of \autoref{Theorem: PTO strict comparison and Pi 1 1 reflection comparison}]
    For one direction, suppose that $|S|_{\Pi^1_1} < |T|_{\Pi^1_1}$. One can see that $|T|_{\Pi^1_1}$ is a limit ordinal, so there is $\alpha\in \mathsf{PRWO}$ such that $T\vdash \WO(\alpha)$ and $|S|_{\Pi^1_1}\le \alpha$. Hence the following $\Sigma^1_1$ sentence is true:
    \begin{equation*}
        \sigma \equiv \exists^1 F [\text{$F$ is an embedding from $\|S\|_{\Pi^1_1}$ to $\alpha$.}]
    \end{equation*}
    We need the arithmetical definability of $S$ to guarantee $\sigma$ is $\Sigma^1_1$. We have $T+\sigma\vdash \WO(\|S\|_{\Pi^1_1})$, so \autoref{Theorem: Pi 1 1 reflection and WOness of proof-theoretic ordinal} implies $T+\sigma \vdash \RFN[\Pi^1_1](S)$.
    For the other direction, suppose that $T\vdash^{\Sigma^1_1} \RFN[\Pi^1_1](S)$, so we can find a true $\Sigma^1_1$-sentence $\sigma$ such that $T+\sigma\vdash \WO(\|S\|_{\Pi^1_1})$. This implies 
    \begin{equation*}
        |S|_{\Pi^1_1} \cong \|S\|_{\Pi^1_1} < |T + \sigma|_{\Pi^1_1} = |T|_{\Pi^1_1}.
        \qedhere 
    \end{equation*}
\end{proof}
}

\subsection{\texorpdfstring{$\Pi^1_2$}{Pi 1 2}-proof theory}
Proof-theoretic ordinal gauges the $\Pi^1_1$-consequences of a theory. We need a more complicated object than ordinals to describe an appropriate characteristic for the $\Pi^1_2$-consequences of a theory. It turns out that dilators are the right object to describe $\Pi^1_2$-consequences. The following definition is due to \added{Aguilera-Pakhomov} \cite{AguileraPakhomov2023Pi12}:
\begin{definition}
    Let $T$ be a $\Pi^1_2$-sound extension of $\ACA_0$. Then $|T|_{\Pi^1_2}$ is the dilator unique up to bi-embeddability satisfying the following conditions:
    \begin{enumerate}
        \item\label{Item: Pi12dilator} For a recursive predilator $D$, if $T\vdash \Dil(D)$ then $D$ embeds into $|T|_{\Pi^1_2}$.
        \item (Universality) If a dilator $\hat{D}$ satisfies \eqref{Item: Pi12dilator}, then $|T|_{\Pi^1_2}$ embeds into $\hat{D}$.
    \end{enumerate}
\end{definition}
\added{Aguilera-Pakhomov} \cite{AguileraPakhomov2023Pi12} showed that every $\Pi^1_2$-sound extension of $\ACA_0$ has a proof-theoretic dilator, and is recursive if $T$ is. In fact, \added{the ordered sum of all recursive predilators that are $T$-provably dilators satisfies the criteria for $|T|_{\Pi^1_2}$}. Furthermore, if $T$ is r.e.\added{,} then we can \added{find a recursive dilator satisfying the criteria for $|T|_{\Pi^1_2}$}. 
We may try to define $|T|_{\Pi^1_2}$ for $\Pi^1_2$-unsound theories. In this case, however, $|T|_{\Pi^1_2}$ will not be a dilator, and there is no guarantee that $|T|_{\Pi^1_2}$ satisfies the universality condition.

Also, like proof-theoretic ordinals \added{are} associated with the $\Pi^1_1$-reflection, proof-theoretic dilators are related to $\Pi^1_2$-reflection:
\begin{theorem}[Aguilera-Pakhomov {\cite[Theorem 7]{AguileraPakhomov2023Pi12}}] \label{Theorem: Proof-theoretic dilator and Pi 1 2 reflection}\pushQED{\qed}
    $\ACA_0$ proves the following: If $T$ is a $\Pi^1_2$-sound r.e.\ extension of $\ACA_0$, then $\Dil(|T|_{\Pi^1_2})$ is equivalent to $\RFN[\Pi^1_2](T)$. \qedhere 
\end{theorem}

$|T|_{\Pi^1_2}$ has the following extensional description for ordinals less than $\omega_1^\CK$:
\begin{theorem}[Aguilera-Pakhomov {\cite[Theorem 9]{AguileraPakhomov2023Pi12}}] \label{Theorem: Extensional behavior of Pi12 dilators below omega1CK}
\pushQED{\qed}
    Let $T$ be a $\Pi^1_2$-sound extension of $\ACA_0$.
    If $\alpha$ is a recursive well-order, then
    \begin{equation*}
        |T|_{\Pi^1_2}(\alpha)=|T+\WO(\alpha)|_{\Pi^1_1}. \qedhere 
    \end{equation*}
\end{theorem}

The following theorem shows an extensional description of some proof-theoretic dilators. Some of the results are available in \cite{PakhomovWalsh2023omegaReflection}, and some others are unpublished results by Aguilera and Pakhomov:
\begin{theorem} \pushQED{\qed} \phantom{a}
    \begin{enumerate}
        \item Let $\varepsilon^+$ be a dilator such that $\varepsilon^+(\alpha)$ is the least epsilon number greater than $\alpha$. Then $|\ACA_0|_{\Pi^1_2}=\varepsilon^+$.
        \item Let $\varphi_2^+$ be a dilator such that $\varphi_2^+(\alpha)$ is the least Velben number of the form $\varphi_2(\xi)$ greater than $\alpha$. Then $|\ACA^+_0|_{\Pi^1_2}=\varphi_2^+$.
        \item Let $\Gamma^+$ be a dilator such that $\Gamma^+(\alpha)$ is the least fixpoint for Veblen function greater than $\alpha$. Then $|\ATR_0|_{\Pi^1_2} = \Gamma^+$. 
        \qedhere 
    \end{enumerate}
\end{theorem}

\subsection{Reflection rank, again}
In this subsection, we review and establish facts about the reflection rank mentioned in \autoref{Subsection: Reflection rank}. Let us start with the following result:

\begin{theorem}[Pakhomov-Walsh {\cite[Theorem 3.2]{PakhomovWalsh2021Reflection}}] \pushQED{\qed}
    Let $\Gamma=\Pi^1_n$ or $\Gamma=\Sigma^1_{n+1}$ for $n\ge 1$. Then $\prec_\Gamma$ is well-founded. \qedhere 
\end{theorem}
Pakhomov and Walsh \cite[Theorem 3.2]{PakhomovWalsh2021Reflection} \added{proved the above result only} for $\Gamma=\Pi^1_1$, but the same proof works for any $\Sigma^1_n$ or $\Pi^1_n$ for $n\ge 2$. It allows us to define the \emph{$\Gamma$-reflection rank} $\rank_\Gamma(T)$ for $\Gamma$-sound r.e.\ theories $T$ extending $\ACA_0$. We also recall that $\prec_\Gamma$ is defined over the set of $\Gamma$-sound r.e. extensions of $\ACA_0$.

It is natural to ask about the value and the behavior of the $\Gamma$-reflection rank for natural theories. The following lemma says the \added{reflection} rank is always less than $\omega_1^\CK$:
\begin{lemma}
    Let $T$ be a $\Gamma$-sound r.e.\ extension of $\ACA_0$. Then $\rank_\Gamma(T)<\omega_1^\CK$.
\end{lemma}
\begin{proof}
    Fix a $\Gamma$-sound r.e.\ extension $T$ of $\ACA_0$, and consider the following set:
    \begin{equation*}
        \mathbb{P}=\{S \mid T\vdash\RFN[\Gamma](S)\},
    \end{equation*}
    where $S$ is ranged over all \emph{r.e.\ extension of $\ACA_0$}. That is, we do not assume $\Gamma$-soundness of $S$ in $\bbP$, but the $\Gamma$-soundness of $S$ follows from the $\Gamma$-soundness of $T$ by \autoref{Lemma: Transitivity of Gamma-soundness for r.e. theories}.
    
    We can see that $\bbP$ is a recursive set, and the relation $\prec_{\Gamma}$ over $\Gamma$ is a recursive well-founded relation. Hence we can define a $\prec_\Gamma$-rank function $\rho$ over $\bbP$. Moreover, the supremum of all $\rho(S)$ for $S\in \bbP$ must be less than $\omega_1^\CK$.
    However, we can see that $\rho(S) = \rank_\Gamma(S)$ for all $S\in \bbP$ by $\prec_\Gamma$-induction, and the supremum of all $\rho(S)$ for $S\in \bbP$ coincides with $\rank_\Gamma(T)$. Hence $\rank_\Gamma(T) < \omega_1^\CK$
\end{proof}

We will see that $\prec_\Gamma$ does not cohere well with other proof-theoretic characteristics like the $\Sigma^1_2$-proof-theoretic ordinal $s^1_2(T)$. Hence, let us define the following variant of $\prec_\Gamma$:

\begin{definition}
    Let $\Gamma$ be a complexity class. Define
    \begin{equation*}
        S\prec^{\check{\Gamma}}_\Gamma T\iff T\vdash^{\check{\Gamma}} \RFN[\Gamma](S).
    \end{equation*}
\end{definition}

The argument provided in \added{the proof of} \cite[Theorem 3.2]{PakhomovWalsh2021Reflection} shows $\prec^{\check{\Gamma}}_\Gamma$ is well-founded for $\Gamma = \Pi^1_n$ for extensions of $\Sigma^1_n\mhyphen\AC_0$. This allows us to define the $\prec^{\check{\Gamma}}_\Gamma$-rank for theories extending $\Sigma^1_n\mhyphen\AC_0$ for $\Gamma=\Pi^1_n$. Let us provide the proof for $n=2$:
\begin{proposition}[$\Sigma^1_2\mhyphen\AC_0$] \label{Proposition: WFness Pi 1 2 reflection rank}
    $\prec^{\Sigma^1_2}_{\Pi^1_2}$ for r.e.\ $\Pi^1_2$-sound extensions of $\Sigma^1_2\mhyphen\AC_0$ is well-founded.
\end{proposition}
\begin{proof}
    To prove the above statement from $\Sigma^1_2\mhyphen\AC_0$, we instead prove the inconsistency $\Sigma^1_2\mhyphen\AC_0$ with the following statement:
    \begin{equation*}
        \theta \equiv \exists^1 \lag T_n\mid n\in\mathbb{N}\rag \bigl[\RFN[\Pi^1_2](T_0) \land \forall^0 n \Prv_{T_n}^{\Sigma^1_2}\bigl(\RFN[\Pi^1_2](T_{n+1})\bigr) \land \forall^0 n (T_n \supseteq \Sigma^1_2\mhyphen\AC_0)\bigr].
    \end{equation*}
    To prove this, we claim that $\Sigma^1_2\mhyphen\AC_0 + \theta$ proves its own consistency. 
    Reasoning over $\Sigma^1_2\mhyphen\AC_0 + \theta$, let $\lag T_n \mid n\in\mathbb{N}\rag$ be a sequence witnessing $\theta$. Then consider the following statement:
    \begin{equation*}
        \tau_0 \equiv \exists^1 \lag S_n\mid n\in\mathbb{N}\rag  \bigl[S_0 = T_1\land \forall^0 n \Prv_{S_n}^{\Sigma^1_2}\bigl(\RFN[\Pi^1_2](S_{n+1})\bigr) \land \forall^0 n (S_n \supseteq \Sigma^1_1\mhyphen\AC_0)\bigr].
    \end{equation*}
    By $\Sigma^1_2\mhyphen\AC_0$, $\tau_0$ is equivalent to a $\Sigma^1_2$-statement  $\tau$. Furthermore, $\tau$ is true since $\lag T_{n+1}\mid n\in\mathbb{N}\rag$ witnesses $\tau_0$.
    From $\RFN[\Sigma^1_2](T_0)$ and that $\tau$ is true $\Sigma^1_2$, we get
    \begin{equation*}
        \mathsf{Con}(T_0 + \tau).
    \end{equation*}
    By $T_0\supseteq \Sigma^1_2\mhyphen\AC_0$ and $\Pr_{T_0}(\RFN[\Sigma^1_2](T_1))$, we have
    \begin{equation*}
        \mathsf{Con}\bigl(\Sigma^1_2\mhyphen\AC_0 + \RFN[\Sigma^1_2](T_1) + \tau\bigr).
    \end{equation*}
    However, we can see that $\RFN[\Sigma^1_2](T_1) + \tau$ implies $\theta$ over $\Sigma^1_1\mhyphen\AC_0$, so we have $\mathsf{Con}(\Sigma^1_2\mhyphen\AC_0 + \theta)$.
\end{proof}
We will prove in \autoref{Section: Pi 1 2 case} that $\prec^{\Sigma^1_2}_{\Pi^1_2}$ is well-founded for $\Pi^1_2$-sound r.e.\ \emph{extensions of $\ACA_0$.} Although the above proof gives a weaker result, it has the benefit that a similar argument works for the well-foundedness of $\prec^{\Sigma^1_n}_{\Pi^1_n}$ for extensions of $\Sigma^1_n\mhyphen\AC_0$. 
However, the well-foundedness of $\prec^{\Pi^1_2}_{\Sigma^1_2}$ does not follow from the proof similar to that of \autoref{Proposition: WFness Pi 1 2 reflection rank} even for theories extending $\Sigma^1_2\mhyphen\AC_0$: The main trouble is that $\theta$ in \autoref{Proposition: WFness Pi 1 2 reflection rank} becomes $\Sigma^1_3$ instead of $\Pi^1_2$. We will go back to this issue in \autoref{Section: Sigma 1 2 reflection rank}.

We can ask whether there is any connection between the $\prec_{\Pi^1_1}$-rank and the proof-theoretic ordinal. \cite{PakhomovWalsh2021ReflectionInfDeriv} proved that they coincide for theories extending $\ACA_0^+$:
\begin{theorem}[Pakhomov-Walsh \cite{PakhomovWalsh2021ReflectionInfDeriv}]
    Let $T$ be a $\Pi^1_1$-sound r.e.\ extension of $\ACA_0^+$. Then $\rank_{\Pi^1_1}(T) = |T|_{\Pi^1_1}$.
\end{theorem}

Then what can we say about $\prec^{\Sigma^1_1}_{\Pi^1_1}$-rank?
It turns out that  $\prec^{\Sigma^1_1}_{\Pi^1_1}$-rank gauges the size of the proof-theoretic ordinal in the following manner:
\begin{theorem}[Pakhomov-Walsh {\cite[Theorem 5.20]{PakhomovWalsh2021Reflection}}]
    Let $T$ be a $\Pi^1_1$-sound extension of $\ACA_0$. If $\rank^{\Sigma^1_1}_{\Pi^1_1}(T)=\alpha$ then $|T|_{\Pi^1_1}=\varepsilon_\alpha$.
\end{theorem}

Hence, we immediately get the following:
\begin{theorem} \label{Theorem: Proof theoretic ordinal and Pi 1 1 rank equivalence} \pushQED{\qed }
    Let $S$, $T$ be $\Pi^1_1$-sound r.e.\ extensions of $\ACA_0$. Then we have
    \begin{equation*}
        |S|_{\Pi^1_1} \le |T|_{\Pi^1_1} \iff \rank^{\Sigma^1_1}_{\Pi^1_1}(S)\le \rank^{\Sigma^1_1}_{\Pi^1_1}(T). \qedhere 
    \end{equation*}
\end{theorem}

\section{Pseudodilators: A transition to growth and climax} \label{Section: Pseudodilators}
Recursive predilators that are not dilators appear frequently when we handle $\Sigma^1_2$-formulas. Thus, let us name this concept separately.
\begin{definition}
    A predilator $D$ is a \emph{pseudodilator} if it is not a dilator. For a pseudodilator $D$, the \emph{climax of $D$} is the least ordinal $\alpha$ such that $D(\alpha)$ is not well-founded. We denote the climax of $D$ by $\Clim(D)$.
\end{definition}
Pseudodilators can be used to represent ordinals via their climax. For example, we can find a recursive pseudodilator whose climax is $\omega_1^\CK$, which is impossible to represent via a recursive well-ordering:
\begin{example}
    Let $\phi(i)$ be a universal $\Pi^1_1$-formula \added{with a natural number variable $i$}; That is, for each $\Pi^1_1$-sentence $\sigma$ we can find a natural number $n$ such that $\phi(n)$ and $\sigma$ are equivalent.
    By the Kleene normal form theorem, we can find a recursive well-order $\alpha_i$ such that $\phi(i)\lr \WO(\alpha_i)$ for every natural number $i$. Then we can see that the family $\lag (\alpha_i\to\omega^*) \mid i<\omega\rag$ of dilators is also recursive, where $\omega^*$ is $\omega$ with the reversed order. Notice that if $\alpha_i$ is well-founded, then $(\alpha_i\to\omega^*)$ is a pseudodilator of climax $\alpha_i$.

    Now define $D_i = (\alpha_i\to\omega^*)\land C_{\alpha_i}$,
    where $C_x$ is a constant predilator with value $x$, i.e., $C_x(\alpha)=x$ for all $\alpha$.
    Then take $D$ as the disjunctive sum of all $D_i$s.
    Since the enumeration $\lag D_i\mid i<\omega\rag$ is a recursive enumeration of recursive predilators, $D$ is a recursive predilator. Observe that
    \begin{enumerate}
        \item If $\alpha_i$ is ill-founded, then $D_i(\gamma)$ is ill-founded regardless of $\gamma$. Thus, $D_i$ is a pseudodilator of climax 0.
        \item If $\alpha_i$ is well-founded, then $D_i(\gamma)$ is ill-founded iff $(\alpha_i\to\omega^*)(\gamma)$ is ill-founded, which is equivalent to $\alpha_i\le \gamma$. Hence $D_i$ is a pseudodilator of climax $\alpha_i$.
    \end{enumerate}
    Hence, by \autoref{Lemma: Disjunctive join of two predilators} again, $D(\gamma)$ is ill-founded iff $\gamma\ge \alpha_i$ for all $i$ when $\alpha_i$ is well-founded. Hence $D$ is a recursive predilator whose climax is $\omega_1^\CK$!
\end{example}

However, there is a limit to expressing ordinals in terms of the climax of a recursive pseudodilator:
In fact, if $D$ is a $\Sigma^1_2$-\added{singleton}, then its climax must be less than $\delta^1_2$:
\begin{proposition}
    Let $D$ be a $\Sigma^1_2$-\added{singleton} pseudodilator. Then $\Clim(D)<\delta^1_2$. In addition, the supremum of all $\Clim(D)$ for recursive pseudodilators $D$ is $\delta^1_2$.
\end{proposition}
\begin{proof}
    The statement `$D$ is a pseudodilator' is stated as follows:
    \begin{equation} \label{Formula: D is a pseudodilator}
        \exists^1 \alpha [\WO(\alpha)\land \lnot\WO(D(\alpha))].
    \end{equation}
    Since $D$ is a $\Sigma^1_2$-\added{singleton}, $\lnot\WO(D(\alpha))$ is also $\Sigma^1_2$ because it is equivalent to
    \begin{equation*}
        \exists^1 E \phi(E) \land \bigl[\exists^1 f \forall^0 n [f(n)\in E(\alpha) \land E(\alpha)\vDash f(n) > f(n+1)]\bigr].
    \end{equation*}
    Here \added{$\phi$ is a $\Sigma^1_2$-formula witnessing $D$ being a $\Sigma^1_2$-singleton}, the innermost bracketed formulas are arithmetical, and $\forall n$ is a universal quantifier over the natural numbers. Hence, we can reduce the above formula to a $\Sigma^1_2$-formula.
    Thus \eqref{Formula: D is a pseudodilator} also holds over $L_{\delta^1_2}$ by Shoenfield absoluteness theorem, \autoref{Proposition: Characterizing delta 1 n}, and \autoref{Proposition: Reducing Pi 1 1 and Sigma 1 2 into L}. This shows $\Clim(D)<\delta^1_2$. 
    
    To show $\delta^1_2$ is the supremum of all $\Clim(D)$ for a recursive $D$, fix an ordinal $\alpha<\delta^1_2$ and its $\Sigma_1$-definition $\phi(x)$ over $L$. We follow an argument through $\beta$-preproof presented in \cite[Theorem 28]{AguileraPakhomov2023Pi12}: Consider the $\beta$-preproof $P=\{P(\xi)\mid\xi\in\Ord\}$ trying to show the conjunction of the axioms of $\mathsf{KPi}$, a theory $\KP$ plus a proper class of admissible ordinals with the following:
    \begin{enumerate}
        \item $V=L$ holds.
        \item $\lnot\exists x \phi(x)$.
    \end{enumerate}
    Then $P$ gives a recursive predilator.
    Since $T = \mathsf{KPi} + (V=L) + \exists x \phi(x)$ has a transitive model (namely $L_{\delta^1_2}$), so $P$ is a pseudodilator.
    Furthermore, $P(\xi)$ is well-founded for $\xi<\alpha$ since no model of $T$ of height less than $\alpha$. Hence $\Clim(P)\ge \alpha$.
\end{proof}

It is well-known that every ordinal below $\omega_1^\CK$ has an isomorphic recursive well-order. Similarly, we can ask if every ordinal below $\delta^1_2$ is a climax of a recursive pseudodilator.
However, the proof of \cite[Theorem 24]{AguileraPakhomov2023Pi12} implies there is no recursive pseudodilator $D$ such that $\alpha = \Clim(D)$ if $\alpha$ is \emph{parameter-free $\Sigma^1_1$-reflecting} and is either admissible or a limit of admissibles. The case when $\alpha$ is \emph{neither} an admissible nor a limit admissible is open:
\begin{question}
    Can we characterize ordinals $\alpha < \delta^1_2$ such that $\alpha = \Clim(D)$ for some recursive pseudodilator $D$?
\end{question}
\added{Ressayre \cite[Proposition III.7]{Ressayre1982BoundingGRF} states that ``$\alpha = \Clim(D)$ for a recursive pseudodilator $D$'' is equivalent to what he called the \emph{$\beta^-$-definability of $\alpha$}. However, he only proves one direction of the equivalence, namely, every $\beta^-$-definable ordinal is a climax of a pseudodilator.}

One may ask why we do not use semidilators instead of predilators when we define pseudodilators. The reason is that non-pre-semidilators have trivial climaxes, which make them less interesting.
The following result follows from the proof of \autoref{Proposition: Dilator is a predilator}:
\begin{proposition} \label{Proposition: Climax is a semidilator is trivial} \pushQED{\qed}
    Suppose that $D$ is a semidilator that is not a predilator. Then $\Clim(D)<\omega\cdot 2$. \qedhere 
\end{proposition}

The following notion describes how to compare the climaxes of two pseudodilators:
\begin{definition}
    Let $D_0$ and $D_1$ be two pseudodilators. We say \emph{$D_0$ grows over $D_1$} if $\Clim(D_0)\le \Clim(D_1)$. Equivalently, $D_0$ grows over $D_1$ if the following holds:
    \begin{equation*}
        \Grow(D_0,D_1)\equiv \forall^1 \alpha \bigl[[\WO(\alpha)\land \WO(D_0(\alpha))]\to \WO(D_1(\alpha))\bigr].
    \end{equation*}
\end{definition}
Note that if $D_0$ and $D_1$ are recursive, then $\Grow(D_0,D_1)$ is $\Pi^1_2$.
The choice of the terminology `$D_0$ grows over $D_1$' may be awkward because it states $\Clim(D_0) \le \Clim(D_1)$ and not the \added{reverse}. The reason of the current choice is that if $\Clim(D_0)\le \Clim(D_1)$, then $D_0(\alpha)$ `diverges' to the ill-founded state earlier than $D_1(\alpha)$. This means we may regard $D_0$ `grows' faster than $D_1$, so it diverges sooner. 

In the case of a well-order, we can effectively construct its successor. We can ask the same for pseudodilators: That is, for a given pseudodilator $D$, can we effectively construct a new pseudodilator $D'$ such that $\Clim(D') = \Clim(D)+1$? One way to do this is to consider the following predilator:
\begin{equation*}
    {\textstyle\int}D(\alpha) = {\textstyle \sum}_{i\in\field \alpha} D(\alpha\restricts i),
\end{equation*}
where $\alpha\restricts i$ is the well-order given from $\alpha$ by restricting its field to $\{j\in\field(\alpha)\mid j<_\alpha i\}$.
Let us formulate the construction of $D'$ from $D$ more formally, in a way that the construction works over $\ACA_0$.
\begin{lemma}[$\ACA_0$] \label{Lemma: The successor pseudodilator}
    Let $D$ be a predilator. Then we can construct a predilator $\int D$ such that the following holds: For a well-order $\alpha$,
    \begin{enumerate}
        \item If $D(\alpha)$ is ill-founded, then so is $\int D(\alpha+1)$.
        \item If $D(\alpha\restricts i)$ is well-founded for all $i\in\field(\alpha)$, then so is $\int D(\alpha)$.
        \item If $D$ is recursive, then so is $\int D$.
    \end{enumerate}
\end{lemma}
\begin{proof}
    Let $D$ be a predilator. Let us recall the definition of $\int D$: It comprises a term $t^{\int}$ of arity $(\arity t+1)$. Let us fix a linear order $\alpha$, then for given $D$-terms $s$, $t$, and elements $\vec{\xi},\vec{\eta},\xi',\eta'$ of $\alpha$ such that $\vec{\xi}<\xi'$ and $\vec{\eta}<\eta'$,
    \begin{equation*} \textstyle
        \int D\vDash s^{\int}(\vec{\xi},\xi') < t^{\int}(\vec{\eta},\eta') \iff \text{either } \xi'<\eta' \text{ or }[\xi'=\eta' \text{ and } D(\alpha)\vDash s(\vec{\xi}) < t(\vec{\eta})].
    \end{equation*}

    On the one hand, we can see that there is an embedding from $D(\alpha)$ to $(\int D)(\alpha+1)$ given by $t(\vec{\xi}) \mapsto t^{\int}(\vec{\xi},\alpha)$.
    Therefore if $D(\alpha)$ is ill-founded, then so is $\int D(\alpha+1)$. 
    On the other hand, if $D(\alpha\restriction i)$ is well-founded for every $i\in \field(\alpha)$, then we can find $D$-terms $s_0,s_1,\cdots$ and ordinals $\vec{\xi}_0,\vec{\xi}_1,\cdots\in \alpha$, $\eta_0,\eta_1,\cdots\in \alpha$ \added{such that $\vec{\xi}_i <\eta_i$ for every $i$, we have}
    \begin{equation*}
        \textstyle \int D(\alpha)\vDash 
        s_0^{\int}(\vec{\xi}_0,\eta_0) >
        s_1^{\int}(\vec{\xi}_1,\eta_1) >
        s_2^{\int}(\vec{\xi}_2,\eta_2) > \cdots .
    \end{equation*}
    Since \added{$\alpha$ is well-founded}, there must be $n$ such that $\eta_n = \eta_{n+1} = \eta_{n+2} = \cdots$. For such $n$, we get
    \begin{equation*}
        D(\alpha\restriction \eta_n) \vDash 
        s_0(\vec{\xi}_n) >
        s_1(\vec{\xi}_{n+1}) >
        s_2(\vec{\xi}_{n+2}) > \cdots ,
    \end{equation*}
    contradicting with that $D(\alpha\restriction \eta_n)$ is well-founded.
    Clearly, if $D$ is recursive, then so is $\int D$.
\end{proof}

\subsection{Proof-theoretic pseudodilator}
We will define the \emph{proof-theoretic pseudodilator} $\|T\|_{\Sigma^1_2}$ for a given theory $T$, which is characterized by the following property: For every well-order $\alpha$,
\begin{equation} \label{Formula: Main property of proof-theoretic pseudodilator}
    [\forall^0 D\in \RecPreDil (S\vdash \lnot\Dil(D) \implies \lnot \WO(D(\alpha))) \iff \lnot \WO(\|T\|_{\Sigma^1_2}(\alpha)]
\end{equation}
where $\RecPreDil$ is the set of all recursive predilators. One may naively define $\|T\|_{\Sigma^1_2}$ by
\begin{equation*}
    \|T\|_{\Sigma^1_2} = {\textstyle \bigvee}\lag D \mid \text{$D$ is a recursive predilator such that } T\vdash \lnot\Dil(D)\rag.
\end{equation*}
However, this definition has a \added{fault}: $\|T\|_{\Sigma^1_2}$ is not necessarily recursive even if $T$ is \added{ since the assertion ``$D$ is a recursive predilator'' is $\Pi^0_2$}. We will resolve this issue by modifying our naive definition in the following way:

\begin{definition}
    Let $T$ be a $\Sigma^1_2$-sound extension of $\ACA_0$.
    A \emph{proof-theoretic pseudodilator} $\|T\|_{\Sigma^1_2}$ is defined as follows:
    Consider the enumeration of all pairs $\lag D, \pi\rag$, where $D$ is a recursive predilator and $\pi$ is a $T$-proof for $\lnot\Dil(D)$. For a given pair $\lag D,\pi\rag$, let $D_\pi$ be a recursive predilator obtained from $D$ by replacing all $D$-terms $t$ with $\lag \pi, t\rag$.
    Then define
    \begin{equation*}
        \|T\|_{\Sigma^1_2} = {\textstyle \bigvee}\lag D_\pi \mid D\in\RecPreDil \land \text{$\pi$ is a proof for $\lnot\Dil(D)$ from $T$}\rag.
    \end{equation*}
\end{definition}
If $T$ is recursive, then the statements `$\pi$ is a $T$-proof of a given statement' and the map $\lag D.\pi\rag\mapsto D_\pi$ are also recursive. Hence $\|T\|_{\Sigma^1_2}$ is also recursive.
Like proof-theoretic dilators, we can still define $\|T\|_{\Sigma^1_2}$ for $\Sigma^1_2$-unsound theories; We just do not know if $\|T\|_{\Sigma^1_2}$ is a pseudodilator or not in this case.

We need to check our proof-theoretic pseudodilator meets \eqref{Formula: Main property of proof-theoretic pseudodilator}:
\begin{lemma}[$\ACA_0$] \label{Lemma: Characterization of proof-theoretic pseudodilator}
    Let $D$ be a predilator and $T$ be a $\Sigma^1_2$-sound extension of $\ACA_0$.
    Then for a well-order $\alpha$, the following two are equivalent:
    \begin{enumerate}
        \item For every recursive predilator $D$ such that $T\vdash \lnot\Dil(D)$, $D(\alpha)$ is ill-founded, and
        \item $\|T\|_{\Sigma^1_2}(\alpha)$ is ill-founded.
    \end{enumerate}
\end{lemma}
\begin{proof}
    Fix a well-order $\alpha$.
    Suppose that for every $T$-provably recursive pseudodilator $D$, $D(\alpha)$ is ill-founded, and let $\pi$ be a $T$-proof for $D$ being a pseudodilator. Then $D$ and $D_\pi$ are isomorphic predilators, so $D_\pi(\alpha)$ is also ill-founded. Since $D$ and $\pi$ are arbitrarily, we have that $\|T\|_{\Sigma^1_2}(\alpha)$ is also ill-founded. The remaining direction can be shown similarly.
\end{proof}

Proof-theoretic pseudodilator of $T$ is associated with the $\Sigma^1_2$-reflection of $T$:
\begin{proposition}[$\Sigma^1_2\mhyphen\AC_0$] \label{Proposition: Sigma 1 2 reflection and pseudodilators}
    Let $T$ be a $\Sigma^1_2$-sound extension of $\ACA_0$.  
    Then $\RFN[\Sigma^1_2](T)$ iff $\lnot\Dil(\|T\|_{\Sigma^1_2})$. 
\end{proposition}
\begin{proof}
    For one direction, suppose that $\RFN[\Sigma^1_2](T)$ holds. Then if $D$ is a recursive predilator such that $T\vdash\lnot\Dil(D)$, then $D$ is a pseudodilator. That is, if $T\vdash \lnot\Dil(D)$, then we can find a well-order $\alpha$ such that $D(\alpha)$ is ill-founded. 
    Then by $\Sigma^1_2\mhyphen\AC_0$ applied to the sentence
    \begin{equation*}
        \added{\forall^0 D}\in \mathcal{X} \exists^1 \alpha \bigl[\WO(\alpha)\land \lnot\WO(D(\alpha))\bigr]
    \end{equation*}
    for the set $\mathcal{X}=\bigl\{D\in \RecPreDil \bigm| T\vdash \lnot\Dil(D)\bigr\}$, which can be viewed as a subset of $\mathbb{N}$, we have a set $A$ such that
    \begin{equation*}
        \added{\forall^0 D}\in \mathcal{X} \bigl[\WO\bigl((A)_D\bigr) \land \lnot\WO\bigl(D((A)_D)\bigr)\bigr].
    \end{equation*}
    Now take $\gamma = \sum_{D\in \mathcal{X}}(A)_D$. Then $\gamma$ is a well-order embedding all $(A)_D$ for $D\in \mathcal{X}$, so $D(\gamma)$ is ill-founded for every recursive $D$ such that $T\vdash \lnot\Dil(D)$.
    Hence by \autoref{Lemma: Characterization of proof-theoretic pseudodilator}, $\|T\|_{\Sigma^1_2}$ is also a pseudodilator.

    For the remaining direction, suppose that $\|T\|_{\Sigma^1_2}$ is a pseudodilator. By the completeness of dilators, every $\Sigma^1_2$-sentence is equivalent to $\lnot\Dil(D)$ for some recursive predilator $D$, and this is a theorem of $\ACA_0$. Thus, we can identify a $\Sigma^1_2$-sentence with $\lnot\Dil(D)$ for some recursive predilator $D$.

    Now assume that $T\vdash\lnot\Dil(D)$. Since $\|T\|_{\Sigma^1_2}$ is a pseudodilator, we can find a well-order $\alpha$ such that $\lnot\WO(\|T\|_{\Sigma^1_2}(\alpha))$ holds. But by \autoref{Lemma: Characterization of proof-theoretic pseudodilator}, we have $\lnot\WO(D(\alpha))$. Hence, $D$ is not a dilator.
\end{proof}

\subsection{\texorpdfstring{$\Sigma^1_2$}{Sigma 1 2}-proof-theoretic ordinal}
\begin{definition} \label{Definition: s12 ordinal}
    Let $T$ be a $\Sigma^1_2$-sound extension of $\ACA_0$. We define $s^1_2(T)$ as follows:
    \begin{equation*}
        s^1_2(T) = \sup\{\Clim(D) \mid \text{$D$ is a recursive predilator and }T\vdash \lnot\Dil(D)\}.
    \end{equation*}
\end{definition}
\added{We will see in the next section that the $\Sigma^1_2$-proof-theoretic ordinal has the role of proof-theoretic ordinal for $\Sigma^1_2$-sentences.}
The following is immediate by \autoref{Lemma: Characterization of proof-theoretic pseudodilator}:
\begin{lemma} \pushQED{\qed}
    Let $T$ be a $\Sigma^1_2$-sound extension of $\ACA_0$. Then $\Clim(\|T\|_{\Sigma^1_2}) = s^1_2(T)$. \qedhere 
\end{lemma}

\begin{lemma} \label{Lemma: s 1 2 is the strict supremum}
    Let $T$ be a $\Sigma^1_2$-sound extension of $\ACA_0$.
    Then $s^1_2(T)$ is the strict supremum of all $\Clim(D)$ for $T$-provably recursive pseudodilators $D$. That is, if $T\vdash \lnot\Dil(D)$, then $\Clim(D)<s^1_2(T)$.
\end{lemma}
\begin{proof}
    Let $D$ be a recursive pseudodilator such that $T\vdash \lnot\Dil(D)$.
    Then by \autoref{Lemma: The successor pseudodilator}, we can construct a new recursive predilator $\int D$. 
    
    First, we need to show $T\vdash \lnot\Dil\bigl(\int D\bigr)$.
    Reasoning over $T$, let $\alpha$ be a well-order such that $\lnot\WO(D(\alpha))$. Then by \autoref{Lemma: The successor pseudodilator}, $\lnot\WO\bigl(\int D(\alpha+1)\bigr)$. Hence $\int D$ is also a pseudodilator.
    Furthermore, \autoref{Lemma: The successor pseudodilator} ensures $\Clim\bigl(\int D\bigr) = \Clim D + 1$. Hence we get $\Clim(D) < \Clim\bigl(\int D\bigr) < s^1_2(T)$.
\end{proof}

It is known that the supermum of all $|T|_{\Pi^1_1}$ for $\Pi^1_1$-sound r.e.\ extension of $\ACA_0$ is $\omega_1^\CK$. What about the supremum of $s^1_2(T)$? The following theorem \added{answers}: 
\begin{theorem}[Aguilera-Pakhomov {\cite[Theorem 33]{AguileraPakhomov2023Pi12}}] \pushQED{\qed}
    $\delta^1_2$ is the supremum of all $s^1_2(T)$ for $\Sigma^1_2$-sound r.e.\ extension of $\ACA_0$. \qedhere 
\end{theorem}

The following result shows the $s^1_2(T)$ for some theories $T$:
\begin{theorem}[Aguilera-Pakhomov {\cite[Theorem 36]{AguileraPakhomov2023Pi12}}] \phantom{a}
    \pushQED{\qed} \label{Theorem: s 1 2 for some theories}
    \begin{enumerate}
        \item $s^1_2(\ACA_0) = s^1_2(\KP) = \omega_1^\CK$.
        \item $s^1_2(\Pi^1_1\mhyphen \CA_0) = \omega_\omega^\CK$.
        \item $s^1_2(\Pi^1_2\mhyphen \CA_0) = \sup_{n<\omega} \sigma_n$, where $\sigma_n$ is the least ordinal $\alpha$ such that there are $\alpha_0<\cdots<\alpha_{n-1}=\alpha$ satisfying $L_{\alpha_0}\prec_{\Sigma_1}\cdots\prec_{\Sigma_1} L_{\alpha_{n-1}}$.\footnote{The $\Sigma^1_2$-proof-theoretic ordinal for $\Pi^1_2\mhyphen\CA_0$ presented in \cite{AguileraPakhomov2023Pi12} is incorrect, and the value presented in this paper is due to Aguilera (Private communication).}
        \qedhere 
    \end{enumerate}
\end{theorem}
\begin{proposition} \label{Proposition: s 1 2 for bar induction}
    Let $\BI$ stand for Bar Induction schema (or Transfinite Induction for all second-order formulas.) Then we have the following:
    \begin{enumerate}
        \item $s^1_2(\ATR_0 + \BI) = \omega_1^\CK$.
        \item $s^1_2(\Pi^1_1\mhyphen \CA_0 + \BI) = \omega_\omega^\CK$.
    \end{enumerate}
\end{proposition}
\begin{proof}
    The right-to-left inequality follows from \autoref{Theorem: s 1 2 for some theories}. To show $s^1_2(\ATR_0 + \BI) \le \omega_1^\CK$, let us consider a $\beta$-model $M$ of height $\omega_1^\CK$ (For example, one constructed from the hyperjump of $\emptyset$.) By Theorem VII.2.7. and Lemma VII.2.15 of \cite{Simpson2009}, $M$ satisfies $\ATR_0 + \BI$. If $D$ is a recursive pseudodilator such that $\ATR_0 + \BI \vdash \lnot\Dil(D)$, then we get
    \begin{equation*}
        M \vDash \lnot\Dil(D).
    \end{equation*}
    Hence the climax of $D$ must be less than $\omega_1^\CK$. To prove $s^1_2(\Pi^1_1\mhyphen \CA_0 +\BI) \le \omega_\omega^\CK$, repeat the same argument with the minimal $\beta$-model of $\Pi^1_1\mhyphen\CA_0$ instead.
\end{proof}

It is open whether $s^1_2(T)$ must be either an admissible ordinal or a limit of admissible ordinals if $T$ is a sound r.e.\ extension of $\ACA_0$.

\section{\texorpdfstring{$s^1_2$}{s 1 2}, \texorpdfstring{$\Sigma^1_2$}{Sigma 1 2}-consequences, and \texorpdfstring{$\Sigma^1_2$}{Sigma 1 2}-reflection} \label{Section: s12 metaproperties}
\subsection{The first equivalence}
The main goal of this section is to provide the $\Sigma^1_2$-version of the characterization of $\Pi^1_1$-proof theory.
Let us state the $\Sigma^1_2$ version of the first equivalence stated in \cite{Walsh2023characterizations}:
\begin{proposition} \label{Proposition: First equivalence for Sigma 1 2}
    For $\Sigma^1_2$-sound extensions $S$, $T$ of $\ACA_0$, we have
    \begin{equation*}
        S\subseteq^{\Pi^1_2}_{\Sigma^1_2} T\iff s^1_2(S)\le s^1_2(T).
    \end{equation*}
\end{proposition}

Instead of proving the above claim directly, let us prove the following lemma that will immediately imply \autoref{Proposition: First equivalence for Sigma 1 2}. We separate this lemma from the proof of \autoref{Proposition: First equivalence for Sigma 1 2} for later purposes.

\begin{lemma} \label{Lemma: When the climax happens before s12 ordinal of a theory}
    Let $T$ be a $\Sigma^1_2$-sound extension of $\ACA_0$, and let $D$ be a recursive predilator.
    Then
    \begin{equation*}
        T\vdash^{\Pi^1_2} \lnot\Dil(D)\iff \Clim(D)<s^1_2(T).
    \end{equation*} 
\end{lemma}
\begin{proof}
    For one direction, assume that $T\vdash^{\Pi^1_2} \lnot\Dil(D)$. By $\Pi^1_2$-completeness of dilators, we have a recursive dilator $E$ such that $T + \Dil(E) \vdash \lnot\Dil(D)$.
    Hence we get
    \begin{equation*}
        T\vdash\lnot(\Dil(E)\land\Dil(D)).
    \end{equation*}
    
    Now let $\hat{D}=D\land E$ given by \autoref{Lemma: Conjunction of two predilators}.
    Then $\hat{D}$ is a recursive predilator and $T\vdash \lnot\Dil(\hat{D})$, so $\Clim(\hat{D})<s^1_2(T)$. 
    Since $E$ is a dilator, we have for each ordinal $\alpha$,
    \begin{equation*}
        \WO(\hat{D}(\alpha))\iff\WO(D(\alpha)).
    \end{equation*}
    This implies $\Clim(\hat{D})=\Clim(D)$, which gives the desired inequality.

    Let us show the remaining direction. Suppose that $\Clim(D)< s^1_2(T)$. Then we can find a recursive pseudodilator $E$ such that $T\vdash\lnot \Dil(E)$ and $\Clim(D)\le \Clim(E) < s^1_2(T)$. 
    Since $\Clim(D)\le \Clim(E)$, a $\Pi^1_2$-statement
    \begin{equation*}
        \Grow(D,E)\equiv \forall^1\alpha \bigl[[\WO(\alpha)\land\WO(D(\alpha))]\to \WO(E(\alpha))\bigr]
    \end{equation*}
    is true. Since $T$ proves $\lnot\Dil(E)$,
    \begin{equation*}
        T  \vdash \exists^1 \alpha \bigl[\WO(\alpha) \land \lnot\WO(E(\alpha))\bigr].
    \end{equation*}
    Now let us reason in $T + \Grow(D,E)$: Suppose that $\alpha$ is a well-order such that \added{$\lnot\WO(E(\alpha))$}. Then by the contrapositive of $\Grow(D,E)$, $D(\alpha)$ is also ill-founded. Hence, $D$ is also not a dilator. That is, we get
    \begin{equation*}
        T +\Grow(D,E) \vdash \lnot\Dil(D).
    \end{equation*}
    Since $\Grow(D,E)$ is a true $\Pi^1_2$-statement, we have $T\vdash^{\Pi^1_2}\lnot\Dil(D)$.
\end{proof}

\begin{proof}[Proof of \autoref{Proposition: First equivalence for Sigma 1 2}]
    For one direction, assume that $S\subseteq^{\Pi^1_2}_{\Sigma^1_2} T$.
    If $D$ is a recursive predilator such that $S\vdash\lnot\Dil(D)$, then $T\vdash^{\Pi^1_2}\lnot\Dil(D)$ by the assumption. Thus, by \autoref{Lemma: When the climax happens before s12 ordinal of a theory}, we have $\Clim(D)<s^1_2(T)$. Hence $s^1_2(S)\le s^1_2(T)$.

    For the other direction, assume that $s^1_2(S)\le s^1_2(T)$ and $S\vdash^{\Pi^1_2} \phi$ for a $\Sigma^1_2$-statement $\phi$. By $\Pi^1_2$-completeness of dilators, we may assume that $\phi$ takes the form $\lnot\Dil(D)$ for some recursive \added{pseudo}dilator $D$.
    Then by \autoref{Lemma: When the climax happens before s12 ordinal of a theory}, we have $\Clim(D) < s^1_2(S) \le s^1_2(T)$. Hence, by \autoref{Lemma: When the climax happens before s12 ordinal of a theory} again, we get $T\vdash^{\Pi^1_2} \lnot\Dil(D)$.
\end{proof}

Then the following is immediate, which is a \added{$\Sigma^1_2$-version of \autoref{Theorem: Kriesel's theorem on PTO}}:
\begin{corollary} \pushQED{\qed} \label{Corollary: s 1 2 is stable under true Pi 1 2}
    Let $T$ be an $\Sigma^1_2$-sound extension of $\ACA_0$. If $\sigma$ is any true $\Pi^1_2$ sentence, then $s^1_2(T)=s^1_2(T+\sigma)$. \qedhere
\end{corollary}

\subsection{The second equivalence}
\added{Walsh} \cite{Walsh2023characterizations} also provided a connection between proof-theoretic ordinal and $\Pi^1_1$-reflection. The following shows the connection between $s^1_2$ and $\Sigma^1_2$-reflection:
\begin{proposition} \label{Proposition: Second equivalence LtoR for Sigma 1 2}
    Let $S$ and $T$ be $\Sigma^1_2$-sound extensions of $\ACA_0$ such that $S$ is $\Pi^1_2$-definable and $T$ is $\Sigma^1_2$-definable.
    Then we have
    \begin{equation} \label{Formula: Second equivalence LtoR}
        s^1_2(S)\le s^1_2(T)\implies \Sigma^1_2\mhyphen \AC_0\vdash^{\Pi^1_2} \RFN[\Sigma^1_2](T)\to \RFN[\Sigma^1_2](S).
    \end{equation}
\end{proposition}
\begin{proof}
    Suppose that $s^1_2(S)\le s^1_2(T)$, which implies $S\subseteq^{\Pi^1_2}_{\Sigma^1_2} T$.
    That is, the following sentence becomes true:
    \begin{equation*}
        \theta :\equiv \forall^0 \phi\in \Sigma^1_2 [\Prv_S(\phi)\to \Prv^{\Pi^1_2}_T(\phi)].
    \end{equation*}
    Let us first prove the following lemma:
    \begin{lemma*}
        $\ACA_0+\theta \vdash \RFN[\Sigma^1_2](T)\to \RFN[\Sigma^1_2](S)$.
    \end{lemma*}
    \begin{proof}
        Let us reason over $\ACA_0+\theta$, and suppose that we have $\RFN[\Sigma^1_2](T)$.
        If $\Prv_S(\phi)$ holds for a $\Sigma^1_2$-sentence $\phi$, then by $\theta$, we have $\Prv^{\Pi^1_2}_T(\phi)$. That is, we can find a true $\Pi^1_2$-sentence $\psi$ such that $\Prv_T(\psi\to\phi)$.
        Hence by $\RFN[\Sigma^1_2](T)$ and since $\psi\to\phi$ is $\Sigma^1_2$, $\psi\to\phi$ is true. Since $\psi$ is true, so is $\phi$. In sum, we have
        \begin{equation*}
            \forall^0 \phi\in \Sigma^1_2 [\Prv_S(\phi)\to \mathsf{True}_{\Sigma^1_2}(\phi)]. \qedhere 
        \end{equation*}
    \end{proof}
    Now, let us compute the complexity of $\theta$. One can see \added{that} the following holds:
    \begin{enumerate}
        \item If a theory $T$ is $\Gamma$-definable and $\Gamma\supseteq \Sigma^0_1$, then $\Prv_T(\phi)$ has the complexity of the form $\exists^0\forall^0\Gamma$.
        \item Under the same assumption, $\Prv_T^{\Pi^1_2}$ has the complexity of the form $\exists^0\forall^0(\Pi^1_2\land\Gamma)$.
    \end{enumerate}
    Hence $\theta$ has the following complexity:
    \begin{equation*}
        \forall^0 [(\exists^0\forall^0 \Pi^1_2) \lor (\exists^0\forall^0\Pi^1_2)],
    \end{equation*}
    which reduces to $\forall^0 \exists^0 \Pi^1_2$. 
    Furthermore, by \autoref{Lemma: Reducing forall n Sigma 1 2}, for every $\Pi^1_2$-formula $\phi(m)$ we can find a $\Pi^1_2$-sentence $\tau$ such that 
    \begin{itemize}
        \item $\ACA_0\vdash \exists^0 m \phi(m)\to\tau$, and
        \item $\Sigma^1_2\mhyphen\AC_0\vdash \tau\to \exists^0 m\phi(m)$.
    \end{itemize}
    Hence by using $\Sigma^1_2\mhyphen \AC_0$, we can replace $\theta$ to an equivalent formula $\theta'$ of the form $\forall^0 \Pi^1_2$. Then we have
    \begin{equation*}
        \Sigma^1_2\mhyphen\AC_0+\theta'\vdash \RFN[\Sigma^1_2](T)\to\RFN[\Sigma^1_2](S).
        \qedhere
    \end{equation*}
\end{proof}

Now, let us prove the remaining direction:
\begin{proposition} \label{Proposition: Second equivalence RtoL for Sigma 1 2}
    Suppose that $S$, $T$ be $\Sigma^1_2$-sound extensions of $\ACA_0$ such that \added{$S$ and $T$ are} arithmetically definable. If
    \begin{equation} \label{Formula: ACA0 Pi 1 2 proves Sigma 1 2 RFN implication}
        \ACA_0\vdash^{\Pi^1_2} \RFN[\Sigma^1_2](T)\to\RFN[\Sigma^1_2](S),
    \end{equation}
    then we have $s^1_2(S)\le s^1_2(T)$. The same holds if we replace all occurrences of $\ACA_0$ \added{with} $\Sigma^1_2\mhyphen \AC_0$.
\end{proposition}
\begin{proof}
    Suppose the contrary that \eqref{Formula: ACA0 Pi 1 2 proves Sigma 1 2 RFN implication} holds but $s^1_2(T) < s^1_2(S)$ holds. Then we have a true $\Pi^1_2$-sentence $\tau$ such that the following holds:
    \begin{equation} \label{Formula: ACA0 Pi 1 2 proves Sigma 1 2 RFN implication w witness}
        \ACA_0 + \tau\vdash \RFN[\Sigma^1_2](T)\to\RFN[\Sigma^1_2](S),
    \end{equation}
    
    We can find a recursive pseudodilator $\hat{D}$ such that $S\vdash\lnot\Dil(\hat{D})$ and $s^1_2(T) \le \Clim(\hat{D}) < s^1_2(S)$. \added{In particular}, the following $\Pi^1_2$ sentence is true:
    \begin{equation*}
        \theta \equiv \forall^0 D \in \mathsf{RecPreDil} [\Prv_T(\lnot\Dil(D))\to \Grow(D,\hat{D})].
    \end{equation*}
    Here $\mathsf{RecPreDil}$ is the set of all (codes for) recursive predilators. Since $S$ extends $\ACA_0$, we have
    \begin{equation*}
        S + \theta \vdash \forall^0 D \in \mathsf{RecPreDil} [\Prv_T(\lnot\Dil(D))\to \lnot\Dil(D)].
    \end{equation*}
    Thus, by the $\Pi^1_2$-completeness of recursive dilators, we have
    \begin{equation*}
        S + \theta \vdash \RFN[\Sigma^1_2](T).
    \end{equation*}
    By applying \eqref{Formula: ACA0 Pi 1 2 proves Sigma 1 2 RFN implication w witness}, we can derive $S + \theta + \tau \vdash \RFN[\Sigma^1_2](S)$. Hence, by \autoref{Lemma: RFN is stable under adding true check Gamma statements}, we have
    \begin{equation*}
        S + \theta +\tau \vdash \RFN[\Sigma^1_2](S + \theta + \tau).
    \end{equation*}
    However, no \added{$\Sigma^1_2$-sound arithmetically definable} theory can prove its own $\Sigma^1_2$-reflection, a contradiction.
\end{proof}

Hence we get the following:
\begin{corollary} \pushQED{\qed}
    Let $S$ and $T$ be $\Sigma^1_2$-sound arithmetically definable extensions of $\Sigma^1_2\mhyphen\AC_0$. Then we have
    \begin{equation*}
        s^1_2(S) \le s^1_2(T) \iff \Sigma^1_2\mhyphen\AC_0 \vdash^{\Pi^1_2} \RFN[\Sigma^1_2](T)\to\RFN[\Sigma^1_2](S). \qedhere 
    \end{equation*}
\end{corollary}

\subsection{The restriction in the second equivalence}
The restriction for theories being extensions of $\Sigma^1_2\mhyphen \AC_0$ might be annoying at first glance, which is inherited from \autoref{Proposition: Second equivalence LtoR for Sigma 1 2}. The reader may want to fortify the conclusion \eqref{Formula: Second equivalence LtoR} of \autoref{Proposition: Second equivalence LtoR for Sigma 1 2} by
\begin{equation} \label{Formula: Desired Second equivalence LtoR}
    s^1_2(S)\le s^1_2(T)\implies \ACA_0\vdash^{\Pi^1_2} \RFN[\Sigma^1_2](T)\to \RFN[\Sigma^1_2](S).
\end{equation}
However, the following example shows \eqref{Formula: Desired Second equivalence LtoR} cannot be achieved even if we replace $\ACA_0$-provability with $\Pi^1_1\mhyphen\CA_0$-provability for sound recursive extensions of $\ACA_0$:
\begin{example} \label{Example: Desired Second equivalence LtoR fails}
    By \autoref{Theorem: s 1 2 for some theories} and \autoref{Proposition: s 1 2 for bar induction}, we have $s^1_2(\Pi^1_1\mhyphen\CA_0) = s^1_2(\Pi^1_1\mhyphen\CA_0 + \BI)$. Now suppose that we have
    \begin{equation*}
        \Pi^1_1\mhyphen\CA_0 \vdash^{\Pi^1_2} \RFN[\Sigma^1_2](\Pi^1_1\mhyphen\CA_0) \to \RFN[\Sigma^1_2](\Pi^1_1\mhyphen\CA_0 + \BI). 
    \end{equation*}
    Let $\theta$ be a true $\Pi^1_2$-sentence satisfying
    \begin{equation*}
        \Pi^1_1\mhyphen\CA_0 + \theta + \RFN[\Sigma^1_2](\Pi^1_1\mhyphen\CA_0) \vdash \RFN[\Sigma^1_2](\Pi^1_1\mhyphen\CA_0 + \BI). 
    \end{equation*}
    Then by \autoref{Lemma: RFN is stable under adding true check Gamma statements}, we have
    \begin{equation} \label{Formula: Sigma 1 2 Counterexample 0}
        \Pi^1_1\mhyphen\CA_0 + \RFN[\Sigma^1_2](\Pi^1_1\mhyphen\CA_0) + \theta \vdash \RFN[\Sigma^1_2](\Pi^1_1\mhyphen\CA_0 + \BI + \theta). 
    \end{equation}
    
    \cite[Theorem IX.4.10]{Simpson2009} says $\Pi^1_1\mhyphen\CA_0$ with $\Sigma^1_2$-Induction implies $\RFN[\Pi^1_3](\Sigma^1_2\mhyphen\AC_0)$. Hence by \autoref{Lemma: RFN is stable under adding true check Gamma statements} and $\Sigma^1_2\mhyphen \AC_0 \supseteq \Pi^1_1\mhyphen\CA_0$, we get
    \begin{equation} \label{Formula: Sigma 1 2 Counterexample 1}
        \Pi^1_1\mhyphen\CA_0 + \BI \vdash \RFN[\Sigma^1_2](\Pi^1_1\mhyphen\CA_0).
    \end{equation}
    Combining \eqref{Formula: Sigma 1 2 Counterexample 0} with \eqref{Formula: Sigma 1 2 Counterexample 1}, we have
    \begin{equation*}
        \Pi^1_1\mhyphen\CA_0+ \RFN[\Sigma^1_2](\Pi^1_1\mhyphen\CA_0) + \theta  \vdash \RFN[\Sigma^1_2](\Pi^1_1\mhyphen\CA_0 + \RFN[\Sigma^1_2](\Pi^1_1\mhyphen\CA_0) + \theta). 
    \end{equation*}
    However, no \added{recursive $\Sigma^1_2$-sound} theory can prove its own $\Sigma^1_2$-reflection, a contradiction.
    In sum, if we let $S = \Pi^1_1\mhyphen \CA_0 + \BI$ and $T = \Pi^1_1\mhyphen\CA_0$, then
    \begin{equation*}
        s^1_2(S) \le s^1_2(T) \qquad \text{but} \qquad \Pi^1_1\mhyphen\CA_0 \nvdash^{\Pi^1_2} \RFN[\Sigma^1_2](T)\to \RFN[\Sigma^1_2](S).
    \end{equation*}
\end{example}

\added{
\begin{remark}
    \cite[Theorem IX.4.9]{Simpson2009} states that $\Sigma^1_2\mhyphen\AC_0$ is $\Pi^1_3$-conservative over $\Pi^1_1\mhyphen\CA_0$. Hence, the reader may think
    \begin{equation*}
        \Pi^1_1\mhyphen\CA_0 \nvdash^{\Pi^1_2} \RFN[\Sigma^1_2](T)\to \RFN[\Sigma^1_2](S) \implies \Sigma^1_2\mhyphen\AC_0 \nvdash^{\Pi^1_2} \RFN[\Sigma^1_2](T)\to \RFN[\Sigma^1_2](S),
    \end{equation*}
    for recursive $S$ and $T$, because for a true $\Pi^1_2$-sentence $\tau$, the sentence
    \begin{equation} \label{Formula: tau implies Sigma 1 2 reflection comparison}
        \tau \to [\RFN[\Sigma^1_2](T)\to \RFN[\Sigma^1_2](S)]
    \end{equation}
    looks like a boolean combination of $\Pi^1_2$-sentences, which is $\Pi^1_3$. However, the complexity of $\RFN[\Sigma^1_2](S)$ is $\Pi^1_1\mhyphen\CA_0$-provably $\forall^0 \Sigma^1_2$, at best. Therefore, \eqref{Formula: tau implies Sigma 1 2 reflection comparison} is $\Pi^1_1\mhyphen\CA_0$-provably $\Sigma^1_2\lor \exists^0 \Pi^1_2\lor \forall^0 \Sigma^1_2$ at best, so we do not know if \eqref{Formula: tau implies Sigma 1 2 reflection comparison} is $\Pi^1_1\mhyphen\CA_0$-provably equivalent to a $\Pi^1_3$-sentence.
\end{remark}

\subsection{An equivalent condition for strict comparison}
We finish this section by proving the following:
\begin{theorem} \label{Theorem: Strict s 1 2 comparison equivalence condition}
    Suppose that $S$ and $T$ are $\Sigma^1_2$-sound extensions of $\ACA_0$ such that $S$ is arithmetically definable and $T\supseteq \Sigma^1_2\mhyphen\AC_0$. Then we have
    \begin{equation*}
        s^1_2(S) < s^1_2(T) \iff T\vdash^{\Pi^1_2} \RFN[\Sigma^1_2](S).
    \end{equation*}
\end{theorem}
\begin{proof}
    For one direction, suppose that $s^1_2(S) < s^1_2(T)$. We can find a recursive pseudodilator $D$ such that $s^1_2(S) \le \Clim(D) <s^1_2(T)$. By \autoref{Lemma: When the climax happens before s12 ordinal of a theory}, there is a true $\Pi^1_2$-sentence $\tau$ such that $T + \tau\vdash  \lnot\Dil(D)$. 
    We know that $ s^1_2(S) = \Clim(\|S\|_{\Sigma^1_2}) < \Clim(D)$, so $\Grow(\|S\|_{\Sigma^1_2}, D)$ is a true $\Pi^1_2$-sentence. We have
    \begin{equation*}
        T + \tau + \Grow(\|S\|_{\Sigma^1_2}, D) \vdash \lnot\Dil(\|S\|_{\Sigma^1_2}),
    \end{equation*}
    and \autoref{Proposition: Sigma 1 2 reflection and pseudodilators} with $T\supseteq \Sigma^1_2\mhyphen\AC_0$ implies $T\vdash^{\Pi^1_2} \RFN[\Sigma^1_2](S)$.

    Conversely, suppose that $T + \tau\vdash \RFN[\Sigma^1_2](S)$ for some true $\Pi^1_2$-sentence $\tau$. This implies $T + \tau\vdash \lnot\Dil(\|S\|_{\Sigma^1_2})$ by \autoref{Proposition: Sigma 1 2 reflection and pseudodilators}. $s^1_2(T)$ is a strict supremum, so we have 
    \begin{equation*}
        s^1_2(S) = \Clim(\|S\|_{\Sigma^1_2}) < s^1_2(T+\tau) = s^1_2(T). \qedhere 
    \end{equation*}
\end{proof}
Note that the forward implication of \autoref{Theorem: Strict s 1 2 comparison equivalence condition} holds even when we do not require $T\supseteq \Sigma^1_2\mhyphen\AC_0$. See \autoref{Corollary: s 1 2 comparison and Sigma 1 2 reflection} for its proof.
}

\section{\texorpdfstring{$\Sigma^1_2$}{Sigma 1 2}-Reflection rank}
\label{Section: Sigma 1 2 reflection rank}
The well-foundedness of $\prec^{\Pi^1_2}_{\Sigma^1_2}$ does not follow from the proof similar to that of \autoref{Proposition: WFness Pi 1 2 reflection rank} even for theories extending $\Sigma^1_2\mhyphen\AC_0$. The main trouble is that $\theta$ in \autoref{Proposition: WFness Pi 1 2 reflection rank} becomes $\Sigma^1_3$ instead of $\Pi^1_2$.
To establish the well-foundedness of $\prec^{\Pi^1_2}_{\Sigma^1_2}$, we follow an argument in \cite{Walsh2022incompleteness}:

\begin{proposition} \label{Proposition: Sigma 1 2 modulo Pi 1 2 reflection rank wellfoundedness}
    $\prec^{\Pi^1_2}_{\Sigma^1_2}$ is well-founded for $\Sigma^1_2$-sound r.e.\ extensions of $\Sigma^1_2\mhyphen \AC_0$.
\end{proposition}
\begin{proof}
    To attack \autoref{Proposition: Sigma 1 2 modulo Pi 1 2 reflection rank wellfoundedness}, it is natural to observe the connection between $\prec^{\Pi^1_2}_{\Sigma^1_2}$ and $s^1_2$.

    Suppose not, and let $\lag T_n\mid n<\omega\rag$ be a $\prec^{\Pi^1_2}_{\Sigma^1_2}$-decreasing sequence of $\Sigma^1_2$-sound r.e.\ extensions of $\Sigma^1_2\mhyphen\AC_0$.
    Then we have, for each $n$,
    \begin{equation*}
        T_n \vdash^{\Pi^1_2} \RFN[\Sigma^1_2](T_{n+1}).
    \end{equation*}
    Since $T_n\supseteq \Sigma^1_2\mhyphen\AC_0$, we have $T_n\vdash\lnot\Dil(|T_{n+1}|_{\Sigma^1_2})$ for each $n$ by \autoref{Proposition: Sigma 1 2 reflection and pseudodilators}.
    Hence, by \autoref{Lemma: When the climax happens before s12 ordinal of a theory}, we have
    \begin{equation*}
        s^1_2(T_{n+1}) = \Clim(|T_{n+1}|_{\Sigma^1_2}) < s^1_2(T_n)
    \end{equation*}
    for each $n$, a contradiction.
\end{proof}

It is natural to ask whether $\prec^{\Pi^1_2}_{\Sigma^1_2}$-rank is related to the $\Sigma^1_2$-proof-theoretic ordinal. The subsequent results say it is:

\begin{proposition} \label{Proposition: Sigma 1 2 reflection rank and s 1 2 incr case}
    Let $S$, $T$ be $\Sigma^1_2$-sound r.e.\ extensions of $\Sigma^1_2\mhyphen\AC_0$.
    If $s^1_2(S)\le s^1_2(T)$, then $\rank^{\Pi^1_2}_{\Sigma^1_2}(S)\le \rank^{\Pi^1_2}_{\Sigma^1_2}(T)$.
\end{proposition}
\begin{proof}
    Suppose that $\alpha < \rank^{\Pi^1_2}_{\Sigma^1_2}(S)$. We claim the following lemma:
    \begin{lemma*}
        There is a $\Sigma^1_2$-sound r.e.\ extension $U$ of $\Sigma^1_2\mhyphen\AC_0$ such that $\rank^{\Pi^1_2}_{\Sigma^1_2}(U) = \alpha$ and $S\vdash^{\Pi^1_2}\RFN[\Sigma^1_2](U)$.
    \end{lemma*}
    \begin{proof}
        Let $\bbP$ be the set of all $\Sigma^1_2$-sound r.e.\ extension of $\Sigma^1_2\mhyphen\AC_0$ such that $S\vdash \RFN[\Sigma^1_2](U)$, and consider the structure $(\bbP, \prec^{\Pi^1_2}_{\Sigma^1_2})$.
        By \autoref{Proposition: Sigma 1 2 modulo Pi 1 2 reflection rank wellfoundedness}, $\prec^{\Pi^1_2}_{\Sigma^1_2}$ over $\bbP$ is well-founded, so it defines a rank function $\rho$ over $\bbP$.

        We claim that $\rho(U) = \rank^{\Pi^1_2}_{\Sigma^1_2}(U)$ for all $U\in \bbP$ by $\prec^{\Pi^1_2}_{\Sigma^1_2}$-induction on $\bbP$:
        Suppose that $\rho(V) = \rank^{\Pi^1_2}_{\Sigma^1_2}(V)$ holds for all $V\in\bbP$ such that $V\prec^{\Pi^1_2}_{\Sigma^1_2}U$.
        Then
        \begin{equation*}
            \rho(U) = \sup\{\rho(V)+1\mid V\prec^{\Pi^1_2}_{\Sigma^1_2} U\} = \sup\{\rank^{\Pi^1_2}_{\Sigma^1_2}(V)+1\mid V\prec^{\Pi^1_2}_{\Sigma^1_2} U\} = \rank^{\Pi^1_2}_{\Sigma^1_2}(U),
        \end{equation*}
        proving the claim.
        By a similar argument, we can see that the supremum of all $\rho(U)$ for $U\in \bbP$ is $\rank^{\Pi^1_2}_{\Sigma^1_2}(S)$, which proves the desired claim.
    \end{proof}

    Going back to the proof of the main proposition, let $U$ be a theory such that $\rank^{\Pi^1_2}_{\Sigma^1_2}(U) = \alpha$ and $S\vdash\RFN[\Sigma^1_2](U)$.
    Since $s^1_2(S)\le s^1_2(T)$, which implies $S\subseteq^{\Pi^1_2}_{\Sigma^1_2}T$, we get $T\vdash^{\Pi^1_2}\RFN[\Sigma^1_2](U)$.
    That is, $U\prec^{\Pi^1_2}_{\Sigma^1_2} T$, which implies $\alpha = \rank^{\Pi^1_2}_{\Sigma^1_2}(U) < \rank^{\Pi^1_2}_{\Sigma^1_2}(T)$.
\end{proof}

\begin{proposition} \label{Proposition: Sigma 1 2 reflection rank and s 1 2 strict case}
    Let $S$, $T$ be $\Sigma^1_2$-sound r.e.\ extensions of $\Sigma^1_2\mhyphen\AC_0$.
    If $s^1_2(S)< s^1_2(T)$, then $\rank^{\Pi^1_2}_{\Sigma^1_2}(S)< \rank^{\Pi^1_2}_{\Sigma^1_2}(T)$.
\end{proposition}
\begin{proof}
    Let $\hat{D}$ be a recursive pseudodilator such that $s^1_2(S) \le \Clim(\hat{D}) < s^1_2(T)$.
    Then by \autoref{Lemma: When the climax happens before s12 ordinal of a theory}, we have $T\vdash^{\Pi^1_2}\lnot\Dil(\hat{D})$.

    Now let $D$ be a recursive pseudodilator such that $S\vdash\lnot\Dil(D)$. Then we have $\Clim(D) \le \Clim(\hat{D})$, so the following sentence is true:
    \begin{equation*}
        \theta \equiv \forall^0 D\in \RecPreDil [\Prv_S(\lnot\Dil(D))\to \Grow(D,\hat{D})].
    \end{equation*}
    Clearly $\theta$ is $\Pi^1_2$.
    Furthermore, we get
    \begin{equation*}
        T + \theta \vdash \forall^0 D\in \RecPreDil [\Prv_S(\lnot\Dil(D))\to \lnot\Dil(D)].
    \end{equation*}
    Hence by the $\Pi^1_2$-completeness of dilators, we get
    \begin{equation*}
        T + \theta \vdash \RFN[\Sigma^1_2](S).
    \end{equation*}
    Since $\theta$ is true, $T \vdash^{\Pi^1_2} \RFN[\Sigma^1_2](S)$, which is the definition of $S\prec^{\Pi^1_2}_{\Sigma^1_2} T$. Hence $\rank^{\Pi^1_2}_{\Sigma^1_2}(S)<\rank^{\Pi^1_2}_{\Sigma^1_2}(T)$.
\end{proof}

Combining \autoref{Proposition: Sigma 1 2 reflection rank and s 1 2 incr case} and \autoref{Proposition: Sigma 1 2 reflection rank and s 1 2 strict case}, we get the following:
\begin{corollary} \label{Corollary: s 1 2 and Sigma 1 2 reflection rank coincides}
    \pushQED{\qed}
    Let $S$, $T$ be $\Sigma^1_2$-sound r.e.\ extensions of $\Sigma^1_2\mhyphen\AC_0$. Then we have
    \begin{equation*}
        s^1_2(S)\le s^1_2(T) \iff \rank^{\Pi^1_2}_{\Sigma^1_2}(S)\le \rank^{\Pi^1_2}_{\Sigma^1_2}(T).
        \qedhere 
    \end{equation*}
\end{corollary}

Note that the only reason we assumed $S$ and $T$ be extensions of $\Sigma^1_2\mhyphen\AC_0$ is to impose them $\prec^{\Pi^1_2}_{\Sigma^1_2}$-rank. From the proof of \autoref{Proposition: Sigma 1 2 reflection rank and s 1 2 strict case}, we can extract the following result:
\begin{corollary} \label{Corollary: s 1 2 comparison and Sigma 1 2 reflection}
    \pushQED{\qed}
    Let $S$, $T$ be $\Sigma^1_2$-sound r.e.\ extensions of $\ACA_0$.
    If $s^1_2(S)< s^1_2(T)$, then $T \vdash^{\Pi^1_2}\RFN[\Sigma^1_2](S)$. \qedhere 
\end{corollary}

\section{A glimpse to Case \texorpdfstring{$\Pi^1_2$}{Pi 1 2}}
\label{Section: Pi 1 2 case}

\autoref{Corollary: s 1 2 is stable under true Pi 1 2} shows $s^1_2(T)$ is stable under true $\Pi^1_2$-statements, and we may ask \added{whether} a similar property holds at \added{the} $\Pi^1_2$-level. The following proposition, which shows adding a true $\Sigma^1_2$-true sentence does not change the eventual behavior of the proof-theoretic dilator, gives some hint:
\begin{proposition} \label{Proposition: True Sigma 1 2 does not change the eventual value of dilator}
    Let $T$ be a $\Pi^1_2$-sound theory and let $F$ be a recursive pseudodilator. If $\alpha$ is a well-order such that $F(\alpha)$ is ill-founded, then for every $\beta\ge\alpha$ we have
    \begin{equation*}
        |T|_{\Pi^1_2}(\beta) = |T+\lnot\Dil(F)|_{\Pi^1_2}(\beta).
    \end{equation*}
\end{proposition}
\begin{proof}
    Clearly, $T\subseteq T+\lnot\Dil(F)$, so $|T|_{\Pi^1_2}$ is embeddable into $|T+\lnot\Dil(F)|_{\Pi^1_2}$. This shows $|T|_{\Pi^1_2}(\beta)\le |T+\lnot\Dil(F)|_{\Pi^1_2}(\beta)$ for every $\beta$.

    To show the reverse direction, let $E$ be a recursive dilator such that $T + \lnot\Dil(F)\vdash \Dil(E)$. Then we have
    \begin{equation*}
        T \vdash \Dil(F)\lor \Dil(E).
    \end{equation*}
    Hence $T \vdash \Dil(F\lor E)$. By \autoref{Lemma: Disjunctive join of two predilators}, $F\lor E$ is a recursive dilator. Thus $F\lor E$ is embedded into $|T|_{\Pi^1_2}$. Also, $\beta\ge\alpha$ implies $F(\beta)$ is ill-founded, so there is an embedding $E(\beta)\le (F\lor E)(\beta)$. Therefore
    \begin{equation*}
        E(\beta)\le (F\lor E)(\beta)\le |T|_{\Pi^1_2}(\beta). \qedhere 
    \end{equation*}
\end{proof}

\added{In particular}, if $F$ is a recursive pseudodilator, then there is $\alpha<\delta^1_2$ such that $F(\alpha)$ is ill-founded. Thus, we get the following:
\begin{corollary} \pushQED{\qed}
    Let $T$ be a $\Pi^1_2$-sound extension of $\ACA_0$ and let $F$ be a recursive pseudodilator. Then
    \begin{equation*}
        |T|_{\Pi^1_2}(\delta^1_2) = |T+\lnot\Dil(F)|_{\Pi^1_2}(\delta^1_2). \qedhere 
    \end{equation*}
\end{corollary}

By using the eventual value of the proof-theoretic dilator $|T|_{\Pi^1_2}(\delta^1_2)$, we can show that $\prec^{\Sigma^1_2}_{\Pi^1_2}$ is well-founded:
\begin{lemma}
    Suppose that $S$, $T$ be $\Pi^1_2$-sound r.e.\ extensions of $\ACA_0$. If $T\vdash^{\Sigma^1_2} \RFN[\Pi^1_2](S)$, then we can find $\alpha<\delta^1_2$ such that 
    \begin{equation*}
        \forall \beta\ge\alpha \bigl[|S|_{\Pi^1_2}(\beta) < |T|_{\Pi^1_2}(\beta)\bigr].
    \end{equation*}
    \added{In particular}, we have $|S|_{\Pi^1_2}(\delta^1_2) < |T|_{\Pi^1_2}(\delta^1_2)$.
\end{lemma}
\begin{proof}
    Let us prove the following claim first: If $T \vdash \RFN[\Pi^1_2](S)$, then we have 
    \begin{equation*}
        |S|_{\Pi^1_2}\cdot 2\le |T|_{\Pi^1_2}.
    \end{equation*}
    Since $\ACA_0$ proves $\RFN[\Pi^1_2](S)$ is equivalent to $\Dil(|S|_{\Pi^1_2})$, $|S|_{\Pi^1_2}$ appears in the enumeration of all $T$-provable recursive dilators.
    Furthermore, $T \vdash \RFN[\Pi^1_2](S)$ implies every $\Pi^1_2$-consequence of $S$ is a theorem of $T$. This means every $S$-provable recursive dilator is also a $T$-recursive dilator.
    However, no $S$-provable recursive dilator is equal to $|S|_{\Pi^1_2}$, or $S$ proves its own $\Pi^1_2$-reflection that is impossible. Since we defined $|T|_{\Pi^1_2}$ by the sum of all $T$-provable recursive dilators, we have the desired inequality.

    Now, let us prove the lemma we stated. If $F$ is a recursive pseudodilator and if $T+\lnot\Dil(F)$ proves $\RFN[\Pi^1_2](S)$, then we have
    \begin{equation*}
        |S|_{\Pi^1_2}\cdot 2\le |T+\lnot\Dil(F)|_{\Pi^1_2}.
    \end{equation*}
    If $\alpha$ is a well-order such that $F(\alpha)$ is ill-founded, then by \autoref{Proposition: True Sigma 1 2 does not change the eventual value of dilator}, we have 
    \begin{equation*}
        |S|_{\Pi^1_2}(\beta)\cdot 2 \le |T+\lnot\Dil(F)|_{\Pi^1_2}(\beta) = |T|_{\Pi^1_2}(\beta)
    \end{equation*}
    for all $\beta\ge\alpha$. Furthermore, since non-zero constant dilators of a natural number value are $S$-provable dilators, we have $|S|_{\Pi^1_2}(\beta)\ge \omega$ for all $\beta$. This shows 
    \begin{equation*}
        |S|_{\Pi^1_2}(\beta) < |S|_{\Pi^1_2}(\beta)\cdot 2,
    \end{equation*}
    proving the desired inequality.
\end{proof}

Now the following is immediate by the definition of $\prec^{\Sigma^1_2}_{\Pi^1_2}$ and the previous lemma:
\begin{corollary} \pushQED{\qed}
    $\prec^{\Sigma^1_2}_{\Pi^1_2}$ is well-founded for $\Pi^1_2$-sound r.e.\ theories extending $\ACA_0$. \qedhere 
\end{corollary}

We proved in previous sections that $s^1_2(T)$ ranks $\Sigma^1_2$-consequences of $\Sigma^1_2$-sound theories (modulo $\Pi^1_2$-oracle) in a linear way, and it gauges the strength of $\Sigma^1_2$-reflection of $T$ if $T$ is an extension of $\Sigma^1_2\mhyphen\AC_0$.
We may ask whether we can get an ordinal characteristic of $T$ that captures its  $\Pi^1_2$-consequences modulo $\Sigma^1_2$-oracle. However, the following result by Aguilera and Pakhomov says \added{that} obtaining such an ordinal characteristic for $\Pi^1_2$-consequences is impossible:
\begin{theorem}[\added{Aguilera-Pakhomov }\cite{AguileraPakhomov2024Nonlinearity}] \pushQED{\qed}
    \label{Theorem: Non linearity of Pi 1 2}
    We can find a recursive $T\supseteq \ACA_0$ and $\Pi^1_2$-sentences $\phi_0$, $\phi_1$ such that
    \begin{enumerate}
        \item Both of $T + \phi_0$ and $T + \phi_1$ are $\Pi^1_2$-sound.
        \item Neither $T \vdash^{\Sigma^1_2} \phi_0\to\phi_1$ nor $T \vdash^{\Sigma^1_2} \phi_1\to\phi_0$ holds. \qedhere 
    \end{enumerate}
\end{theorem}

\begin{corollary}
    There is no ordinal assignment $o(T)$ for $\Pi^1_2$-sound r.e.\ extensions $T$ of $\ACA_0$ such that
    \begin{equation*}
        o(S) \le o(T) \iff S \subseteq^{\Sigma^1_2}_{\Pi^1_2} T.
    \end{equation*}
\end{corollary}
\begin{proof}
    Let $T$, $\phi_0$, $\phi_1$ be a theory and $\Pi^1_2$-sentences providede by \autoref{Theorem: Non linearity of Pi 1 2}.
    Then neither $T + \phi_0 \vdash^{\Sigma^1_2} \phi_1$ nor $T + \phi_1 \vdash^{\Sigma^1_2} \phi_0$ hold, so $\subseteq^{\Sigma^1_2}_{\Pi^1_2}$ cannot be linear.
\end{proof}

However, the above result does not rule out an ordinal characteristic capturing the $\prec^{\Sigma^1_2}_{\Pi^1_2}$-rank of a theory. We conjecture the eventual value of a proof-theoretic dilator $|T|_{\Pi^1_2}(\delta^1_2)$ should have this role:

\begin{question}
    Let $\rank^{\Sigma^1_2}_{\Pi^1_2}(T)$ be the $\prec^{\Sigma^1_2}_{\Pi^1_2}$-rank of $T$ for $\Pi^1_2$-sound r.e.\ extension $T$ of $\ACA_0$. Then does the following hold? For two $\Pi^1_2$-sound r.e. extensions $S$, $T$ of $\ACA_0$,
    \begin{equation*}
        \rank^{\Sigma^1_2}_{\Pi^1_2}(S) \le \rank^{\Sigma^1_2}_{\Pi^1_2}(T) \iff |S|_{\Pi^1_2}(\delta^1_2)\le |T|_{\Pi^1_2}(\delta^1_2).
    \end{equation*}
\end{question}

\section{Concluding remarks}
\label{Section: Finale}
Throughout this paper, we generalized the main properties of $\Pi^1_1$-proof theory to $\Sigma^1_2$-proof theory. Let us provide some remarks and questions based on the results we produced.

\cite{AguileraPakhomov2023Pi12} stated that ``Although the approach via $\Sigma^1_2$-consequences might initially seem more natural, it appears that this type of analysis is not as informative as the approach via $\Pi^1_2$ consequences...'' The results in this paper bolster this claim in the following sense:
\begin{enumerate}
    \item The usual ordinal analysis for the $\Pi^1_1$-proof-theoretic ordinal allows further improvement in the sense that taking a close look at $\Pi^1_1$-ordinal analysis allows extracting information about $\Pi^0_2$-consequences of a theory. One may hope that a proper method for $\Sigma^1_2$-ordinal analysis should allow extracting information about $\Pi^1_1$- or $\Sigma^1_1$-consequences of a theory, whose way is unclear at least for the author. The current method to obtain an upper bound for $s^1_2(T)$ relies on model-theoretic arguments with $\beta$-models, which does not allow extracting useful information about $\Pi^1_1$- or $\Sigma^1_1$-consequences of $T$ since $\beta$-models are absolute for $\Pi^1_1$-statements. It contrasts with the $\Pi^1_2$-ordinal analysis, which results in the proof-theoretic dilator $|T|_{\Pi^1_2}$: We can easily extract the $\Pi^1_1$-proof-theoretic ordinal from $|T|_{\Pi^1_2}$ by \autoref{Theorem: Extensional behavior of Pi12 dilators below omega1CK}.

    \item Unlike the $\Pi^1_1$-case, the connection between $\Sigma^1_2$-proof-theoretic ordinal and $\Sigma^1_2$-reflection works only for theories extending $\Sigma^1_2\mhyphen\AC_0$. For example, the second equivalence \autoref{Proposition: Second equivalence LtoR for Sigma 1 2} holds only for theories extending $\Sigma^1_2\mhyphen\AC_0$, and \autoref{Example: Desired Second equivalence LtoR fails} shows the equivalence may fail for theories extending $\Pi^1_1\mhyphen\CA_0$.
    Furthermore, we proved \autoref{Proposition: Sigma 1 2 reflection and pseudodilators}, the equivalence between $\RFN[\Sigma^1_2](T)$ and $\lnot\Dil(\|T\|_{\Sigma^1_2})$, over $\Sigma^1_2\mhyphen\AC_0$ and not over $\ACA_0$.
    It is different from what happens for proof-theoretic dilators: For example, the equivalence between $\RFN[\Pi^1_2](T)$ and $\Dil(|T|_{\Pi^1_2})$ is a theorem of $\ACA_0$ as stated in \autoref{Theorem: Proof-theoretic dilator and Pi 1 2 reflection}.

    \item For $\Pi^1_2$-unsound theories, we can still distill the \emph{$\Pi^1_2$-soundness ordinal $o^1_2(T)$} (cf. \cite[\S 4]{AguileraPakhomov2023Pi12}) that gives a useful information about the degree of soundness. We can extract $o^1_2(T)$ from a given theory $T$ by tracking the least ordinal $\alpha$ such that $|T|_{\Pi^1_2}(\alpha)$ is ill-founded.
    However, in $\Sigma^1_2$-case, there is no obvious way to extract a similar characteristic: If $T$ is $\Sigma^1_2$-unsound, then $\|T\|_{\Sigma^1_2}(\alpha)$ will always be well-founded for $\alpha$, and the author does not know how to extract useful information from $\|T\|_{\Sigma^1_2}$ that happens to be a dilator.
\end{enumerate}

Despite faults on $\Sigma^1_2$-proof-theoretic ordinal, it has the following positive sides:
\begin{enumerate}
    \item Computing the $\Sigma^1_2$-proof-theoretic ordinal can rely on a softer model-theoretic argument and gives an ordinal scale for gauging the strength of theories.
    
    \item The currently known value for $s^1_2(T)$ seems somehow tied with $|T|_{\Pi^1_1}$: For example, for theories at the level of iterated inductive definition, $|T|_{\Pi^1_1}$ takes the form of the collapse of $D(s^1_2(T))$ for some dilator $D$. 
    Also, Towsner pointed out that the current ordinal analysis for $\Pi^1_2\mhyphen\CA_0$ relies on stable ordinals, which also have a critical role in computing $s^1_2(\Pi^1_2\mhyphen\CA_0)$.
    It might be possible that $s^1_2(T)$ gives some hint on computing $|T|_{\Pi^1_1}$ for a complicated $T$. 
    
    \item The $\Sigma^1_2$-proof-theoretic ordinal for a set theory looks closely tied with the transitive models of the theory. For example, both $s^1_2(\KP)$ and $s^1_2(\KP\ell)$\footnote{$\KP\ell$ is the theory comprising Primitive Recursive set theory with `every set is contained in an admissible set.' $L_{\omega_\omega^\CK}$ is a model of $\KP\ell$.} are the least height of a transitive model of the theory. 
\end{enumerate}

Computing $s^1_2(T)$ for specific theories, like, the theory of $(<\alpha)$-fold iterated inductive definition $\mathsf{ID}_{<\alpha}$ would be an interesting problem:
\begin{question}
    What is the value of $s^1_2(T)$ for a natural $\Sigma^1_2$-sound $T$? For example, what is $s^1_2(\Pi^1_n\mhyphen\CA_0)$? Also, do we have $s^1_2(\mathsf{ID}_{<\alpha})=\omega_{1+\alpha}^\CK$ for reasonably small $\alpha$?
\end{question}

As pointed out before, the currently known value of $s^1_2(T)$ for \added{a} specific $T$ is closely related to the least height of a transitive model of $T$. \added{A proof of the following theorem will appear in \cite{JeonPhD} and a forthcoming paper:}
\begin{theorem}
    Let $T$ be a $\Sigma^1_2$-sound extension of $\Pi^1_1\mhyphen\CA_0$, and $\Sigma^1_2(T)$ be the set of all $\Sigma^1_2$-consequences of $T$. Then $\ATR_0^\mathsf{set} + \Sigma^1_2(T)$ has a transitive model, and moreover
    \begin{equation*}
        s^1_2(T) = \min\{N\cap \Ord\mid \text{$N$ is a transitive model of }\ATR_0^\mathsf{set} + \Sigma^1_2(T)\}.
    \end{equation*}
\end{theorem}
The general connection between the least height of a transitive model of $T$ and $s^1_2(T)$ is unclear, which should be a future research topic.

The proof for \autoref{Proposition: Sigma 1 2 modulo Pi 1 2 reflection rank wellfoundedness} is based on a connection between $\Sigma^1_2$-proof-theoretic ordinal $s^1_2(T)$ and $\Sigma^1_2$-reflection principle for extensions of $\Sigma^1_2\mhyphen\AC_0$.
The argument breaks down even at the level of $\Pi^1_1\mhyphen\CA_0$ as we examined in \autoref{Example: Desired Second equivalence LtoR fails}.
However, it does not rule out the possibility of well-foundedness of $\prec^{\Pi^1_2}_{\Sigma^1_2}$ for extensions of $\ACA_0$:
\begin{question}
    Can we prove $\prec^{\Pi^1_2}_{\Sigma^1_2}$ is well-founded for $\Sigma^1_2$-sound r.e.\ extensions of $\ACA_0$?
\end{question}

Although we `justified' \autoref{Phenomenon: Sigma 1 2 observation} throughout this paper, it still does not address Steel's full observation \autoref{Phenomenon: SteelObservation}. In the next sequel \cite{Jeon??HigherProofTheoryII}, we discuss evidence for Steel's observation for $\Pi^1_3$ and $\Sigma^1_4$ for theories extending $\ACA_0$ plus $\mathbf{\Delta}^1_2$-Determinacy.

\appendix

\section{Basic theory of dilators} \label{Section: Dilators}
\added{In this section, we briefly review definitions and basic facts about dilators. Every result in this section is not new, albeit in different formulations. See \cite{JeonPhD} for technical details.}

\subsection{Denotation system}
Girard introduced dilators to provide proof-theoretic analysis at the level of $\Pi^1_2$. 
One way to understand a dilator is to understand it as a \emph{denotation system} for a class ordinal. Let us give a simple example to illustrate what dilators are: Let us construct a class ordinal $\Ord+\Ord$. There is no way to express $\Ord+\Ord$ as a transitive class, but there is a way to express its ordertype: We can understand $\Ord+\Ord$ as a collection of pairs of the form $(\xi,0)$ or $(\xi,1)$ for $\xi\in \Ord$, and compare them as follows:
\begin{itemize}
    \item $(\xi,0) < (\eta,0)$ iff $\xi<\eta$. The same holds for $(\xi,1)$ and $(\eta,1)$.
    \item $(\xi,0)< (\eta,1)$ always hold.
\end{itemize}
Moreover, the denotation system is \emph{uniform} in the sense that the same construction over $\alpha$ instead of $\Ord$ gives an ordertype of $\alpha+\alpha$. It turns out that handling the denotation system for class ordinals directly is more fruitful than considering their ordertypes only, which are named by \emph{dilators}. In terms of dilators, the expression for $\Ord+\Ord$ corresponds to the dilator $\mathsf{Id}+\mathsf{Id}$, where $\mathsf{Id}$ is the identity dilator. 
However, the denotation system may result in an ill-founded order, so we define \emph{predilators} first to embrace ill-founded cases, and we define dilators as predilators with certain conditions.

Informally, a denotation system $D$ is a set of terms with a comparison rule between terms. Each term $t$ comes with an \emph{arity} $n=\arity(t)\in \bbN$, and we understand $t$ as a `function' taking an increasing sequence of length $n$, or alternatively, a subset of size $n$. We will form a new linear order $D(\alpha)$ from a given linear order $\alpha$, which takes the form
\begin{equation*}
    D(\alpha) = \{t(a) \mid \text{$t$ is a $D$-term and }a \in [\alpha]^{<\omega}\}.
\end{equation*}
We assume that $s(a)\neq t(b)$ if $s\neq t$. Comparing $s(a)$ and $t(b)$ only depends on the \added{isomorphism type of $(a\cup b;<,a,b)$.}

% \added{In this section, we provide a rigorous formulation of predilators in a denotation system-styled manner. There are alternative formulations for dilators, like that of Girard's \cite{Girard1982Logical}, or Freund's (e.g., \cite{FreundPhD, Freund2021Pi14}). We will see that Freund's formulation of dilators is `equivalent' to the denotation system in the sense of category equivalence. We will not discuss Girard's dilator in this paper, but Freund \cite[Remark 2.2.2]{FreundPhD} proved that Girard's dilator is equivalent to Freund's.}

\begin{definition}
    An \emph{arity diagram} $\cyrDe$ is a commutative diagram over the category $\LO_{<\omega}$ of natural numbers with strictly increasing maps of the form
    \begin{equation} \label{Formula: Arity diagram definition diagram}
    \cyrDe = 
        \begin{tikzcd}
            a_\cap & a_1 \\
            a_0 & a_\cup 
            \arrow[from=1-1, to=1-2]
            \arrow[from=1-1, to=2-1]
            \arrow[from=2-1, to=2-2, "g", swap]
            \arrow[from=1-2, to=2-2, "f"]
        \end{tikzcd}
    \end{equation}
    such that the above diagram is a pullback and $\ran f\cup \ran g = \field a_\cup$.
    We say an arity diagram is \emph{trivial} if $a_\cap = a_0=a_1=a_\cup$.
    For an arity diagram $\cyrDe$ of the form \eqref{Formula: Arity diagram definition diagram}, the diagram $-\cyrDe$ is a diagram obtained by switching the order of $a_0$ and $a_1$:
    \begin{equation*}
    -\cyrDe = 
        \begin{tikzcd}
            a_\cap & a_0 \\
            a_1 & a_\cup 
            \arrow[from=1-1, to=1-2]
            \arrow[from=1-1, to=2-1]
            \arrow[from=2-1, to=2-2, "f", swap]
            \arrow[from=1-2, to=2-2, "g"]
        \end{tikzcd}
    \end{equation*}
\end{definition}

An arity diagram is a way to describe \added{the} relative order between elements of $a_0$ and that of $a_1$. 
A typical example of an arity diagram is induced from an inclusion diagram, for example,
\begin{equation*}
    \begin{tikzcd}
        \{1,3\} & \{0,1,3\} \\
        \{1,2,3,4\} & \{0,1,2,3,4\} 
        \arrow[from=1-1, to=1-2, "\subseteq"]
        \arrow[from=1-1, to=2-1, "\subseteq", swap]
        \arrow[from=2-1, to=2-2, "\subseteq", swap]
        \arrow[from=1-2, to=2-2, "\subseteq"]
    \end{tikzcd}
\end{equation*}
We can see that the above diagram is isomorphic to
\begin{equation*}
    \begin{tikzcd}
        \{\mathbf{0},\mathbf{1}\} & \{0,\mathbf{1},\mathbf{2}\} \\
        \{\mathbf{0},1,\mathbf{2},3\} & \{0,1,2,3,4\} 
        \arrow[from=1-1, to=1-2, "h"]
        \arrow[from=1-1, to=2-1, "k", swap]
        \arrow[from=2-1, to=2-2, "f", swap]
        \arrow[from=1-2, to=2-2, "g"]
    \end{tikzcd}
\end{equation*}
where $h$ and $k$ are maps sending boldface numbers to boldface numbers in an increasing manner, $f(n)=n+1$, and $g(0)=0$, $g(1)=1$, $g(2)=3$.
The following notion will be used to define the linearity condition of a denotation system:
\begin{definition}
    Let $a_0,\cdots, a_{m-1}$ be natural numbers and let $\ttL_m=(\ttL_m,\land,\lor)$ be the free distributive lattice generated by $\{0,1,\cdots,m-1\}$. An \emph{IU diagram $\bfcyrDe$ for $a_0,\cdots, a_{m-1}$} (abbreviation of \emph{Intersection-Union diagram}) is a collection of objects $\{a_i\mid i\in \ttL_m\}$ and strictly increasing functions $f_{ij}\colon a_i\to a_j$ for $i\le j$ such that the following holds:
    \begin{enumerate}
        \item $f_{ii}$ is the identity map.
        \item For $i\le j\le k$, $f_{ik}=f_{jk}\circ f_{ij}$.
        \item For each $i,j\in \ttL_m$, the following diagram is an arity diagram:
        \begin{equation*}
            \bfcyrDe(i,j) = 
            \begin{tikzcd}
            a_{i\land j} & a_j \\
            a_i & a_{i\lor j}
            \arrow[from=1-1, to=1-2, "f_{i\land j, j}"]
            \arrow[from=1-1, to=2-1, "f_{i\land j, i}", swap]
            \arrow[from=2-1, to=2-2, "f_{i, i\lor j}", swap]
            \arrow[from=1-2, to=2-2, "f_{j, i\lor j}"]
            \end{tikzcd}
        \end{equation*}
    \end{enumerate}
    Alternatively, we can think of an IU diagram as a functor from $\ttL_n$ as a category to the category of natural numbers with strictly increasing maps.
\end{definition}

Now, let us define denotation systems, which we will call semidilators.
\begin{definition}
    A \emph{denotation system} or \emph{semidilator} $D$ is a set of \emph{$D$-terms} usually denoted by $t$, and each term $t$ comes with an \emph{arity} $\arity(t)\in \bbN$.
    Also, for each two $D$-terms $t_0$, $t_1$ and an arity diagram
    \begin{equation*}
        \cyrDe = 
        \begin{tikzcd}
            a_\cap & a_0 \\
            a_1 & a_\cup 
            \arrow[from=1-1, to=1-2]
            \arrow[from=1-1, to=2-1]
            \arrow[from=2-1, to=2-2]
            \arrow[from=1-2, to=2-2]
        \end{tikzcd}
    \end{equation*}
    such that $a_i=\arity(t_i)$, we have the relation $t_0<_\cyrDe t_1$ between two terms. We call $\cyrDe$ an \emph{arity diagram for $t_0$ and $t_1$}. Then $D$ satisfies the following:
    \begin{enumerate}
        \item (Irreflexivity) If $t_0=t_1$ and $\cyrDe$ is trivial, then $t_0<_\cyrDe t_0$ does not hold.
        \item (Linearity) If $t_0\neq t_1$ or $\cyrDe$ is not trivial, then one of $t_0<_\cyrDe t_1$ or $t_1<_{-\cyrDe} t_0$ must hold.
        \item (Transitivity) For three $D$-terms $t_0,t_1,t_2$ such that $\arity(t_i) = a_i$, and an IU diagram $\bfcyrDe$ for $t_0,t_1,t_2$, if $t_0<_{\bfcyrDe(0,1)} t_1$ and $t_1 <_{\bfcyrDe(1,2)} t_2$, then $t_0 <_{\bfcyrDe(0,2)} t_2$.
    \end{enumerate}
    \added{We denote the underlying set of $D$ (or, the set of $D$-terms) by $\field D$, following the field of linear orders.}

    We say $D$ is \emph{(primitive) recursive} or \emph{countable} if the set of $D$-terms and the set of comparison rules are (primitive) recursive or countable, respectively.
\end{definition}

The reader should be warned that materials call a semidilator in different ways: Aguilera-Pakhomov \cite{AguileraPakhomov2023Pi12} call it predilator, but it overlaps with the standard definition of a predilator we will introduce later. Girard \cite{Girard1982Logical} call semidilator \emph{paleodilator}, Freund \cite{FreundPhD} call it \emph{prae-dilator}, and Catlow \cite{Catlow1994} call it \emph{1-functor}. 

\subsection{Category-theoretic aspects of semidilators}
We want to understand semidilators as a functor over the category of linear orders $\LO$ with strictly increasing functions as morphisms. Hence let us define the application of semidilators:
\begin{definition} \label{Definition: Application of a predilator}
     From a given semidilator $D$ and a linear order $X$, let us define a new structure $D(X)$ by
    \begin{equation}
        D(X) = \{(t,a)\mid t\in \field D\land a\in [X]^{\arity t}\}.
    \end{equation}
    We write $t(a)$ instead of $(t,a)$, and we identify $a\subseteq X$ with a finite increasing sequence over $X$. The order of $D(X)$ is given by
    \begin{equation*}
        s(a) <_{D(X)} t(b) \iff s <_\cyrDe t
    \end{equation*}
    where $\cyrDe=\Diag_X(a,b)$ is the \emph{induced diagram from $a$ and $b$}, that is the unique arity diagram isomorphic to the inclusion diagram
    \begin{equation*}
        \begin{tikzcd}
            a\cap b & b \\
            a & a \cup b
            \arrow[from=1-1, to=1-2, "\subseteq"]
            \arrow[from=1-1, to=2-1, "\subseteq", swap]
            \arrow[from=2-1, to=2-2, "\subseteq", swap]
            \arrow[from=1-2, to=2-2, "\subseteq"]
        \end{tikzcd}
    \end{equation*}
    More precisely, $\cyrDe=\Diag(a,b)$ is the innermost diagram in the below commutative diagram, where $\en_a\colon |a|\to a$ is the unique order isomorphism for finite linear order $a$.
    \begin{equation*}
        \begin{tikzcd}
            a\cap b &&& b \\
            & {|a\cap b|} & {|b|} & \\
            & {|a|} & {|a\cup b|} & \\
            a &&& a\cup b
            \arrow[from=1-1, to=1-4, "\subseteq"]
            \arrow[from=1-1, to=4-1, "\subseteq", swap]
            \arrow[from=1-4, to=4-4, "\subseteq"]
            \arrow[from=4-1, to=4-4, "\subseteq", swap]
            \arrow[from=2-2, to=2-3]
            \arrow[from=2-2, to=3-2]
            \arrow[from=3-2, to=3-3, "e_0", swap]
            \arrow[from=2-3, to=3-3, "e_1"]
            \arrow[from=2-2, to=1-1, "\en_{a\cap b}" , "\cong"']
            \arrow[from=2-3, to=1-4, "\cong", "\en_{a}"']
            \arrow[from=3-2, to=4-1, "\cong", "\en_{b}"']
            \arrow[from=3-3, to=4-4, "\en_{a\cup b}", "\cong"']
        \end{tikzcd}
    \end{equation*}
    
    For a strictly increasing function $f\colon X\to Y$, consider the map $D(f)\colon D(X)\to D(Y)$ given by
    \begin{equation*}
        D(f)(t,a) = (t,f[a]),
    \end{equation*}
    where $f[a] = \{f(x) \mid x\in a\}$.
\end{definition}

We can see the following facts without difficulty:
\begin{proposition}[$\ACA_0$] \label{Proposition: Predilator is a functor}
    Let $D$ be a semidilator and $X$ be a linear order. Then $D(X)$ is a linear order. Also, $D$ induces a functor $\LO\to\LO$.
\end{proposition}
\iffalse
\begin{proof}
    Let us prove that $D(X)$ is a linear order. The irreflexivity and the linearity of a semidilator imply that $<_{D(X)}$ satisfies the irreflexivity and the linearity, respectively.
    To see $<_{D(X)}$ is transitive, suppose that we have $s_0(a_0)<_{D(X)}s_1(a_1)<_{D(X)}s_2(a_2)$.
    Now let $\bfcyrDe$ be the IU diagram isomorphic to a diagram induced by intersections and unions of $a_0,a_1,a_2$, and the insertion maps. Especially, $\bfcyrDe(i,j)$ is the induced diagram from $a_i$ and $a_j$ for $i\neq j<2$. Then we have
    \begin{equation*}
        s_0 <_{\bfcyrDe(0,1)} s_1 \land s_1 <_{\bfcyrDe(1,2)} s_2.
    \end{equation*}
    Thus, we have $s_0 <_{\bfcyrDe(0,2)} s_2$ by the transitivity of a semidilator.
    To seethat  $D$ induces a functor, it suffices to show that the following holds:
    \begin{enumerate}
        \item $D(\mathsf{Id}_X) = \mathsf{Id}_{D(X)}$,
        \item $D(f\circ g) = D(f)\circ D(g)$.
    \end{enumerate}
    Both of them are easy to verify, so we leave their proof to the reader.
\end{proof}
\fi

Like strictly increasing maps between linear orders, we can define a map between semidilators preserving their structures:
\begin{definition}
    Let $D$, $E$ be two semidilators. A map $\iota\colon D\to E$ is an \emph{embedding} or \added{a \emph{semidilator morphism}} if it satisfies the following conditions:
    \begin{enumerate}
        \item $\iota$ is a function \added{from} $\field D$ to $\field E$.
        \item $\iota$ preserves the arity: i.e., $\arity(\iota(t)) = \arity(t)$ for every $t\in \field D$,
        \item For each two terms $t_0,t_1\in \field D$ and an arity diagram $\cyrDe$ between them, we have $t_0<_\cyrDe t_1$ iff $\iota(t_0) <_\cyrDe \iota(t_1)$.
    \end{enumerate}
    An embedding $\iota$ is an \emph{isomorphism} if $\iota\colon \field D\to  \field E$ is a bijection and $\iota^{-1}$ is also an embedding. We denote $D\cong E$ if there is an isomorphism between $D$ and $E$. 
    We also write $D\le E$ if there is an embedding from $D$ to $E$.
\end{definition}

We did not require an embedding between two semidilators to be \added{injective}, but \added{it can be easily shown that an embedding is injective. Moreover, the inverse of a bijective embedding is also an embedding.}

We can view an embedding between two semidilators as a natural transformation in the following manner:
\begin{definition}
    For an embedding $\iota\colon D\to E$ and a linear order $X$, define $\iota_X\colon D(X)\to E(X)$ by
    \begin{equation*}
        \iota_X(t,a) = (\iota(t),a).
    \end{equation*}
\end{definition}

\begin{proposition}[$\ACA_0$]
    Let $\iota \colon D\to E$ be an embedding between two predilators $D$ and $E$.
    Then $\iota_X\colon D(X)\to E(X)$ is a strictly increasing map. Furthermore, if $f\colon X\to Y$ is a strictly increasing map between two linear orders $X$ and $Y$, then the following diagram commutes:
    \begin{equation*}
        \begin{tikzcd}
            D(X) & D(Y) \\
            E(X) & E(Y)
            \arrow[from=1-1, to=1-2, "D(f)"]
            \arrow[from=1-1, to=2-1, "\iota_X", swap]
            \arrow[from=2-1, to=2-2, "E(f)", swap]
            \arrow[from=1-2, to=2-2, "\iota_Y"]
        \end{tikzcd}
    \end{equation*}
\end{proposition}
\iffalse
\begin{proof}
    For $\iota_X$ being strictly increasing, suppose that we have $s(a) <_{D(X)} t(b)$, and let $\cyrDe$ be the induced diagram from $a$ and $b$. Then we have $s <_\cyrDe t$, so $\iota(s) <_\cyrDe \iota(t)$. Since $\arity s = \arity \iota(s)$ and $\arity t=\arity \iota(t)$, $(\iota(s),a), (\iota(t),b)$ are in $E(X)$ and moreover we have $\iota(s)(a) <_{E(X)} \iota(t)(b)$, which is $\iota_X(s(a)) <_{E(X)} \iota_X(t(b))$.
    The commutativity is easy to check.
\end{proof}
\fi 

\subsection{Predilators and Dilators}
\begin{definition}
    A semidilator $D$ is a \emph{predilator} if it satisfies the \emph{monotonicity condition}:
    For every $t\in\field(D)$ and a non-trivial arity diagram 
    \begin{equation} \label{Formula: Arity diagram with the same arities}
        \cyrDe = 
        \begin{tikzcd}
            n_\cap & n \\
            n & n_\cup 
            \arrow[from=1-1, to=1-2]
            \arrow[from=1-1, to=2-1]
            \arrow[from=2-1, to=2-2, "e_0", swap]
            \arrow[from=1-2, to=2-2, "e_1"]
        \end{tikzcd}
    \end{equation}
    if $e_0(i)\le e_1(i)$ for each $i<n$, then $t <_\cyrDe t$ holds.\footnote{By the non-triviality of $\cyrDe$, there must be $i<n$ such that $e_0(i)<e_0(i)$.}
\end{definition}

\begin{definition}
    A semidilator $D$ is a \emph{dilator} if for every well-order $\alpha$, $D(\alpha)$ is also well-ordered.
\end{definition}

Why do we introduce predilators? One reason is that every semidilator arising in our context is a predilator. Another reason is that finite predilators will turn out to be \emph{dilators}, but not every finite semidilator is a dilator. We expect the role between predilators and dilators to be similar to that between linear orders and well-orders, so the current notion of predilators is a more natural choice than semidilators.

Girard's original formulation of a predilator is more category-theoretic. The next lemma says Girard's formulation is `equivalent' to our formulation. For two strictly increasing maps $f, g\colon X\to Y$, let us say $f\le g$ if $f(x)\le g(x)$ for every $x\in X$.
\begin{lemma}[$\ACA_0$] \label{Lemma: Monotone semidilator iff predilator}
    A semidilator $D$ is a predilator iff for every linear order $X$, $Y$ and strictly increasing maps $f,g\colon X\to Y$, if $f\le g$ then $D(f)\le D(g)$.
\end{lemma}

It turns out that checking $D$ being a dilator only requires checking the well-orderedness of $D(\alpha)$ for countable $\alpha$:
\begin{lemma} \label{Lemma: Being a dilator is captured by countable WOs}
    Let $D$ be a predilator. Then $D$ is dilator iff for every countable well-order $\alpha$, $D(\alpha)$ is well-ordered.
\end{lemma}

We can see that every predilator is a dilator. \added{Although its proof appears in prior materials, e.g., \cite[Proposition 2.3.7]{Girard1981Dilators}, we present a different proof since the following proof is used in \autoref{Proposition: Climax is a semidilator is trivial}.\footnote{We also remark that Girard's proof only tells that if $D$ is a semidilator that is not a predilator, then $\Clim(D)<\omega^\omega$.}}
\begin{proposition}[$\ACA_0$] \label{Proposition: Dilator is a predilator}
    If $D$ is a dilator, then $D$ is a predilator.
\end{proposition}
\begin{proof}
    Suppose that $D$ is a semidilator that is not a predilator.
    Then we have a non-trivial monotone arity diagram $\cyrDe$ of the form \eqref{Formula: Arity diagram with the same arities} and a $D$-term $t$ of arity $n$ such that $D\vDash t\nless_\cyrDe t$.
    We begin the proof by decomposing $\cyrDe$ into other monotone arity diagrams.

    Let $\{k_0<\cdots<k_{N-1}\}\subseteq n_\cup$ be set of natural numbers such that $e_0(k)<e_1(k)$. For each $i\le N$, define
    \begin{equation*}
        f_i(j) = 
        \begin{cases}
            e_0(j), & \text{If }j\le k_{N-1-i}\\
            e_1(j), & \text{If }j>k_{N-1-i}.
        \end{cases}
    \end{equation*}
    (For technical convenience let $k_{-1}=-1$.)
    Then $f_0=e_0$, $f_N = e_1$, and $f_0\le f_1\le\cdots\le f_N$.
    In fact, $f_i$ and $f_{i+1}$ differs at just one point, namely $k_{N-1-i}$.
    Hence the induced diagram $\Diag_{n_\cup}(\ran f_i,\ran f_{i+1})$ is also a non-trivial monotone diagram.

    Now define the IU diagram $\bfcyrDe$ for $N+1$ copies of $n$ over $\bbN$ by the unique diagram isomorphic to the inclusion diagram generated by the intersection and union of $\ran f_i$ for $i=0,1,\cdots, N$.
    That said, we define an inclusion diagram $\bfcyrDe'\colon \ttL_{N+1}\to \mathcal{P}(n_\cup)$ by $\bfcyrDe'(i)=\ran f_i$, $\bfcyrDe'(i\land j)=\bfcyrDe'(i)\cap \bfcyrDe'(j)$, and $\bfcyrDe'(i\lor j)=\bfcyrDe'(i)\cup \bfcyrDe'(j)$, and $\bfcyrDe$ is the unique IU diagram over $\bbN$ isomorphic to $\bfcyrDe'$.
    Then $\bfcyrDe[i,i+1]=\Diag_{n_\cup}(\ran f_i,\ran f_{i+1})$ is monotone and non-trivial, and $\bfcyrDe[0,N]=\Diag_{n_\cup}(\ran e_0,\ran e_1) = \cyrDe$.
    Hence $D\vDash t\nless_\cyrDe t$ implies there is $i<N$ such that $D\vDash t\nless_{\bfcyrDe[i,i+1]} t$, and let $\cyrZhe = \bfcyrDe[i,i+1]$.
    Since $f_i$ and $f_{i+1}$ differ only at one point, $\cyrZhe$ must take the form
    \begin{equation*}
        \cyrZhe = 
        \begin{tikzcd}
            n-1 & n \\
            n & n+1 
            \arrow[from=1-1, to=1-2]
            \arrow[from=1-1, to=2-1]
            \arrow[from=2-1, to=2-2, "g_0", swap]
            \arrow[from=1-2, to=2-2, "g_1"]
        \end{tikzcd}
    \end{equation*}
    and there is a unique $k<n$ such that $g_0(k)<g_1(k)$.

    Now we claim that $D(\omega+(n-k))$ is ill-founded. For each $j\in\bbN$, define $a_i = \{0,\cdots, k-1, k+i,\omega,\cdots,\omega+n-k-1\}$.
    Then $|a_i|=n$ and $\Diag(a_i,a_j)=\cyrZhe$ for every $i<j\in\bbN$.
    Hence $D(\omega+n-k)\vDash t(a_i)>t(a_j)$ for $i<j$, so $D(\omega+n-k)$ has an infinite decreasing sequence.
\end{proof}

For finite semidilators, being a predilator is equivalent to being a dilator. That is, if $D$ is a \emph{finite}\footnote{Or \emph{strongly finite} in Girard's terminology.} predilator, being monotone is sufficient for being a dilator. We start with the following lemma:
\begin{lemma} \label{Lemma: Finitely many arity diagrams}
    For $n\in\bbN$, there are finitely many arity diagrams between $n$ and $n$.
\end{lemma}

\begin{proposition}[$\ACA_0$] \label{Proposition: Finite predilator is a dilator}
    Every finite predilator is a dilator.
\end{proposition}
\begin{proof}
    Suppose that $D$ is a finite semidilator that is not a dilator. We claim that $D$ is not a predilator. Fix a countable well-order $\alpha$ such that $D(\alpha)$ is ill-founded, so we have a sequence of $D$-terms $\lag t_i\mid i\in\bbN\rag$ and finite subsets $\lag a_i \mid i\in\bbN\rag$ of $\alpha$ such that 
    \begin{equation*}
        D(\alpha) \vDash t_0(a_0) > t_1(a_1) > t_2(a_2) > \cdots.
    \end{equation*}
    By the Pigeonhole principle, we may assume that all of $t_i$ are equal to $t$ whose arity is $n$. Now consider the coloring $F$ of domain $[\bbN]^2$ defined by
    \begin{equation*}
        F(i,j) = \text{The arity diagram induced from $a_i$ and $a_j$}.
    \end{equation*}
    Then the range of $F$ is finite by \autoref{Lemma: Finitely many arity diagrams}. Hence by Ramsey's theorem for pairs, we have an infinite homogeneous set $H\subseteq \bbN$ for $F$ with value $\cyrDe$. 
    
    For $i,j\in H$, $t(a_i) >_{D(\alpha)} t(a_j)$ implies $t \nless_\cyrDe t$. $\cyrDe$ is not trivial since otherwise we have $a_i=a_j$ for $i,j\in H$, so $t(a_i)=t(a_j)$.
    Now let us claim that $\cyrDe$ witnesses $D$ is not a predilator: 
    Suppose that $\cyrDe$ has the form \eqref{Formula: Arity diagram with the same arities}. We claim that $e_0\le e_1$ holds. Suppose not, let $k<n$ be a natural number such that $e_0(k)>e_1(k)$. Recall that we have the unique isomorphism $\en_{a_i}\colon n\to a_i$. Since $\cyrDe$ is the induced diagram from $a_i$ and $a_j$, we have a commutative diagram
    \begin{equation*}
        \begin{tikzcd}[column sep=large]
            a_i\cap a_j &&& a_j \\
            & n_\cap & n & \\
            & n & n_\cup & \\
            a_i &&& a_i\cup a_j
            \arrow[from=1-1, to=1-4, "\subseteq"]
            \arrow[from=1-1, to=4-1, "\subseteq", swap]
            \arrow[from=1-4, to=4-4, "\subseteq"]
            \arrow[from=4-1, to=4-4, "\subseteq", swap]
            \arrow[from=2-2, to=2-3]
            \arrow[from=2-2, to=3-2]
            \arrow[from=3-2, to=3-3, "e_0", swap]
            \arrow[from=2-3, to=3-3, "e_1"]
            \arrow[from=2-2, to=1-1, "\en_{a_i\cap a_j}" , "\cong"']
            \arrow[from=2-3, to=1-4, "\cong", "\en_{a_j}"']
            \arrow[from=3-2, to=4-1, "\cong", "\en_{a_i}"']
            \arrow[from=3-3, to=4-4, "\en_{a_i\cup a_j}", "\cong"']
        \end{tikzcd}
    \end{equation*}
    If we denote the $k$th element of $a_i$ by $a_i(k)$, then $e_0(k)>e_1(k)$ implies $a_i(k) > a_j(k)$.
    It leads to a contradiction since $H$ is infinite, so we get an infinite descending sequence over $\alpha$.

    Thus, we have an arity diagram $\cyrDe$ of the form \eqref{Formula: Arity diagram with the same arities} such that $e_0\le e_1$, but also have $t\in \field(D)$ of arity $n$ such that $D\vDash t\nless_\cyrDe t$. Hence $D$ is not a predilator.
\end{proof}

\subsection{Dilators as functors with support transformations}
There are different ways to define predilators and dilators.
Girard \cite{Girard1981Dilators} defined dilators as autofunctors $F\colon \WO\to \WO$ over the category \added{of} well-orders preserving direct limits and pullbacks. 
Freund \added{(e.g., \cite{FreundPhD, Freund2021Pi14})} defined dilators and predilators as functors $F\colon \LO\to\LO$ admitting the support transformation $\supp \colon F\to [\cdot]^{<\omega}$, a transformation taking an element $t(a)\in F(X)$ and returning the element part $a$. Freund also proved \cite[Remark 2.2.2]{FreundPhD} that his definitions of a predilator and a dilator are equivalent to Girard's ones.
\added{Freund's dilators are useful in practice, so we introduce how Freund defined dilators, and its equivalence with a denotation system.}

Let $\LO$ be the category of all linear orders with strictly increasing functions. We say a functor $F\colon \LO\to\LO$ is \emph{monotone} if $f\le g$ implies $D(f)\le D(g)$.
For a given set $X$, $[X]^{<\omega}$ is the set of all finite subsets of $X$. We can turn $[\cdot]^{<\omega}$ to a functor by defining $[f]^{<\omega}(a)=\{f(x)\mid x\in a\}$ for $f\colon X\to Y$.

\begin{definition}
    An \emph{F-semidilator} (F abbreviates `Freund') is a pair $(F,\supp)$, where $F\colon \LO\to\LO$ is a functor and $\supp^F\colon F\to [\cdot]^{<\omega}$ is a natural transformation satisfying the \emph{support condition}
    \begin{equation} \label{Formula: Support condition}
        \{\sigma\in F(Y) \mid \supp^F_Y(\sigma)\subseteq \ran(f)\} \subseteq \ran(F(f))
    \end{equation}
    for every morphism $f\colon X\to Y$. We often omit $\supp^F$ and simply write $F$ to mean $(F,\supp)$.
    If $F$ is monotone, then we call $F$ an \emph{F-predilator}.
\end{definition}

We can extract a denotation system from Fruend-style semidilator as follows:
\begin{definition}
    The \emph{trace} of an F-semidilator $F$ is given by
    \begin{equation*}
        \Tr(F) = \{(n,\sigma) \mid \sigma\in F(n)\land \supp^F_n(\sigma)=n\}.
    \end{equation*}
\end{definition}
Roughly, $\Tr(D)$ will be the set of all elements of the form $t(\{0,1,\cdots, n-1\})$ for a $F$-term $t$ of arity $n$. 

\begin{proposition} \label{Proposition: F-semidilator to a denotation system}
    For an F-semidilator $F$, if we define $\frakf(F)$ by $\field(\frakf(F))=\Tr(F)$, $\arity^{\frakf(F)}(n,t)=n$, and for an arity diagram
    \begin{equation*}
    \cyrDe = 
        \begin{tikzcd}
            n_\cap & n_1 \\
            n_0 & n_\cup 
            \arrow[from=1-1, to=1-2]
            \arrow[from=1-1, to=2-1]
            \arrow[from=2-1, to=2-2, "e_0", swap]
            \arrow[from=1-2, to=2-2, "e_1"]
        \end{tikzcd}
    \end{equation*}
    and $(n_0,t_0),(n_1,t_1)\in \Tr(F)$, let us define
    \begin{equation*}
        (n_0,t_0) <_\cyrDe (n_1,t_1) \iff F(n_\cup)\vDash F(e_0)(t_0) < F(e_1)(t_1).
    \end{equation*}
    Then $\frakf(F)$ is a semidilator.
\end{proposition}

Conversely, an application of a denotation system turns a denotation system into a Freund-style semidilator:
\begin{theorem}
    Let $D$ be a semidilator. Let us define $\fraka(D)(X)=D(X)$ for a linear order $X$. For increasing $f\colon X\to Y$, let us define $\fraka(D)(f)\colon \fraka(D)(X)\to\fraka(D)(Y)$ by 
    \begin{equation*}
        \fraka(D)(f)(t(a)) = t(f[a]).
    \end{equation*}
    Then $\fraka(D)$ is an F-semidilator.
\end{theorem}

Now let $\SDil$ and $\SDil_\mathsf{F}$ be the category of denotation systems and the category of F-semidilators respectively. We can extend $\frakf$ and $\fraka$ into functors $\frakf\colon \SDil_\mathsf{F}\to\SDil$ and $\fraka\colon \SDil\to\SDil_\mathsf{F}$. The following theorem says denotation systems and Freund-style semidilators are `equivalent' in the sense of category equivalence:
\begin{theorem}
    $\fraka\circ \frakf$ and $\frakf\circ \fraka$ are naturally isomorphic to the identity functor over the domain category. Hence, $\fraka$ and $\frakf$ form a category equivalence.
\end{theorem}

\printbibliography

@book {Barwise1975,
    AUTHOR = {Barwise, Jon},
     TITLE = {Admissible sets and structures},
    SERIES = {Perspectives in Mathematical Logic},
      NOTE = {An approach to definability theory},
 PUBLISHER = {Springer-Verlag, Berlin-New York},
      YEAR = {1975}
}

@book{Simpson2009,
	title={Subsystems of second order arithmetic},
	author={Simpson, Stephen G.},
	volume={1},
	year={2009},
	publisher={Cambridge University Press}
}

@article{PakhomovWalsh2021Reflection,
    title={Reflection ranks and Ordinal analysis},
    volume={86},
    number={4},
    journal={The Journal of Symbolic Logic},
    publisher={Cambridge University Press},
    author={Pakhomov, Fedor and Walsh, James},
    year={2021},
    pages={1350–1384}
}

@article {AguileraPakhomov2023Pi12,
    AUTHOR = {Aguilera, J. P. and Pakhomov, F.},
     TITLE = {The {$\Pi^1_2$} consequences of a theory},
   JOURNAL = {J. Lond. Math. Soc. (2)},
  FJOURNAL = {Journal of the London Mathematical Society. Second Series},
    VOLUME = {107},
      YEAR = {2023},
    NUMBER = {3},
     PAGES = {1045--1073}
}

@unpublished{Girard1982Logical,
	title={Proof theory and logical complexity {II}},
	author={Girard, Jean-Yves},
	note={Unpublished manuscript},
	url={https://girard.perso.math.cnrs.fr/ptlc2.pdf}
}

@article {Catlow1994,
    AUTHOR = {Catlow, J. R. G.},
     TITLE = {A proof-theoretical analysis of ptykes},
   JOURNAL = {Arch. Math. Logic},
  FJOURNAL = {Archive for Mathematical Logic},
    VOLUME = {33},
      YEAR = {1994},
    NUMBER = {1},
     PAGES = {57--79}
}

@phdthesis{FreundPhD,
	title={Type-{T}wo {W}ell-Ordering {P}rinciples, {A}dmissible {S}ets, and $\Pi^1_1$-{C}omprehension},
	author={Freund, Anton Jonathan},
	year={2018},
	school={University of Leeds},
	type={{PhD} thesis},
}

@article{Walsh2023characterizations,
  title={Characterizations of ordinal analysis},
  author={Walsh, James},
  journal={Annals of Pure and Applied Logic},
  volume={174},
  number={4},
  pages={103230},
  year={2023},
  publisher={Elsevier}
}

@article{Walsh2022incompleteness,
  title={An incompleteness theorem via ordinal analysis},
  author={Walsh, James},
  journal={The Journal of Symbolic Logic},
  pages={1--17},
  year={2022},
  publisher={Cambridge University Press}
}

@article{PakhomovWalsh2023omegaReflection,
author = {Pakhomov, Fedor and Walsh, James},
title = {Reducing $\omega$-model reflection to iterated syntactic reflection},
journal = {Journal of Mathematical Logic},
volume = {23},
number = {02},
pages = {2250001},
year = {2023},
doi = {10.1142/S0219061322500015}
}

@article{AguileraPakhomov2023Pi13spectrum,
author = {Aguilera, J. P.  and Pakhomov, F. },
title = {The spectrum of $\Pi^1_3$-soundness},
journal = {Philosophical Transactions of the Royal Society A: Mathematical, Physical and Engineering Sciences},
volume = {381},
number = {2248},
pages = {20220013},
year = {2023},
doi = {10.1098/rsta.2022.0013}
}

@article {Spector1955RecursiveWO,
    AUTHOR = {Spector, Clifford},
     TITLE = {Recursive well-orderings},
   JOURNAL = {J. Symbolic Logic},
  FJOURNAL = {The Journal of Symbolic Logic},
    VOLUME = {20},
      YEAR = {1955},
     PAGES = {151--163},
      ISSN = {0022-4812,1943-5886},
   MRCLASS = {02.0X},
  MRNUMBER = {74347},
MRREVIEWER = {G.\ Kreisel},
       DOI = {10.2307/2266902},
       URL = {https://doi.org/10.2307/2266902},
}

@unpublished{AguileraPakhomov2024Nonlinearity,
	title={Non-linearities in the analytical hierarchy},
    author = {Aguilera, J. P.  and Pakhomov, F. },
	note={Preprint},
    URL = {https://www.dropbox.com/scl/fi/2i2e7x9yfyl9f0jmloowg/NonLinearities.pdf},
}

@unpublished{Jeon??HigherProofTheoryII,
	title={The behavior of Higher proof theory {II}},
    author = {Jeon, Hanul},
	note={In preperation}
}

@misc{PakhomovWalsh2021ReflectionInfDeriv,
	author = {Pakhomov, Fedor and Walsh, James},
	title={Reflection ranks via infinitary derivations},
	year = {2021},
	eprint = {2107.03521},
	eprinttype = {arXiv}
}

@article {Girard1981Dilators,
    AUTHOR = {Girard, Jean-Yves},
     TITLE = {{$\Pi \sp{1}\sb{2}$}-logic. {I}. {D}ilators},
   JOURNAL = {Ann. Math. Logic},
  FJOURNAL = {Annals of Mathematical Logic},
    VOLUME = {21},
      YEAR = {1981},
    NUMBER = {2-3},
     PAGES = {75--219}
}

@misc{Arai2023PiN,
	author = {Arai, Toshiyasu},
	title = {An ordinal analysis of $\Pi_N$-Collection},
	year = {2023},
	eprint = {2311.12459},
	eprinttype = {arXiv},
}

@incollection {Simpson2010GodelHierarchy,
    AUTHOR = {Simpson, Stephen G.},
     TITLE = {The {G}\"{o}del hierarchy and reverse mathematics},
 BOOKTITLE = {Kurt {G}\"{o}del: essays for his centennial},
    SERIES = {Lect. Notes Log.},
    VOLUME = {33},
     PAGES = {109--127},
 PUBLISHER = {Assoc. Symbol. Logic, La Jolla, CA},
      YEAR = {2010}
}

@misc{Hamkins2022Nonlinearity,
      title={Nonlinearity and illfoundedness in the hierarchy of large cardinal consistency strength}, 
      author={Joel David Hamkins},
      year={2022},
      eprint={2208.07445},
      archivePrefix={arXiv},
      primaryClass={math.LO}
}

@misc{Towsner2024ProofsModifyProofs,
      title={Proofs that Modify Proofs}, 
      author={Henry Towsner},
      year={2024},
      eprint={2403.17922},
      archivePrefix={arXiv},
      primaryClass={math.LO}
}

@incollection {Steel2014GodelProgram,
    AUTHOR = {Steel, John R.},
     TITLE = {G\"{o}del's program},
 BOOKTITLE = {Interpreting {G}\"{o}del},
     PAGES = {153--179},
 PUBLISHER = {Cambridge Univ. Press, Cambridge},
      YEAR = {2014}
}

@InCollection{Koellner2011IndependenceLC,
	author       =	{Koellner, Peter},
	title        =	{{Independence and Large Cardinals}},
	booktitle    =	{The {Stanford} Encyclopedia of Philosophy},
	editor       =	{Edward N. Zalta},
	howpublished =	{\url{https://plato.stanford.edu/archives/sum2011/entries/independence-large-cardinals/}},
	year         =	{2011},
	edition      =	{{S}ummer 2011},
	publisher    =	{Metaphysics Research Lab, Stanford University}
}

@article {Freund2021Pi14,
    AUTHOR = {Freund, Anton},
     TITLE = {Well ordering principles and {$\Pi ^1_4$}-statements: a pilot
              study},
   JOURNAL = {J. Symb. Log.},
  FJOURNAL = {The Journal of Symbolic Logic},
    VOLUME = {86},
      YEAR = {2021},
    NUMBER = {2},
     PAGES = {709--745}
}

@phdthesis{BussThesis,
    author = {Buss, Samuel R.},
    title = {Bounded Arithmetic},
    school = {Princeton University},
    year = {1985}
}

@article {Takeuti1967ConsistencySubsystemAnalysis,
    AUTHOR = {Takeuti, Gaisi},
     TITLE = {Consistency proofs of subsystems of classical analysis},
   JOURNAL = {Ann. of Math. (2)},
  FJOURNAL = {Annals of Mathematics. Second Series},
    VOLUME = {86},
      YEAR = {1967},
     PAGES = {299--348}
}

@article {Veldman2022ProjectiveSets,
    AUTHOR = {Veldman, Wim},
     TITLE = {Projective sets, intuitionistically},
   JOURNAL = {J. Log. Anal.},
  FJOURNAL = {Journal of Logic and Analysis},
    VOLUME = {14},
      YEAR = {2022},
     PAGES = {Paper No. 5, 85}
}

@article {Rathjen2014RelativizedOrdinalKP,
    AUTHOR = {Rathjen, Michael},
     TITLE = {Relativized ordinal analysis: the case of power
              {K}ripke-{P}latek set theory},
   JOURNAL = {Ann. Pure Appl. Logic},
  FJOURNAL = {Annals of Pure and Applied Logic},
    VOLUME = {165},
      YEAR = {2014},
    NUMBER = {1},
     PAGES = {316--339}
}

@article{Jeon2023WA,
	title={On Separating Wholeness Axioms}, 
	author={Hanul Jeon},
	pubstate= {toappearin},
	journal={Journal of Symbolic Logic}
}

@incollection {Rathjen2015Goodstein,
    AUTHOR = {Rathjen, Michael},
     TITLE = {Goodstein's theorem revisited},
 BOOKTITLE = {Gentzen's centenary},
     PAGES = {229--242},
 PUBLISHER = {Springer, Cham},
      YEAR = {2015}
}

@phdthesis{JeonPhD,
	title={Proof theory for {H}igher {P}ointclasses},
	author={Jeon, Hanul},
	year={2026},
	school={Cornell University},
	type={{PhD} dissertation},
}

@inproceedings {Ressayre1982BoundingGRF,
    AUTHOR = {Ressayre, J. P.},
     TITLE = {Bounding generalized recursive functions of ordinals by
              effective functors; a complement to the {G}irard theorem},
 BOOKTITLE = {Proceedings of the {H}erbrand symposium ({M}arseilles, 1981)},
    SERIES = {Stud. Logic Found. Math.},
    VOLUME = {107},
     PAGES = {251--279},
 PUBLISHER = {North-Holland, Amsterdam},
      YEAR = {1982}
}

@article{TowsnerWalsh2024Classification,
	author = {Towsner, Henry and Walsh, James},
	title = {A classification of incompleteness statements},
	pubstate= {toappearin},
    journal={Canadian Mathematical Bulletin}
}

\end{document}